\documentclass[reqno]{amsart}
\usepackage{enumerate}
\usepackage{mhequ}
\usepackage[margin=1.2in]{geometry}
\usepackage{amssymb,url,color, booktabs,nccmath}
\usepackage{mathrsfs}
\usepackage{enumitem}
\usepackage{graphicx}
\usepackage{enumerate}

\usepackage{color}
\usepackage[colorlinks=true]{hyperref}
\hypersetup{
	linkcolor=blue,          
	citecolor=red,        
	filecolor=blue,      
	urlcolor=cyan
}

\usepackage[colorinlistoftodos,prependcaption,color=yellow,textsize=tiny]{todonotes}

\definecolor{darkergreen}{rgb}{0.0, 0.5, 0.0}

\numberwithin{equation}{section}
\def\theequation{\arabic{section}.\arabic{equation}}
\newcommand{\be}{\begin{eqnarray}}
	\newcommand{\ee}{\end{eqnarray}}
\newcommand{\ce}{\begin{eqnarray*}}
	\newcommand{\de}{\end{eqnarray*}}
\newtheorem{theorem}{Theorem}[section]
\newtheorem{lemma}[theorem]{Lemma}

\newtheorem{remark}[theorem]{Remark}
\newtheorem{definition}[theorem]{Definition}
\newtheorem{proposition}[theorem]{Proposition}
\newtheorem{Example}[theorem]{Example}
\newtheorem{corollary}[theorem]{Corollary}

\newtheorem{customassumption}{Assumption}

\newenvironment{nouppercase}{%
	\renewcommand{\uppercasenonmath}[1]{}}{}
\newtheorem{customthm}{Theorem}

\def\geq{\geqslant}
\def\leq{\leqslant}

\allowdisplaybreaks

\tikzset{
	dot/.style={circle,fill=black,inner sep=0pt, outer sep=0.7pt, minimum size=1mm},
	Phi/.style={white!40!red,thick,snake=coil,segment amplitude=0.6pt, segment length=2pt},
	Z/.style={black!40!green,thick,snake=coil,segment amplitude=0.6pt, segment length=2pt},
	C/.style={thick,black!20!blue},
	Cr/.style={thick,black!20!red},
	Cg/.style={thick,black!20!green},
}

\begin{document}
	
	\title[Stochastic SQG equation]{\LARGE Probabilistic Approaches to The Energy Equality in Forced Surface Quasi-Geostrophic Equations}

	\author[Lin Wang]{\large Lin Wang}
	\address[L. Wang]{School of Mathematical Sciences, University of Chinese Academy of Sciences, Beijing 100049, China; Academy of Mathematics and Systems Science, Chinese Academy of Sciences, Beijing 100190, China.}
	\email{wanglin2021@amss.ac.cn}
	
\author[Zhengyan Wu]{\large Zhengyan Wu}
\address[Z. Wu]{Department of Mathematics, Technische Universit\"at M\"unchen, Boltzmannstr. 3, 85748 Garching, Germany}
\email{wuzh@cit.tum.de}

\begin{abstract}
We explore probabilistic approaches to the deterministic energy equality for the forced Surface Quasi-Geostrophic (SQG) equation on a torus. First, we prove the zero-noise dynamical large deviations for a corresponding stochastic SQG equation, where the lower bound matches the upper bound on a certain closure of the weak-strong uniqueness class for the deterministic forced SQG equation. Furthermore, we show that the energy equality for the deterministic SQG equation holds on arbitrary time-reversible subsets of the domain where we match the upper bound and the lower bound. Conversely, the violation of the deterministic energy equality breaks the lower bound of large deviations. These results extend the existing techniques in Gess, Heydecker, and the second author \cite{GHW24_LLNS} to generalized Sobolev spaces with negative indices. Finally, we provide an analysis of the restricted quasi-potential and prove a conditional equivalence compared to the rate function of large deviations for the Gaussian distribution. This suggests a potential connection between non-Gaussian large deviations in equilibrium for the stochastic SQG equation and the open problem regarding the uniqueness of the deterministic SQG equation.	\end{abstract}

	
	\date{\today}
	
	\begin{nouppercase}
		\maketitle
	\end{nouppercase}
	
	\setcounter{tocdepth}{1}

	\section{Introduction}\label{sec-intro}

This paper is dedicated to exploring probabilistic approaches to the deterministic energy equality for a forced dissipative Surface Quasi-Geostrophic (SQG) equation on the torus $\mathbb{T}^2=\mathbb{R}^2/\mathbb{Z}^2$: 
\begin{equation}\label{PDE-SQG}
	\begin{aligned}
		\partial_t\theta=&-\Lambda^{2\alpha}\theta-u_\theta\cdot\nabla\theta+\Lambda^{2\beta}g,\\
		u_\theta=&R^{\perp}\theta=(-R_2\theta,R_1\theta),
	\end{aligned}
\end{equation}
where $\Lambda=(-\Delta)^{1/2}, R_j$ $(j=1,2)$ is the $j$-th Riesz transform and $g \in L^2([0,T];L^2(\mathbb{T}^2))$. The parameters $\alpha$ and $\beta$ will be discussed within the following  ranges:  
\begin{equation}\label{pama-range}
	\alpha\in(0,1/2),\, \beta=\alpha/2, \ \ \text{and}\ \  \alpha\in [1/2,1),\, \beta=\alpha/2 \text{ or } \alpha/2+1/4.
\end{equation}

A key element of these probabilistic approaches is the interpretation of \eqref{PDE-SQG} as the so-called skeleton equation of the following stochastic SQG equation:  
\begin{align}\label{SPDE-SQG}
	\partial_t\theta_{\varepsilon}=&-\Lambda^{2\alpha}\theta_{\varepsilon}-u_{\theta_\varepsilon}\cdot\nabla\theta_{\varepsilon}+\varepsilon^{1/2}\Lambda^{2\beta}\xi_{\delta(\varepsilon)},\\
	u_{\theta_\varepsilon}=&R^{\perp}\theta_{\varepsilon}=(-R_2\theta_{\varepsilon},R_1\theta_{\varepsilon}),\notag
\end{align}
within the framework of large deviations theory. Here, $\xi_{\delta(\varepsilon)}$ is a spatial regularization of space-time white noise $\xi$. We will explore the connections between dynamical large deviations of \eqref{SPDE-SQG}  as $ (\varepsilon,\delta(\varepsilon))\rightarrow(0,0) $ and the validity of the deterministic energy equality in \eqref{PDE-SQG}. Additionally, we will highlight the potential relationships between large deviations in equilibrium for \eqref{SPDE-SQG} and the uniqueness problem for \eqref{PDE-SQG}. Furthermore, when $ \varepsilon=0 $  in \eqref{SPDE-SQG}, a scaling argument suggests that the fractional dissipation index $\alpha=1/2$ belongs to the critical regime,
$\alpha>1/2$ falls within the subcritical regime, and  $\alpha\in (0,1/2) $ stands for the supercritical regime \cite{Res95}.	Hence the ranges in \eqref{pama-range} cover subcritical, critical, and supercritical cases.

Regarding the deterministic PDE \eqref{PDE-SQG}, the initial data is considered to be an $L^2(\mathbb{T}^2)$-function with zero mean, and the Leray-Hopf solution theory (see \cite[Theorem 3.1]{Res95}, analogous to the Navier-Stokes equations) is employed. One of the key features of the Leray solutions is the so-called energy inequality. However, it is unknown whether the equality case of the energy inequality
\begin{equation}\label{energy-equality}
	\frac{1}{2}\|\theta(T)\|_{L^2(\mathbb{T}^2)}^2+\int_0^T\|\Lambda^{\alpha}\theta\|_{L^2(\mathbb{T}^2)}^2\,\mathrm{d}s= \frac{1}{2}\|\theta(0)\|_{L^2(\mathbb{T}^2)}^2+\int_0^T \langle\Lambda^{2\beta}\theta,g\rangle \,\mathrm{d}s
\end{equation}
holds for a Leray solution of \eqref{PDE-SQG}. In the system of the SQG equation, the field $ \theta $ represents the temperature or surface buoyancy for a rapidly rotating stratified fluid with uniform potential vorticity, and $u_{\theta}=R^{\perp}\theta$  denotes the transport velocity field \cite{HPGS95}. From a mathematical perspective, the energy equality (1.4) is expected to hold formally because $u_\theta$ is divergence-free. For sufficiently regular $\theta$, taking the $L^2$-inner product of both sides of equation \eqref{PDE-SQG} with $\theta$ then yields \eqref{energy-equality}. Moreover, since the Riesz transform is an isometry on $L^2$ \cite[Chapter VI, Section 2]{SW71}, the $L^2$-norm of $\theta$ equals that of $u_\theta$ and thus corresponds to the kinetic energy
\begin{equation*}	\mathcal{K}(t)=\frac{1}{2}\int_{\mathbb{T}^2} |u_{\theta}(t)|^2\,\mathrm{d}x = \frac{1}{2}\int_{\mathbb{T}^2} |\theta(t)|^2\,\mathrm{d}x,\, t\in[0,T].
\end{equation*}Analogous to the celebrated result of Lions-Ladyzhenskaya for three-dimensional Navier-Stokes equations, it is possible to verify the energy equality for the SQG equation under some regularity conditions, see for example \cite{Dai17_energy}. On the other hand, the uniqueness problem of \eqref{PDE-SQG} remains open as well, with various regularity conditions specified to show uniqueness. We refer readers to \cite[Section 3.4]{Res95}, \cite{DC06_weak_strong}, and the references therein for more details. To the best of our knowledge, in the study of SQG equation, there are no direct relationships between the regularity conditions for the energy equality and those for uniqueness. This is one of the motivations for this work, and we point out that a direct relationship will be provided in this context based on the perspective of probability. More precisely, this is linked to the study of dynamical large deviations of \eqref{SPDE-SQG}. Furthermore, this unveils a sufficient condition for proving the energy equality \eqref{energy-equality}, although it is unclear how to verify such a condition.

The stochastic PDE \eqref{SPDE-SQG} (proposed by Totz \cite{Tot20} and Hofmanov\'a et al. \cite{HLZZ24}) is an SQG equation driven by additive noise with small intensity $\varepsilon$ and small correlation $\delta(\varepsilon)$. The consideration of such noise is inspired by the fluctuation-dissipation relation at a formal level. Generally speaking, the SQG equation can be regarded as a toy model for understanding the regularity of the Navier-Stokes equations. We point out that an important fluctuating Navier-Stokes model known as the Landau-Lifshitz-Navier-Stokes (or Navier-Stokes-Fourier system for the full equations) is governed by the fluctuation-dissipation principle, a fundamental concept in the theory of fluctuating hydrodynamics (see \cite{O05}). The stochastic SQG  equation \eqref{SPDE-SQG} could be considered as a stochastic toy model for the Landau-Lifshitz-Navier-Stokes equations. In particular, the choice of $\beta=\frac{\alpha}{2}$, corresponding to the fractional dissipation $-\Lambda^{2\alpha}$, exhibits a structural resemblance to the fluctuation-dissipation principle. Regarding \eqref{SPDE-SQG}, a dynamical large deviation principle is proved, in a joint scaling regime $(\varepsilon,\delta(\varepsilon))\rightarrow(0,0)$, with initial data allowing for fluctuations as well. Moreover, we provide an analysis on the identification of the quasi-potential 
\begin{equation}\label{Quasi-potential}
	\mathcal{U}(\phi)=\frac{1}{2}\inf\bigg\{\int^0_{-\infty}\big\|\partial_t\tilde\theta+\Lambda^{2\alpha}\tilde\theta+R^{\perp}\tilde\theta\cdot\nabla \tilde\theta\big\|_{H^{-2\beta}(\mathbb{T}^2)}^2\,\mathrm{d}s:\ -\tilde\theta(-\cdot)\in\mathcal{A}_0,\ \tilde\theta(0)=\phi\bigg\},	
\end{equation}
where $ \phi \in H^{\alpha-2\beta}(\mathbb{T}^2)$ and $\mathcal{A}_0$ is a restricted domain related to the dynamical large deviations result. We will specify $\mathcal{A}_0$ later on.

Here, we emphasize that in the study of large deviations for \eqref{SPDE-SQG}, the initial data could also be considered in $\dot H^{-1/2}(\mathbb{T}^2)$. For the inviscid SQG equation, the functional
\begin{equation*}
	\mathcal{H}(t):=\frac{1}{2}\|\theta(t)\|_{\dot H^{-1 / 2}\left(\mathbb{T}^2\right)}^2=\frac{1}{2}\int_{\mathbb{T}^2}\Lambda^{-1}\theta(t) \cdot\theta(t)\,\mathrm{d}x,\ \ t\in[0,T]
\end{equation*}
can be taken as the Hamiltonian of the system. This is derived from the Euler-Poincar\'e variational principle. We refer readers to Arnold's work \cite{Arn66}, which is an infinite-dimensional generalization of \cite{Poi01}. (See \cite[Section 2.2]{Res95}, \cite[Section 1.4, Appendix A.1]{BSV19_Nonuniqueness_SQG} for more details.)  In such a lower regularity regime, a commutator estimate associated with the  Riesz transform is employed to make the nonlinear term well-defined. This distinguishes the technique used in proving large deviations for SQG equations from that for three-dimensional Landau-Lifshitz-Navier-Stokes equations established by Gess, Heydecker, and the second author \cite{GHW24_LLNS}. Further comments regarding the technique for such an $\dot{H}^{-1/2}(\mathbb{T}^2)$ framework will be provided later on.

A theoretical basis for considering the fluctuating hydrodynamics equation is that the Gibbs measure preserves the dynamics invariant informally \cite{Spo91}.  As long as we replace the correlated noise $\xi_{\delta(\varepsilon)}$ with a space-time white noise $\xi$, and choose $\beta=\frac{\alpha}{2}$, we are able to see that the $L^2(\mathbb{T}^2)$-cylindrical Gaussian measure $\mathcal{G}(0,\varepsilon I/2)$ is informally invariant for the dynamics, where $I$ is the identity operator on $L^2(\mathbb{T}^2)$, and the dynamics are time-reversible with respect to this measure. A rigorous mathematical approach can be found within the energy solution framework \cite{Tot20}, where the existence of stationary energy solutions with Gaussian distribution is demonstrated. This raises an important question: when regularizing space-time white noise and studying fluctuations in a more regular space under the scaling regimes $(\varepsilon,\delta(\varepsilon))$, can the Gaussian equilibrium feature be preserved in terms of the rate function for large deviations?

We highlight the work \cite{BC17} and \cite{CP22}, where the authors established a large deviation principle for the invariant measure of the two-dimensional Navier-Stokes equation and showed that the rate function is governed by its quasi-potential, which coincides with the rate function for Gaussian large deviations. Inspired by their work, one can investigate large deviations in equilibrium of regularized stochastic PDEs \eqref{SPDE-SQG}. Whether an asymptotic Gaussian or non-Gaussian equilibrium feature arises depends on whether the corresponding quasi-potential coincides with the Gaussian rate function. This motivates us to investigate the identification of the quasi-potential \eqref{Quasi-potential} for the SQG equation.

Roughly speaking, the main results in this paper provide the relationships described by the following picture,
\\
\begin{figure*}[h]
	\centering
	\includegraphics[width=1.0\textwidth]{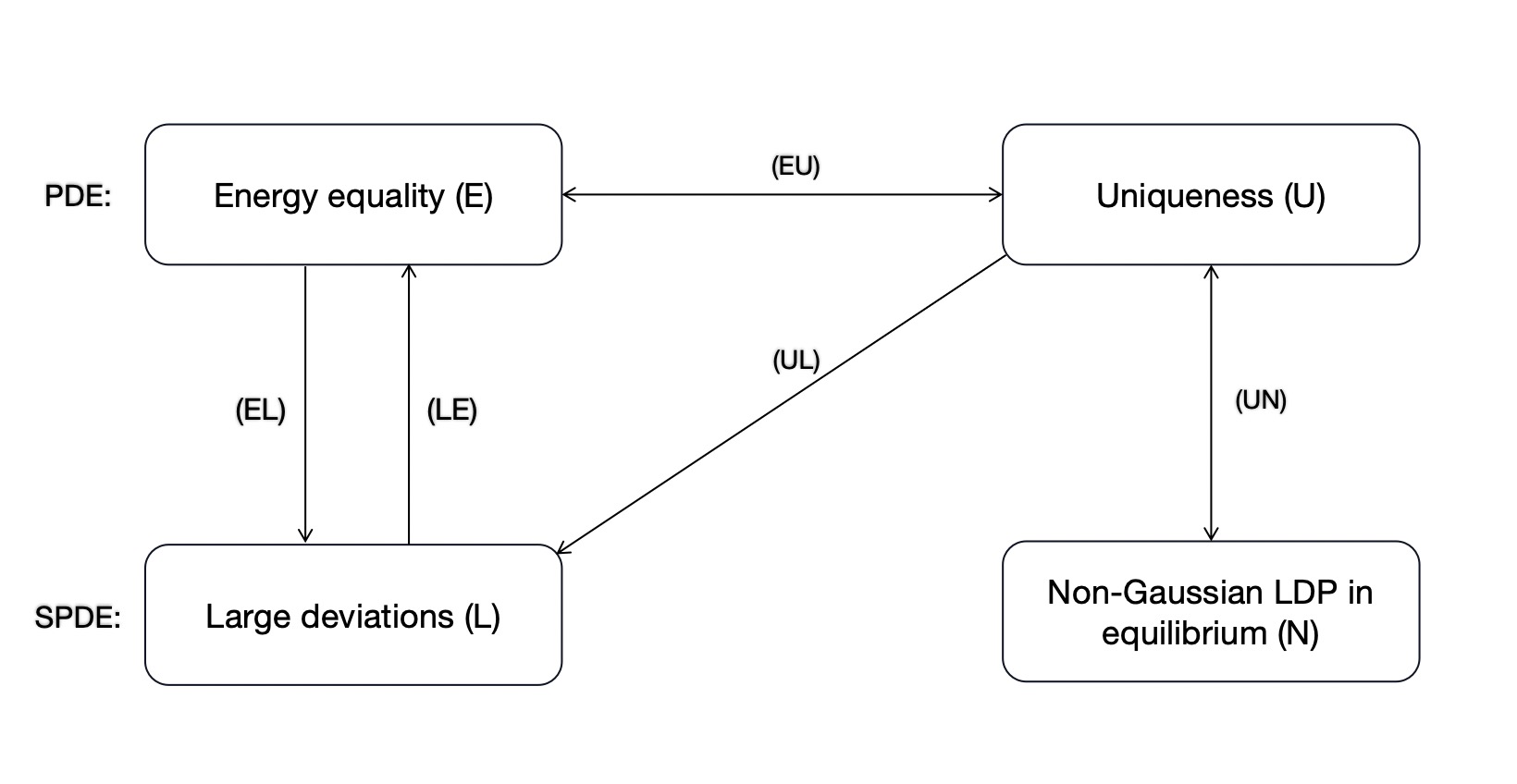}
\end{figure*}

with 

(UL): (Theorem \ref{Main-Theorem-1}) Weak-strong uniqueness implies a restricted dynamical large deviation principle;

(LE): (Theorem \ref{Main-Theorem-2}, probabilistic approach) The energy equality holds on any time-reversible domain where the lower bound of large deviations matches the upper bound;

(EL): (Corollary \ref{Main-Theorem-corollary}) Violation of the kinetic energy equality breaks the lower bound of large deviations;

(EU): (Theorem \ref{Main-Theorem-2}, Theorem \ref{theorem-generalized-energy-equality}, analytic approach)  The energy equality holds on a certain time-reversible closure of the weak-strong uniqueness regularity class; 

(UN): (Theorem \ref{Main-Theorem-3}) The conditional equivalence of the quasi-potential and the Gaussian rate function is proved.

We explain why (UN) provides a potential connection between non-Gaussian large deviations in equilibrium and the open problems of uniqueness for the deterministic SQG equation. In general, the proof of a large deviation principle for invariant measures of stochastic PDEs relies on uniform dynamical large deviations, and the rate function is governed by the quasi-potential, which is defined as the infimum of the dynamical cost on the time interval $(-\infty,0]$ along paths in the space where the dynamical large deviations are proved. However, in this paper, the dynamical large deviation lower bound is shown to match the upper bound only when we restrict fluctuations to a closure of the so-called weak-strong uniqueness class $\mathcal{C}_0$ (see Definition \ref{def-weak-strong-uniqueness-intro} later on). This suggests that the quasi-potential should be defined in a restricted version, see \eqref{Quasi-potential},  where the infimum is taken along paths corresponding to a subset $\mathcal{A}_0\subset\mathcal{C}_0$. To show that the quasi-potential \eqref{Quasi-potential} equals the Gaussian rate function $I_{Gauss}(\phi)=\|\phi\|_{H^{\alpha-2\beta}(\mathbb{T}^2)}^2$ for all $ \phi\in H^{\alpha-2\beta}(\mathbb{T}^2)$, a similar argument as in \cite{BCF15} can be carried out. However, this relies on the existence of solutions for the deterministic SQG equation in $\mathcal{A}_0$. This links to the open problem of the uniqueness for the SQG equation.     

\subsection{Main results}
Let $\alpha$, $\beta$ satisfy \eqref{pama-range}, and $T>0$ be a fixed time horizon in the whole context. We first state the result of dynamical large deviations for \eqref{SPDE-SQG}. Let 

\begin{equation}
	\mathbb{X}_{\alpha,\beta}:= L^2([0,T];H^{\alpha-2\beta}(\mathbb{T}^2)) \cap	L^2_w([0,T];H^{2\alpha-2\beta}(\mathbb{T}^2)) \cap C([0,T];H^{\alpha-2\beta}_w(\mathbb{T}^2)),
\end{equation}
where ``$w$" denotes the weak topology. Let $s>\alpha+1$, we regularize the noise by 
\begin{equation*}
	\xi_{\delta(\varepsilon)}=\sqrt{Q_{\delta(\varepsilon)}}\xi=(I+\delta(\varepsilon)\Lambda^{2s})^{-1/2}\xi.
\end{equation*}
A detailed calculation concerning the blow-up speed of $\xi_{\delta(\varepsilon)}$ will be given in Section \ref{section-Preliminary}. Based on this, we restrict that the scaling regime $(\varepsilon,\delta(\varepsilon))$ satisfies 
\begin{equation}\label{scaling-delta}
	\lim_{\varepsilon\rightarrow0}\varepsilon\delta(\varepsilon)^{-\frac{\alpha+1}{s}}=0
\end{equation}in the study of large deviations.
For every $\theta\in\mathbb{X}_{\alpha,\beta}$, let 
\begin{align*}
	\mathcal{I}_{dyna}(\theta)=\frac{1}{2}\big\|\partial_t\theta+\Lambda^{2\alpha}\theta+R^{\perp}\theta\cdot\nabla\theta\big\| ^2_{L^2\left( [0,T];H^{-2\beta}(\mathbb{T}^2)\right) },\ \ \ \mathcal{I}_{0}(\theta(0))=
	\|\theta(0)\|_{H^{\alpha-2\beta}(\mathbb{T}^2)}^2
\end{align*}
be large deviation costs for the dynamics and the initial data, respectively. Moreover, we set 
\begin{equation}\label{rate-function-intro}
	\mathcal{I}(\theta):=\mathcal{I}_{dyna}(\theta)+\mathcal{I}_0(\theta(0))	
\end{equation}
to be the whole rate function. The following \textit{weak-strong uniqueness} class and its closure are introduced as the spaces where we restrict the lower bound of large deviations. 

\begin{definition}\label{def-weak-strong-uniqueness-intro} 
	We say that
	$\mathcal{C}_0 \subset \mathbb{X}_{\alpha,\beta}$ is a weak-strong uniqueness class of \eqref{PDE-SQG} if for every control $g \in L^2([0,T];L^2(\mathbb{T}^2))$, the following holds:   for arbitrary two weak solutions (see Definition \ref{def-weak-solution-skeleton-equation} later on)  $\theta_1, \theta_2 $  of the skeleton equation \eqref{PDE-SQG} with the same initial data $\theta_1(0)=\theta_2(0)$ and the same control $ g $, we have $\theta_1=\theta_2$ in $ \mathbb{X}_{\alpha,\beta}$ as long as $\theta_1\in\mathcal{C}_0$ and  $\theta_2$ satisfies the $ H^{\alpha-2\beta}$-energy inequality: for every $ t\in[0,T], $
	$$
	\frac{1}{2}\|\theta_2(t)\|_{H^{\alpha-2\beta}(\mathbb{T}^2)}^2+\int_0^t\|\theta_2(s)\|_{H^{2\alpha-2\beta}(\mathbb{T}^2)}^2 \,\mathrm{d}s \leq \frac{1}{2}\left\|\theta_2(0)\right\|_{H^{\alpha-2\beta}(\mathbb{T}^2)}^2+\int_0^t\langle\Lambda^{2\alpha-2\beta} \theta_2, g\rangle \,\mathrm{d}s.
	$$
	Furthermore, the $\mathcal{I}$-closure of $\mathcal{C}_0$ is defined by
	\begin{equation*}
			\overline{\mathcal{C}_0}^{\mathcal{I}}:=\Big\{\theta \in \mathbb{X}_{\alpha,\beta}: \text{there exists a sequence }\{ \theta^{(n)} \} \subset \mathcal{C}_0, \, \text{such that } \theta^{(n)} \rightarrow \theta \text{ in } \mathbb{X}_{\alpha,\beta}, \, \mathcal{I}\big(\theta^{(n)}\big) \rightarrow \mathcal{I}(\theta)\Big\}.
	\end{equation*}
	
\end{definition}

In particular, examples of regularity classes  $\mathcal{C}_0$ and $\overline{\mathcal{C}_0}^{\mathcal{I}}$ are provided in Appendix  \ref{section-example}. Furthermore, we provide the following assumptions of the initial data for \eqref{SPDE-SQG}. 
\begin{customassumption}[Gaussian initial data] \label{Gaussian-initial-data}
	Assume that the law of the initial data $ \theta_{\varepsilon, \delta(\varepsilon)}(0) $ is $\mathcal{G}\left(0, \varepsilon Q_{\delta(\varepsilon)} / 2\right)$, the Gaussian measure on $ \dot{H}^{\alpha-2\beta}(\mathbb{T}^2)$ with mean zero and covariance $\varepsilon Q_{\delta(\varepsilon)}/2.$
\end{customassumption}

Now we provide the definition of solutions. The following definition can be regarded as a generalization of the Leray solution. 

\begin{definition}\label{def-stochastic-leray}   We say that $(\theta_{\varepsilon,\delta(\varepsilon)},W)$,  $\left(\Omega, \mathcal{F},\{\mathcal{F}_t\}_{t \in[0, T]}, \mathbb{P}\right)$ is a stochastic generalized Leray solution of \eqref{SPDE-SQG} with initial data $\theta_0$ if, $ \theta_0 $ is an $ H^{\alpha-2\beta}(\mathbb{T}^2) $-valued random element, $ W $  is a cylindrical Wiener process on $ H^{\alpha-2\beta}(\mathbb{T}^2)$ independent of $\theta_0$, and
	$\theta_{\varepsilon,\delta(\varepsilon)}$ is a progressively measurable process satisfying:
	
	(i)   $\mathbb{P}$-almost surely, $\theta_{\varepsilon,\delta(\varepsilon)}\in L^\infty([0,T];H^{\alpha-2\beta}(\mathbb{T}^2)) \cap L^2([0,T];H^{2\alpha-2\beta}(\mathbb{T}^2))\cap C([0,T];H_w^{\alpha-2\beta}(\mathbb{T}^2))$.
	
	(ii) $\mathbb{P}$-almost surely, for every $\varphi \in C^{\infty}\left([0, T];H^\infty(\mathbb{T}^2)\right)$,  for all $t\in[0,T]$, 
	\begin{equation*}
		\begin{aligned}
			&\left\langle\theta_{\varepsilon,\delta(\varepsilon)}(t), \varphi(t)\right\rangle+\int_0^t\left\langle\theta_{\varepsilon,\delta(\varepsilon)}(s), \Lambda^{2\alpha} \varphi(s)\right\rangle \,\mathrm{d}s=\left\langle \theta_0, \varphi(0)\right\rangle+\int_0^t\langle\theta_{\varepsilon,\delta(\varepsilon)}(s),\partial_s\varphi(s) \rangle\,\mathrm{d}s\\&-\frac{1}{2}\int_0^t\left\langle R_2 \theta _{\varepsilon,\delta(\varepsilon)} (s) ,\left[\Lambda, \partial_1 \varphi(s)\right]\Lambda^{-1} \theta_{\varepsilon,\delta(\varepsilon)}(s)\right\rangle \,\mathrm{d} s\\&+\frac{1}{2}\int_0^t\left\langle R_1 \theta _{\varepsilon,\delta(\varepsilon)} (s) ,\left[\Lambda, \partial_2 \varphi(s)\right]\Lambda^{-1} \theta_{\varepsilon,\delta(\varepsilon)}(s)\right\rangle \,\mathrm{d} s+\sqrt{\varepsilon}\int_{0}^t \langle \Lambda^{2\beta}\varphi(s),\sqrt{Q_\delta(\varepsilon)}\,\mathrm{d}W(s)\rangle.
		\end{aligned}
	\end{equation*}
	Here and subsequently, the nonlinear terms $ \left\langle R_j \theta _{\varepsilon,\delta(\varepsilon)} (s) ,\left[\Lambda, \partial_i \varphi(s)\right]\Lambda^{-1} \theta_{\varepsilon,\delta(\varepsilon)}(s)\right\rangle (i,j=1,2)$ are defined by the $ H^{-1/2}-H^{1/2} $ dual, see Section \ref{section-Preliminary} for details.
	
	(iii)  $ \theta_{\varepsilon,\delta(\varepsilon)} $ satisfies the pathwise $ H^{\alpha-2\beta}$-energy inequality: $\mathbb{P}$-almost surely, for all $t\in[0,T]$, 
	\begin{equation}\label{pathwise energy inequality}
		\begin{aligned}
			\frac{1}{2}\left\| \theta_{\varepsilon,\delta(\varepsilon)}(t)\right\|_{H^{\alpha-2\beta}(\mathbb{T}^2)}^2 +\int_0^t\left\| \theta_{\varepsilon,\delta(\varepsilon)}(s)\right\|_{H^{2\alpha-2\beta}(\mathbb{T}^2)}^2 \,\mathrm{d}s&\leq  \frac{1}{2}\left\|\theta_0\right\|_{H^{\alpha-2\beta}(\mathbb{T}^2)}^2  \\
			+\sqrt{\varepsilon} \int_0^t\langle \Lambda^{2\alpha-2\beta} \theta_{\varepsilon,\delta(\varepsilon)}(s),\sqrt{Q_{\delta(\varepsilon)}}\,\mathrm{d}W(s)\rangle&+\frac{\varepsilon}{2}\left\|\Lambda^{\alpha} \circ \sqrt{Q_{\delta(\varepsilon)}}\right\|_{HS}^2 t.
		\end{aligned}
	\end{equation}
\end{definition}
Based on the above preparation, the large deviations result is stated below. 
\begin{theorem}[Proposition \ref{proposition-upperbound}, Proposition \ref{prop-restricted-lower-bound}]\label{Main-Theorem-1}
	
	For every $\varepsilon,\delta(\varepsilon)>0$, let $\theta_{\varepsilon,\delta(\varepsilon)}$ be a stochastic generalized Leray solution of \eqref{SPDE-SQG} in the sense of Definition \ref{def-stochastic-leray} with initial data $ \theta_{\varepsilon, \delta(\varepsilon)}(0)$ satisfying Assumption \ref{Gaussian-initial-data}. Let $\mu_{\varepsilon}=\mu_{\varepsilon,\delta(\varepsilon)}$ be the laws of $\theta_{\varepsilon,\delta(\varepsilon)}$ on $\mathbb{X}_{\alpha,\beta}$. Let $ \mathcal{C}_0\subset\mathbb{X}_{\alpha,\beta}$ be a weak-strong uniqueness class of \eqref{PDE-SQG} and   $\mathcal{C}=\overline{\mathcal{C}_0}^{\mathcal{I}}$. Assume that the scaling regime \eqref{scaling-delta} holds for $(\varepsilon,\delta(\varepsilon))$. Then for any closed set $F\subset\mathbb{X}_{\alpha,\beta}$,
	\begin{equation*}
		\limsup_{\varepsilon\rightarrow0}\varepsilon\log\mu_{\varepsilon}(F)\leq-\inf_{\theta\in F}\mathcal{I}(\theta),  	
	\end{equation*}
	and for any open set $G\subset\mathbb{X}_{\alpha,\beta}$, 	\begin{equation*}
		\liminf_{\varepsilon\rightarrow0}\varepsilon\log\mu_{\varepsilon}(G)\geq-\inf_{\theta\in G\cap \mathcal{C}}\mathcal{I}(\theta). 	
	\end{equation*}
\end{theorem}
We further comment that the above large deviations result holds for a sequence of Galerkin approximations of \eqref{SPDE-SQG} as well, see Proposition \ref{proposition-upperbound} and Proposition \ref{prop-restricted-lower-bound} for details. Based on this large deviations result, we provide a direct relationship between regularity classes of the weak-strong uniqueness and regularity classes where the energy equality holds. Concretely, we show that the kinetic energy equality holds on every time-reversible set of a certain closure of the weak-strong uniqueness class $\mathcal{C}_0$. This introduces the following result. We introduce the time-reversal operator $\mathfrak{T}_T: \mathbb{X}_{\alpha,\alpha/2}\rightarrow\mathbb{X}_{\alpha,\alpha/2}$  defined by $ (\mathfrak{T}_T\theta)(\cdot):=-\theta(T-\cdot) $.
\begin{theorem}[Theorem \ref{theorem-energy-equality}]\label{Main-Theorem-2}
	
	Assume that	$ \beta=\alpha /2$. Let $\mathcal{C}_0\subset \mathbb{X}_{\alpha,\alpha/2}$ be a weak-strong uniqueness class of \eqref{PDE-SQG} such that $\mathcal{C}=\overline{\mathcal{C}_0}^{\mathcal{I}}$ contains non-empty time-reversible subsets. Let $\mathcal{R} \subset \mathcal{C}$ satisfy $\mathcal{R}=\mathfrak{T}_T \mathcal{R}$.   Suppose that $ \theta\in \mathcal{R} $ is a weak solution of \eqref{PDE-SQG} for some $g\in L^2([0,T];L^2(\mathbb{T}^2))$ in the sense of Definition \ref{def-weak-solution-skeleton-equation}, then the following kinetic energy equality holds:
	\begin{equation*}			\frac{1}{2}\|\theta(T)\|_{L^2(\mathbb{T}^2)}^2+\int_0^T\|\theta(s)\|_{H^\alpha(\mathbb{T}^2)}^2\,\mathrm{d}s=\frac{1}{2}\|\theta(0)\|_{L^2(\mathbb{T}^2)}^2+ \int_0^T\langle\Lambda^{\alpha}\theta,g\rangle \,\mathrm{d}s.
	\end{equation*} 
\end{theorem}
In the proof of Theorem \ref{Main-Theorem-2}, we rely on a time-reversibility argument for solutions of Galerkin approximation equation starting from equilibrium with $\delta=0$ (see Lemma \ref{time-revesible-property} below).  Hence we consider random initial data rather than deterministic ones. We further remark that in the proof for Theorem \ref{Main-Theorem-1}, the Assumption \ref{Gaussian-initial-data} for initial data could be extended to more general conditions (see \cite[Assumption 2.1]{GHW24_LLNS}).

We point out that the objects of discussion in the above relationship are purely deterministic. However, the understanding is based on the perspective of probability and the theory of large deviations. Furthermore, the proof of Theorem \ref{Main-Theorem-2} provides a sufficient condition for the validity of the energy equality \eqref{energy-equality} for arbitrary Leray solutions as well.  

{\bf Sufficient condition:} As long as one can prove a full large deviation principle for the Galerkin sequence of equation \eqref{SPDE-SQG} with respect to the canonical rate function, as stated in Theorem \ref{Main-Theorem-1}, the deterministic energy equality will hold for arbitrary Leray solutions of \eqref{PDE-SQG}.

More precisely, the following corollary implies that the violation of the kinetic energy equality breaks the lower bound of large deviations. 
\begin{corollary}\label{Main-Theorem-corollary}
	Let $ \varepsilon>0 $ and $ m(\varepsilon)\in \mathbb{N}_+ $ satisfy the scaling regime \eqref{scaling-m}. Let $\theta_{\varepsilon,m(\varepsilon)}$ be the solution of
	\eqref{SPDE-SQG-m} with initial data $ \theta_{\varepsilon, m(\varepsilon)}(0)\sim \mathcal{G}(0,\varepsilon P_{m(\varepsilon)}/2)$ and let $ \mu_{\varepsilon, m(\varepsilon)}$ be the law of $\theta_{\varepsilon, m(\varepsilon)}$ on $\mathbb{X}_{\alpha,\alpha/2}.$  
	Suppose that  $ \theta $ is a weak solution of \eqref{PDE-SQG} for some $g \in L^2([0, T];L^2(\mathbb{T}^2))$ in the sense of Definition \ref{def-weak-solution-skeleton-equation}.
	\item (i) Assume that
	\begin{equation*}
		\frac{1}{2}\|\theta(T)\|_{L^2(\mathbb{T}^2)}^2+\int_0^T\|\theta(s)\|_{H^\alpha(\mathbb{T}^2)}^2\,\mathrm{d}s>\frac{1}{2}\|\theta(0)\|_{L^2(\mathbb{T}^2)}^2+\int_0^T\langle\Lambda^{\alpha}\theta, g\rangle\,\mathrm{d}s,
	\end{equation*}
	then $\theta$ violates the large deviations lower bound:
	\begin{equation*}
		\inf _{\{ G\text{ is open in }\mathbb{X}_{\alpha,\alpha/2}:  \theta\in G\}} \liminf _{\varepsilon \rightarrow 0} \varepsilon \log \mu_{\varepsilon,m(\varepsilon)}(G)<-\mathcal{I}(\theta).
	\end{equation*}
	\item (ii)  Assume that
	\begin{equation*}
		\frac{1}{2}\|\theta(T)\|_{L^2(\mathbb{T}^2)}^2+\int_0^T\|\theta(s)\|_{H^\alpha(\mathbb{T}^2)}^2\,\mathrm{d}s<\frac{1}{2}\|\theta(0)\|_{L^2(\mathbb{T}^2)}^2+\int_0^T\langle\Lambda^{\alpha}\theta, g\rangle\,\mathrm{d}s,
	\end{equation*}
	then $\mathfrak{T}_T \theta$ violates the large deviations lower bound:
	\begin{equation*}
		\inf_{\{G\text{ is open in }\mathbb{X}_{\alpha,\alpha/2}, \,\mathfrak{T}_T \theta\in G\}} \liminf _{\varepsilon \rightarrow 0} \varepsilon \log \mu_{\varepsilon,m(\varepsilon)}(G)<-\mathcal{I}(\mathfrak{T}_T \theta).
	\end{equation*}
\end{corollary}
Finally, we provide an analysis on the exact representation of the quasi-potential \eqref{Quasi-potential}. Let $ \mathcal{C}_0 $ be a  weak-strong uniqueness class on $ [0,\infty)$ (See Definition \ref{def-weak-strong-uniqueness-infty}). Denote \begin{equation*}
	\begin{aligned}
		\mathcal{A}_0:=\Big\lbrace \theta\in\mathcal{C}_0: 	\partial_t\theta\in L^2([0,\infty) ; H^{\alpha-2\beta}(\mathbb{T}^2)), R^{\perp}\theta\cdot\nabla \theta\in L^2([0,\infty) ;  H^{\alpha-2\beta}(\mathbb{T}^2))&, \\ \lim\limits_{t\rightarrow \infty}\|\theta(t)\|_{H^{\alpha-2\beta}(\mathbb{T}^2)}=0& \Big\rbrace.
	\end{aligned}
\end{equation*}		

\begin{theorem}[Proposition \ref{quasi-Gaussian}]\label{Main-Theorem-3} Let $\mathcal{U}$ be the quasi-potential defined by \eqref{Quasi-potential}. Then
	\item (i)  For every $\phi\in H^{\alpha-2\beta}(\mathbb{T}^2)$, it holds that $\mathcal{U}(\phi)\geq\|\phi\|_{H^{\alpha-2\beta}(\mathbb{T}^2)}^2$.
	\item (ii)
	The equality $\mathcal{U}(\phi)=\|\phi\|_{H^{\alpha-2\beta}(\mathbb{T}^2)}^2$ holds for a given $\phi\in H^{\alpha-2\beta}(\mathbb{T}^2)$ as long as the equation
	\begin{equation*}
		\left\{\begin{array}{l}
			\partial_t\bar\theta=-\Lambda^{2\alpha} \bar\theta-R^{\perp}\bar\theta\cdot\nabla \bar\theta, \\
			\bar\theta(0)=\phi,
		\end{array}\right.
	\end{equation*}
	admits a unique weak solution in $\mathcal{A}_0 $ in the sense of Definition \ref{def-sol-infinity}.
\end{theorem}

\subsection{Key idea and technical comment}

\

{\bf Large deviations.} Concerning the large deviations of \eqref{SPDE-SQG}, the main obstacles arise from two aspects. On the one hand, generalized Leray solutions of \eqref{SPDE-SQG} are probabilistically weak and therefore cannot be represented as measurable maps of Brownian paths. This prevents us from using the well-known weak convergence approach by Budhiraja, Dupuis, and Maroulas \cite{BDM08_LDP} for large deviations of stochastic PDEs. On the other hand, the uniqueness of the skeleton equation \eqref{PDE-SQG} is unknown in such critical and supercritical ($\alpha\in(0,1/2]$) regimes. Large deviations for subcritical stochastic SQG equations driven by multiplicative noise have been studied by Liu, R\"{o}ckner, and Zhu in \cite{LRZ13_LDP}, under the assumption that $ \theta_0\in L^p\cap H^\delta $, where $ \delta >2-2\alpha$ and $ 1/p\in(0,\alpha-1/2)$. In this case, the pathwise uniqueness of the stochastic SQG equations is obtained, and the skeleton equation also has a unique solution, which leads to the weak convergence approach. In general, the well-posedness of the skeleton equation plays a key role in large deviations. We refer readers to \cite{Hey23} for more details, where a counterexample is proposed to illustrate the violation of the lower bound due to the lack of uniqueness. Also see \cite[Section 8]{FG23} for the consistency of the rate function and its lower semi-continuous envelope with the help of the uniqueness and stability of the skeleton equation. These two obstacles have also been encountered in the study of three-dimensional Landau-Lifshitz-Navier-Stokes equations and were resolved by Gess, Heydecker, and the second author \cite{GHW24_LLNS}. To summarize the key idea, the upper and lower bounds are investigated separately, with the lower bound being obtained manually by restricting to the weak-strong uniqueness class. In this context, we will adopt the same idea as in \cite{GHW24_LLNS} on the probabilistic side, but extend it on the analytic side. Specifically, we generalize this argument into an $\dot H^{-1/2}$-framework.

We consider a generalized Leray solution theory for \eqref{SPDE-SQG} with an $\dot H^{-\gamma}$-energy inequality, where $\gamma=0$ or $1/2$. The negative Sobolev's regularity directly produces obstacles in the well-definedness of the nonlinear term $u_{\theta_\varepsilon}\cdot\nabla\theta_{\varepsilon}$, since the pointwise product cannot be defined for distributions. Alternatively, a commutator approach is employed to make the nonlinear term well-defined. Specifically, for every $\theta,\phi\in C^{\infty}(\mathbb{T}^2)$, we denote
\begin{equation*}
	\left[\Lambda, \partial_i \phi\right](\Lambda^{-1}\theta):=\Lambda(\partial_i \phi \Lambda^{-1}\theta)-\partial_i \phi \theta 
\end{equation*}
as the commutator between $\Lambda$ and $\partial_i \phi$ for $i=1,2$. An intuitive computation shows the identity 
\begin{equation}	\label{commutator identity}
	\int_{\mathbb{T}^2} \theta R^{\perp}\theta \cdot \nabla \phi \,\mathrm{d}x=\frac{1}{2}\int_{\mathbb{T}^2} R_2 \theta\left[\Lambda, \partial_1 \phi\right]\left(\Lambda^{-1} \theta\right) \,\mathrm{d} x-\frac{1}{2}\int_{\mathbb{T}^2} R_1 \theta\left[\Lambda, \partial_2 \phi\right]\left(\Lambda^{-1} \theta\right) \,\mathrm{d}x. 
\end{equation}
Moreover, for every $\theta\in \mathbb{X}_{\alpha,\beta}$ and every $\phi\in C^{\infty}(\mathbb{T}^2)$, the right-hand side of \eqref{commutator identity} is well-defined as well. This suggests that the right-hand side of \eqref{commutator identity} can be taken as the  definition of the nonlinear term $\langle	\nabla \cdot\left(\theta R^{\perp}\theta\right),\phi\rangle$,  see Lemma \ref{lemma-calderon} later on for more details.

Secondly, the compactness argument lays a fundamental in the study of both the existence of solutions and large deviations. In general, problems concerning the passage to the limit will be encountered in such arguments. Regarding the passage to the limit of the nonlinear term, the Littlewood-Paley theory is employed. Precisely, $ 	\nabla \cdot\left(\theta R^{\perp}\theta\right) $ can be decomposed into three parts: for every $j \in \mathbb{Z}$,
\begin{equation}\label{decomosition-intro}
	\nabla \cdot\left(\theta R^{\perp}\theta\right)=\nabla \cdot\left(\theta R^{\perp}\left(S_j \theta\right)\right)+\nabla \cdot\left(S_j \theta R^{\perp}\left(H_j \theta\right)\right)+\nabla \cdot\left(H_j \theta R^{\perp}\left(H_j \theta\right)\right),
\end{equation}
where $ H_j \theta $ is the high-frequencies part of $\theta$ and $S_j \theta$ is the low-frequencies part  defined by \eqref{low-and-high} in Section \ref{section-Preliminary}. It has been shown by Marchand in \cite[Lemma 9.3]{Mar08_existence} that the first two terms on the right-hand side of \eqref{decomosition-intro}  are easier to estimate due to the regularizing effect of $ S_j$. However, the estimation of $ \nabla \cdot\left(H_j \theta R^{\perp}\left(H_j \theta\right)\right)$ is more challenging and requires the commutator estimates. In particular, for every $\phi\in C^{\infty}(\mathbb{T}^2)$, $\iota>0$ and $ r\in(-1,1/2], $ we have 
$$
_{\dot{H}^{-3-r-\iota}}\left\langle\nabla \cdot\left(H_j \theta R^{\perp}\left(H_j \theta\right)\right), \phi\right\rangle_{\dot{H}^{3+r+\iota}}\leq C\|\phi\|_{\dot H^{3+r+\iota}(\mathbb{T}^2)}\left\|H_j \theta\right\|_{\dot H^{-r}(\mathbb{T}^2)}^2.
$$	

{\bf The energy equality.}
In this part, we summarize the key idea of the proof for Theorem \ref{Main-Theorem-2}. The main idea arises from \cite[Section 11]{GH23} and \cite[Section 9]{GHW24_LLNS}. In particular, in \cite[Section 9]{GHW24_LLNS}, the authors provide a purely analytic approach concerning the relationship between the energy equality and the weak-strong uniqueness class for three-dimensional forced Navier-Stokes equations. Moreover, they provide a brief explanation of its probabilistic interpretation. In this paper, we adopt the probabilistic approach for proving this relationship when $\beta=\alpha/2$.

The key ingredient is to utilize the fact that the Gaussian distribution $\mathcal{G}(0,\varepsilon P_m/2) $ is an invariant measure for the Galerkin approximation of \eqref{SPDE-SQG}: 
\begin{equation}\label{SPDE-SQG-m}
	\partial_t\theta_{\varepsilon,m}=-\Lambda^{2\alpha}\theta_{\varepsilon,m}-P_m(u_{\theta_{\varepsilon,m}}\cdot\nabla\theta_{\varepsilon,m})+\varepsilon^{1/2}P_m\Lambda^{\alpha}\xi,
\end{equation}
where $P_m$ denotes the $m$-dimensional projection operator (see Section \ref{section-Tightness} for a rigorous definition). Thanks to the time-reversibility (see Lemma \ref{time-revesible-property})
$$
\mathfrak{T}_T \theta_{\varepsilon, m}:=\left(-\theta_{\varepsilon, m}(T-t): 0 \leq t \leq T\right)\overset{\mathrm{d}}{=}\left(\theta_{\varepsilon, m}(t): 0 \leq t \leq T\right), 
$$
both $\theta_{\varepsilon, m(\varepsilon)}$ and its time-reversal process $\mathfrak{T}_T \theta_{\varepsilon, m(\varepsilon)}$ satisfy large deviations as presented in Theorem \ref{Main-Theorem-1}. On the domain where the large deviations lower bound matches the upper bound ($\mathcal{C}$ in Theorem \ref{Main-Theorem-1}), the uniqueness of the large deviations rate function holds. Therefore, we have that $\mathcal{I}(\mathfrak{T}_T\theta)=\mathcal{I}(\theta)$ for all $\theta\in\mathbb{X}_{\alpha,\alpha/2}$ satisfying $\theta,\mathfrak{T}_T\theta\in\mathcal{C}$. For any weak solution $\theta$ of \eqref{PDE-SQG}, direct computation shows that 
\begin{align*}
	\mathcal{I}(\theta)&=\|\theta(0)\|_{L^2(\mathbb{T}^2)}^2,\\ \mathcal{I}(\mathfrak{T}_T\theta)&=\|\theta(T)\|_{L^2(\mathbb{T}^2)}^2+2\int_0^T\|\theta(s)\|_{H^\alpha(\mathbb{T}^2)}^2\,\mathrm{d}s-2\int_0^T\langle\Lambda^{\alpha}\theta,g\rangle \,\mathrm{d}s.
\end{align*}
Then the kinetic energy equality holds as long as $\theta,\mathfrak{T}_T\theta\in\mathcal{C}$. Based on this observation, one can see that the uniqueness of the rate function plays a key role. Since the uniqueness holds on the domain where the lower bound matches the upper bound, then one can provide a sufficient condition concerning the open problem of the validity of the kinetic energy equality for arbitrary Leray solutions. As long as  a full large deviation of \eqref{SPDE-SQG-m} on $\mathbb{X}_{\alpha,\alpha/2}$ with respect to the rate function $\mathcal{I}$ was proved, then this implies that the kinetic energy equality holds on $\mathbb{X}_{\alpha,\alpha/2}$ without restriction.
\subsection{Comments of literature}

\ 

{\bf{Large deviations for stochastic PDEs in singular limits.}}
Large deviations of stochastic scalar conservation laws in the joint limits of vanishing noise and viscosity were studied by Mariani \cite{Mariani10}. In the work of Dirr, Fehrman, and Gess \cite{DFG24}, large deviations for a conservative stochastic PDE corresponding to the simple symmetric exclusion process were established. Fehrman and Gess \cite{FG23} demonstrated that a class of generalized Dean-Kawasaki type equations satisfies large deviations with the same rate function as that of the zero range process. The second author and Zhang \cite{WZ24} extended the analysis to a more general case concerning singular interactions for the Dean-Kawasaki equation. Gess, Heydecker, and the second author \cite{GHW24_LLNS} proved the large deviation principle for the Landau-Lifshitz-Navier-Stokes equations, whose rate function is consistent with the Quastel-Yau lattice gas model \cite{QY98}. Further works concerning large deviations for stochastic PDEs in scaling limits with vanishing noise intensity and correlations can be found in \cite{HW15_LDP, CD19_LDP1, CD19_LDP2, CP22}. 

{\bf{Energy equality for the SQG equation.}} Energy conservation of the inviscid SQG equation is related to Onsager's conjecture \cite{Ons49}. As previously introduced, there are two types of energy that have physical significance for the SQG equation: the Hamiltonian and the kinetic energy. Isett and Vicol proved in \cite{IV15_2d_Onsager} that $\theta \in L^3_t L^3_x$ implies the conservation of the Hamiltonian $\mathcal{H}$. For sufficiently regular solutions, the $L^2$-norm (the kinetic energy) is conserved as well. Several Besov-type regularity conditions are proposed for weak solutions to guarantee the conservation of the kinetic energy, see \cite{Zho05_energy}, \cite{Cha06_energy}, \cite{WYY23_energy}. Dai \cite{Dai17_energy} investigated the kinetic energy equality for viscosity solutions of the super-critical dissipative SQG equation under some regularity conditions. 

{\bf{Weak-strong uniqueness for the SQG equation.}} 
Results of weak-strong uniqueness for the SQG equations with different fractional dissipation index $ \alpha $ and different regularity conditions for initial data $ \theta_0 $ can be found in the following literature. Constantin and Wu \cite{CW99_Weak_strong} proved the uniqueness of  strong solution when $ \alpha\in(1/2,1] $ with $ \theta_0$ in $ \dot H^{-1/2} $ or $ L^2 $ , which is analogous to the Ladyzhenskaya-Prodi-Serrin condition of the Navier-Stokes equation \cite{KL57,Pro59,Ser62}.  In the critical and supercritical cases $ \alpha\in(0,1/2] $, there are several results for $L^2$ initial data. Dong and Chen in \cite{DC06_weak_strong}  proposed Ladyzhenskaya-Prodi-Serrin type condition $ L^r_tL^p_x $ to $ \nabla\theta $ and showed a weak-strong uniqueness property. In \cite{DC12_weak_strong_uniqueness}, the authors extended this regularity condition to Besov space $ L^r_tB^0_{p,\infty} $. For the critical case $ \alpha=1/2 $, Marchand \cite{Mar08_Weak_strong_critical} proved the weak-strong uniqueness in BMO-type space. Liu, Jia, and Dong \cite{LJD12} proposed another BMO-type condition and proved the weak-strong uniqueness for all $ \alpha\in(0,1).$
We will summarize examples of weak-strong uniqueness for subcritical, critical, and supercritical cases in Appendix \ref{section-example}, and illustrate that these regularity classes can be taken as examples of the domain where we restrict the lower bound of large deviations. Furthermore, we point out that, to the best of our knowledge, weak-strong uniqueness results for the supercritical case $\alpha\in(0,1/2)$ with $ \theta_0$ in $ \dot H^{-1/2} $ have not been obtained. In this case, the dissipation term is not sufficient to prevent weak solutions from being distribution-valued. Therefore, it is challenging to propose a reasonable regularity condition for proving the uniqueness. 

{\bf{Relationship to the study of convex integration.}} 
From the previous explanations, it can be inferred that obtaining matching upper and lower bounds for large deviations is related to the uniqueness of solutions for  \eqref{PDE-SQG}.
Recently, the convex integration technique
has led to various non-uniqueness results for fluid equations.
Buckmaster, Shkoller, and Vicol \cite{BSV19_Nonuniqueness_SQG} demonstrated the non-uniqueness of weak solutions for both inviscid and dissipative SQG equations by converting the equation into momentum form. Subsequently, different convex integration schemes have been employed to construct non-unique solutions for inviscid and dissipative SQG equations, as illustrated in \cite{CKL21} and \cite{IM21}. All of these works focus on the SQG equations without external forces, where the solutions satisfy $ \Lambda^{-1}\theta\in C_t^\sigma  C_x^s $ for some $ \sigma $ and $ s $. However, it remains an open question whether  $\theta\in L^2_tH^\alpha_x$, which is the space considered in Theorem \ref{Main-Theorem-2}. For example, in \cite{BSV19_Nonuniqueness_SQG}, the solutions were constructed in $C_t^\sigma  C_x^s $ with $ 1/2<s<4/5$ and $ \sigma<s/(2-s). $ Therefore, these non-uniqueness results do not imply a violation of the full large deviations. 

In \cite{HZZ23} and \cite{HLZZ24}, the authors discussed the non-uniqueness and non-Gaussianity for the stochastic SQG equations driven by irregular spatial noise and space-time white noise, respectively.  We remark that, in \cite{HLZZ24}, infinitely many non-Gaussian ergodic stationary solutions of \eqref{SPDE-SQG} (with $ \delta=0 $) in $ B^{-1/2}_{p,1}, p\geq 2, $ were constructed. We hope to connect the discussion of non-Gaussian stationary solutions with the uniqueness problem of deterministic SQG equations through the non-Gaussian large deviations in Section \ref{section-quasipotential}.

\subsection{Structure of the paper}

This paper is organized as follows.  In Section \ref{section-Preliminary}, we introduce some preliminaries which will be used throughout the paper. Section \ref{section-Tightness} is dedicated to proving the existence and exponential tightness for stochastic generalized Leray solutions.  Section \ref{section-rate} presents a variational characterization of the rate function. In Section \ref{section-upper}, we establish the upper bound for large deviations. Section \ref{section-lower} provides a restricted lower bound for large deviations using the entropy method. In Section \ref{section-energy-eq}, we prove a relationship between the weak-strong uniqueness class and the energy equality for \eqref{PDE-SQG}. Finally, in Section \ref{section-quasipotential}, we discuss an explicit representation of the quasi-potential. Additionally, Appendix \ref{section-example} summarizes several examples of weak-strong uniqueness classes concerning different ranges of parameters $\alpha$ and $\beta$.

\ 
						
\section{Preliminary}\label{section-Preliminary}

This section is devoted to presenting some preliminaries. 

\subsection{Notations and spaces}

Throughout the paper, we use the notation $a \lesssim b$ to denote that there exists a constant $C>0$ such that $a \leqslant Cb$, where $ C $ may change from line to line.

By testing the constant function $1$ for \eqref{PDE-SQG}, one can see that the spatial average of the solutions is conserved. Therefore, we assume that the zero Fourier mode of the initial data is zero. In this way, we study \eqref{PDE-SQG} and \eqref{SPDE-SQG} in spaces with zero mean. 

For any two topological spaces $ X_1 $ and $ X_2 $, $C(X_1; X_2)$ is defined as  the space of continuous functions from $X_1$ to $X_2$.  $ C_b(X_1) $ is defined as the space of bounded continuous functions from $X_1$ to $\mathbb{R}$.  Given a Banach space $E$ with norm $\|\cdot\|_E$, for any $ \sigma\in (0,1) $,  we write $C^\sigma([0, T] ; E)$ to denote the space of $\sigma$-H\"older continuous functions from $[0, T]$ to $E$ endowed with the semi-norm  
\begin{equation*} 
	\|f\|_{C^\sigma([0, T] ; E)}:=\sup _{s, t \in[0, T], s \neq t} \frac{\|f(s)-f(t)\|_E}{|t-s|^\sigma}.
\end{equation*}
Given $p>1, \sigma \in(0,1)$, let $W^{\sigma, p}([0,T];E)$ denote the space of  all $f \in L^p([0,T]; E)$ such that
\begin{equation*}
	\int_0^T \int_0^T \frac{\|f(t)-f(s)\|_E^p}{|t-s|^{1+\sigma p}} \mathrm{\,d}t \,\mathrm{d} s<\infty,
\end{equation*}
endowed with the norm
$	\|f\|_{W^{\sigma, p}([0,T]; E)}^p:=\int_0^T\|f(t)\|_E^p  \mathrm{\,d}t +\int_0^T \int_0^T \frac{\|f(t)-f(s)\|_E^p}{|t-s|^{1+\sigma p}}  \mathrm{\,d}t\, \mathrm{d} s$.
We denote the distribution space on $\mathbb{T}^2$ by $\mathscr{D}^\prime(\mathbb{T}^2)$, which is the dual space of $C^\infty(\mathbb{T}^2)$. Let $f\in\mathscr{D}^\prime(\mathbb{T}^2)$ and let $ g =(g(k))_{k \in \mathbb{Z}^2}$ be at most of polynomial growth. The Fourier transform of $f$ and the inverse Fourier transform of $g$  are defined as

\begin{align*}
	\mathcal{F}_{\mathbb{T}^2} f(k)&:=\hat{f}(k):=\int_{\mathbb{T}^2} e^{2\pi \mathrm{i}x \cdot k} f(x) \mathrm{\,d} x,  \quad k \in \mathbb{Z}^2, \\ \mathcal{F}_{\mathbb{T}^2}^{-1} g(x)&:= \sum_{k \in \mathbb{Z}^2} e^{2\pi \mathrm{i}x \cdot k} g(k),\quad x\in \mathbb{T}^2.
\end{align*}								
Let $H:=\left\{f \in L^2\left(\mathbb{T}^2\right): \int_{\mathbb{T}^2} f \,\mathrm{d}x=0\right\}$ and let $\langle\cdot,\cdot\rangle$ denote the inner product on $ H. $ 	Let $\mathbb{Z}_{+}^2:=\left\{\left(k_1, k_2\right) \in \mathbb{Z}^2|k_2>0\right\} \cup\left\{\left(k_1, 0\right) \in \mathbb{Z}^2| k_1>0\right\}$ and $ \mathbb{Z}_{-}^2:=\left\{\left(k_1, k_2\right) \in\mathbb{Z}^2 |-k \in \mathbb{Z}_{+}^2\right\},
$
then $$\left\{\sqrt{2}\sin(2\pi k\cdot)| k \in\mathbb{Z}_{+}^2\right\} \cup\left\{\sqrt{2}\cos(2\pi k\cdot)|k \in \mathbb{Z}_{-}^2\right\} $$
is an orthonormal eigenbasis of $\Lambda$ on $ H $, which we denote by $\{e_k\}$. For every $s>0$, define
$	\|f\|_{H^s(\mathbb{T}^2)}^2:=\sum_{k\in\mathbb{Z}^2}|2\pi k|^{2s}\big|\hat{f}(k)\big|^2$,
and let $H^s(\mathbb{T}^2)$ denote the space of all $f \in H$ for which $\|f\|_{H^s(\mathbb{T}^2)}$ is finite. Define $H^{-s}(\mathbb{T}^2)$ to be the dual of $H^{s}(\mathbb{T}^2)$.
Finally, let $H^\infty(\mathbb{T}^2):=\bigcap\limits_{s>0}H^s(\mathbb{T}^2)$.

Now we recall the Littlewood-Paley decomposition briefly. Let $\varphi \in C_c^\infty\left(\mathbb{R}^2\right)$ be a non-negative radial function so that $\varphi(\xi)=1$ for $|\xi| \leq 1 / 2$ and $\varphi(\xi)=0$ for $|\xi| \geq 1$. Let  $\psi(\xi)=\varphi(\xi / 2)-\varphi(\xi)$. Then $ \big(\psi(\cdot/2^j)\big)_{j\in \mathbb{Z}} $ is a dyadic partition of unity. For every $j \in \mathbb{Z}$, we define the $j$-th  dyadic block of the Littlewood-Paley decomposition of $ f\in \mathscr{D}^\prime(\mathbb{T}^2)$ by $
\Delta_j f:=\mathcal{F}_{\mathbb{T}^2}^{-1}\left(\psi\left(\cdot/ 2^j\right) \hat{f}\right)$.
The low-frequencies and high-frequencies cutting-off operators are defined by
\begin{equation}\label{low-and-high}
	S_j f:=\mathcal{F}_{\mathbb{T}^2}^{-1}\left(\varphi\left(\cdot/ 2^j\right) \hat{f}\right)= \sum_{k=-\infty}^{j-1} \Delta_k f, \quad
	H_j f:=\left(Id-S_j\right) f=\sum_{k=j}^{\infty}  \Delta_k f .
\end{equation}
For every $ \phi\in H^\infty(\mathbb{T}^2)$, the commutator between $\Lambda$ and $\phi$ is defined by $[\Lambda, \phi]g:=\Lambda(\phi g)-\phi\Lambda g$, for any $ g\in H^\infty(\mathbb{T}^2). $ 
The following lemma enables us to introduce the commutator estimate, which will contribute to the well-definedness of the nonlinear term.
\begin{lemma}\label{lemma-comutator}
	Let $\theta, \phi  \in H^\infty\left(\mathbb{T}^2\right)$, then
	\begin{equation*}
		2 \int_{\mathbb{T}^2}\theta R^{\perp}\theta \cdot \nabla \phi \,\mathrm{d}x=\int_{\mathbb{T}^2}R_2 \theta\left[\Lambda, \partial_1 \phi\right]\left(\Lambda^{-1} \theta\right) \,\mathrm{d} x-\int_{\mathbb{T}^2} R_1 \theta\left[\Lambda, \partial_2 \phi\right]\left(\Lambda^{-1} \theta\right) \,\mathrm{d}x.
	\end{equation*}
\end{lemma}

Lemma \ref {lemma-comutator} can be proved by integration by parts, see for example, \cite[Lemma 2.1]{Mar08_existence}. 

\begin{lemma}[Commutator Estimate]\label{lemma-calderon}
	For every $r \in(-1,1)$, let $\phi \in H^\infty\left(\mathbb{T}^2\right)$ and $g \in H^r\left(\mathbb{T}^2\right)$, then 
	\begin{equation*}
		\|[\Lambda, \phi] g\|_{H^r(\mathbb{T}^2)} \leq C(r,\phi) \|g\|_{H^r(\mathbb{T}^2)}.
	\end{equation*}
\end{lemma}
This commutator estimate is stated for $\mathbb{R}^2$ functions in \cite[Lemma 2.2]{Mar08_existence}. The same proof also works in the case of $\mathbb{T}^2$ functions, see  \cite[Lemma 4.2]{ZZ14}, \cite[Lemma A.5]{BSV19_Nonuniqueness_SQG}.

It follows from the above two lemmas that for every $f$ in  $H^{-1 / 2}\left(\mathbb{T}^2\right)$,  $\nabla \cdot\left(f R^{\perp}(f)\right)$ can be defined as a distribution by
\begin{equation}\label{decomposition}
	\nabla \cdot\left(f R^{\perp}(f)\right)=\nabla \cdot\left(f R^{\perp}\left(S_j f\right)\right)+\nabla \cdot\left(S_j f R^{\perp}\left(H_j f\right)\right)+\nabla \cdot\left(H_j f R^{\perp}\left(H_j f\right)\right),
\end{equation}
for arbitrary $j \in \mathbb{Z}$. And $\nabla \cdot\left(H_j f R^{\perp}\left(H_j f\right)\right)$ is the distribution defined by 
\begin{equation}\label{tempered distribution}
	\begin{aligned}
		\left\langle\nabla \cdot\left(H_j f R^{\perp}\left(H_j f\right)\right), \phi\right\rangle :=&-{}	_{H^{-1/2}}\big\langle \frac{1}{2}R_2\left(H_j f\right),\left[\Lambda, \partial_1 \phi\right]\Lambda^{-1}\left(H_j f\right)\big\rangle_{H^{1/2}} \\&+{}_{H^{-1/2}}\big\langle \frac{1}{2}R_1\left(H_j f\right),\left[\Lambda, \partial_2 \phi\right]\Lambda^{-1}\left(H_j f\right)\big\rangle_{H^{1/2}}
	\end{aligned}
\end{equation} for all $\phi \in H^\infty\left(\mathbb{T}^2\right)$.  More precisely, $\nabla \cdot\left(f R^{\perp}\left(S_j f\right)\right)$ and $\nabla \cdot\left(S_j f R^{\perp}\left(H_j f\right)\right)$ are well-defined since $ S_jf  \in  L^\infty(\mathbb{T}^2)$.  For the high-frequency part $ \nabla \cdot\left(H_j f R^{\perp}\left(H_j f\right)\right) $, it follows from Lemma \ref{lemma-calderon} that,  for any $ r\in (-1,1/2] $, $j\in\mathbb{Z}$, 
\begin{equation}\label{high-frequency}
		\left|\langle\nabla \cdot(H_j f R^{\perp}(H_j f)), \phi\rangle \right| \leq C(r,\phi)\|H_j f \|_{H^{-r}(\mathbb{T}^2)}\|H_j f \|_{H^{r-1}(\mathbb{T}^2)}\leq C(r,\phi)\|H_j f \|_{H^{-r}(\mathbb{T}^2)}^2.
\end{equation} 
Hence \eqref{tempered distribution} is well-defined. More concretely, $ C(r,\phi)$ can be controlled by $C(\iota)\|\phi\|_{H^{3+r+\iota}(\mathbb{T}^2)} $ for some $ \iota>0. $ The definition \eqref{decomposition} is  independent of the choice of $ j\in\mathbb{Z}$.

\subsection{The Regularization of Noise}

Recall that $ \{e_k\}_{k\in\mathbb{Z}^2\backslash\{0\}} $ is an orthonormal basis of $ H$.
A cylindrical Wiener process $ W $ on $ H $  has the following representation,
$$
W(t)=\sum_{k \in \mathbb{Z}^2\backslash\{0\}} \beta_k(t) e_k, \quad t \in[0, T],
$$
where $\{\beta_k\}_{k\in\mathbb{Z}^2\backslash\{0\}}$ is a family of independent real-valued Brownian motions. 
For every $\delta>0$ and fixed $s>\alpha+1$, we define the regularizing operator $ Q_{\delta}$ as 
$Q_{\delta}:=(I+\delta(-\Delta)^{s})^{-1},$
then
\begin{align*}
	\sqrt{Q_{\delta}}e_k=(I+\delta|k|^{2s})^{-1/2}e_k=:\lambda_{\delta,k}e_k, \quad k\in \mathbb{Z}^2\backslash\{0\}.
\end{align*}
For a bounded operator $ T: H\rightarrow H $, the Hilbert-Schmidt norm of $ T $ is defined as 
$$	\|T\|_{HS}^2:=\sum_{k \in \mathbb{Z}^2 \backslash \{0\}}\|Te_k\|_H^2.$$
A direct computation shows that
\begin{equation*}
	\begin{aligned}
		\left\|\Lambda^{\alpha} \circ \sqrt{Q_\delta}\right\|_{HS}^2&=
		\sum_{k \in \mathbb{Z}^2 \backslash \{0\}}|k|^{2\alpha}\lambda_{\delta, k}^2 =\sum_{k \in \mathbb{Z}^2 \backslash \{0\}} \frac{|k|^{2\alpha}}{1+\delta|k|^{2 s}} \\&\lesssim \int_1^{\infty} \frac{r^{2\alpha+1}}{1+\delta r^{2s}} \,\mathrm{d}r\lesssim \delta^{-\frac{\alpha+1}{s}} \int_{\delta^{\frac{1}{2s}}}^{\infty} \frac{u^{2\alpha+1}}{1+u^{2 s}} \,\mathrm{d} u\lesssim \delta^{-\frac{\alpha+1}{s}}.
	\end{aligned}
\end{equation*} This indicates the blow-up speed of solutions of \eqref{SPDE-SQG} and is therefore related to the scaling regimes we will encounter later on. This lead us to require that the scaling regime \eqref{scaling-delta} holds for $(\varepsilon,\delta(\varepsilon))$ .
Similarly, we require that the scaling regime $(\varepsilon,m(\varepsilon))$ for the Galerkin sequence \eqref{Galerkin approximation equations-delta=0} satisfies
\begin{equation}\label{scaling-m}
	\lim_{\varepsilon\rightarrow0}\varepsilon m(\varepsilon)^{2+2\alpha}=0.
\end{equation}
In this context, we denote $ \xi=\mathrm{d}W(t)/\mathrm{d}t $ and $ \xi_{\delta(\varepsilon)}=\sqrt{Q_{\delta(\varepsilon)}}\mathrm{d}W(t)/\mathrm{d}t.$

\subsection{Basic results of large deviations} 

The following entropy method will be used in our proof of the lower bound.
\begin{lemma}\label{lemma-entropy-method}\cite[Lemma 7]{Mariani10},\cite[Lemma 3.2]{GHW24_LLNS}
	Let $E$ be a separable and Hausdorff space, $I$ be a function from $ E $ to $ [0,+\infty],$ and  $\{\mu_{\varepsilon}\}_{\varepsilon>0}$ be a family of probability measures on $ E $. Then $\{\mu_{\varepsilon}\}$ satisfies the large deviation lower bound with speed $\varepsilon^{-1}$ and rate function $I$ if and only if, for every $x\in E$, there exists a sequence of probability measures $\{\pi_{\varepsilon,x}\}_{\varepsilon>0}$ that satisfies the following two conditions.
	\item (i)  (Weak convergence)\quad $\pi_{\varepsilon,x}\rightarrow\delta_x$ weakly, as $ \varepsilon\rightarrow0$.
	\item  (ii) (The entropy inequality)
	\begin{equation}\label{entropy-leq-rate-function}
		\limsup_{\varepsilon\rightarrow0}\varepsilon{\rm{Ent}}(\pi_{\varepsilon,x}|\mu_{\varepsilon})\leq I(x).
	\end{equation}
\end{lemma}
Let $ \mu_{\varepsilon,\delta(\varepsilon)}^{0}:=\mathcal{G}\left(0, \varepsilon Q_{\delta(\varepsilon)} / 2\right) $ be the Gaussian measure on $H^{\alpha-2\beta}(\mathbb{T}^2)$.  The following large deviations of initial data satisfying Assumption \ref{Gaussian-initial-data} is an application of Lemma \ref{lemma-entropy-method}, whose conditions are straightforward to verify.
\begin{proposition}\label{prop-ldp-initial}
	Assume that the scaling regime \eqref{scaling-delta} holds for $ \varepsilon, \delta(\varepsilon)>0, $ then $ \{\mu_{\varepsilon,\delta(\varepsilon)}^{0}\}_{\varepsilon>0} $ satisfy the large deviations on $ H^{\alpha-2\beta}_w(\mathbb{T}^2) $ with speed $ \varepsilon^{-1} $ and rate function $ \mathcal{I}_0(\phi)=\|\phi\|_{H^{\alpha-2\beta}(\mathbb{T}^2)}^2$.
\end{proposition}

\section{Existence and Exponential Tightness}\label{section-Tightness}
In this section, we will establish the existence of solutions of the regularized stochastic PDE \eqref{SPDE-SQG}
for every $\varepsilon,\delta(\varepsilon)>0$,  and the exponential tightness under the scaling regimes \eqref{scaling-delta} and \eqref{scaling-m}.

\begin{theorem}[Existence of stochastic generalized Leray solutions]\label{thm-existence-gamma}
	For any $\varepsilon, \delta(\varepsilon)>0,$  let  $\big(\tilde\Omega, \mathcal{\tilde F}, \mathbb{\tilde P}\big)$ be a probability space and let $\tilde{\theta}_0$ be an $H^{\alpha-2\beta}(\mathbb{T}^2)$-valued random variable on $\big(\tilde\Omega, \mathcal{\tilde F}, \mathbb{\tilde P}\big)$ satisfying Assumption \ref{Gaussian-initial-data}. 
	Then there exists  a stochastic generalized Leray solution $(\theta_{\varepsilon,\delta(\varepsilon)}, W)$ of \eqref{SPDE-SQG} on a new stochastic basis $\left(\Omega, \mathcal{F},\{\mathcal{F}_t\}_{t \in[0, T]}, \mathbb{P}\right)$, such that $\theta_{\varepsilon,\delta(\varepsilon)}(0)$ has the same law as  $\tilde{\theta}_0$ and is independent of $ W. $
\end{theorem}
\begin{proof}
	The existence of a martingale solution is established in \cite[Theorem 4.5]{ZZ14},  following from a standard argument similar to \cite{FG95_Martingale_2dNS}. And the pathwise $ H^{\alpha-2\beta}$-energy inequality can be derived in a similar way as in the proof of \cite[Proposition 5.1]{GHW24_LLNS}. 
\end{proof}

For any $\varepsilon>0$ and $ m(\varepsilon)\in\mathbb{N}_{+}$, we denote $ H_{m(\varepsilon)}:=\mathrm{span}\{e_k\}_{0<|k|\leq m(\varepsilon)}$ and let $ P_{m(\varepsilon)} $ be the projection operator from $ H $ to $ H_{m(\varepsilon)} $.
Setting $ \delta=0$ and $ \theta_{m(\varepsilon),0}=P_{m(\varepsilon)}\theta_0\in H_{m(\varepsilon)}$, the Galerkin approximation equation of \eqref{SPDE-SQG} is given by
\begin{equation}\label{Galerkin approximation equations-delta=0}
	\begin{aligned}
		&\mathrm{d}\theta_{\varepsilon,m(\varepsilon)}(t)=-\Lambda^{2\alpha}\theta_{\varepsilon,m(\varepsilon)}\,\mathrm{d}t-P_{m(\varepsilon)}\left( u_{\theta_{\varepsilon,m(\varepsilon)}}\cdot\nabla\theta_{\varepsilon,m(\varepsilon)}\right) \,\mathrm{d}t+\varepsilon^{1/2}\Lambda^{2\beta}P_{m(\varepsilon)} \,\mathrm{d}W(t),\\ &\theta_{\varepsilon,m(\varepsilon)}(0)=\theta_{m(\varepsilon),0}.
	\end{aligned}
\end{equation} For any $\varepsilon>0$ and $ m(\varepsilon)\in\mathbb{N}_{+}$, \eqref{Galerkin approximation equations-delta=0} admits a pathwise unique  probabilistically strong solution $ \theta_{\varepsilon, m(\varepsilon)}$(\cite[Theorem 4.2]{RZZ15},\cite[Theorem 1.1]{DC06_weak_strong}).
For any $\varepsilon, \delta(\varepsilon)>0$, let $ \theta_{\varepsilon, \delta(\varepsilon)} $ be a stochastic generalized Leray solution of \eqref{SPDE-SQG}.
We denote $ \mu_{\varepsilon, \delta(\varepsilon)} $ and $ \mu_{\varepsilon, m(\varepsilon)} $ as the laws of $ \theta_{\varepsilon, \delta(\varepsilon)} $ and $ \theta_{\varepsilon, m(\varepsilon)}  $, respectively.
In the remainder of this section, we will show that $ \{\mu_{\varepsilon, \delta(\varepsilon)} \}$ and  $ \{\mu_{\varepsilon, m(\varepsilon)}\}  $ are exponentially tight in $ \mathbb{X}_{\alpha,\beta} $ under \eqref{scaling-delta} and \eqref{scaling-m} respectively.

The following lemma is aimed at obtaining an estimate in the space $ L^\infty([0,T];H^{\alpha-2\beta}(\mathbb{T}^2))\cap L^2([0,T];H^{2\alpha-2\beta}(\mathbb{T}^2))$. Based on this estimate, we will derive the estimates in $W^{\sigma,2}([0,T];H^{-l}(\mathbb{T}^2))$ and $C^{\sigma}([0,T];H^{-l}(\mathbb{T}^2))$, and then prove the exponential tightness. 

\begin{lemma}\label{exponential-estimate-lemma1}
	Let $ \varepsilon>0 $ and $ m(\varepsilon)\in \mathbb{N}_+$.  Suppose that $ \theta_{\varepsilon, m(\varepsilon)} $ is the solution of \eqref{Galerkin approximation equations-delta=0} with $ \theta_{\varepsilon, m(\varepsilon)}(0)\sim \mathcal{G}(0,\varepsilon P_{m(\varepsilon)}/2)$, then there exists $\eta_0>0$ such that for every  $ \eta \in\left[0, \eta_0\right]$, under the scaling regime \eqref{scaling-m}, we have 
	$$
	\varepsilon \log \mathbb{E} \exp \left\{\frac{\eta}{2\varepsilon}\sup _{t \in[0, T]} \left\| \theta_{\varepsilon, m(\varepsilon)}(t)\right\|_{H^{\alpha-2\beta}(\mathbb{T}^2)}^2
	+\frac{\eta}{\varepsilon} \int_0^T  \left\| \theta_{\varepsilon, m(\varepsilon)}(s)\right\|_{H^{2\alpha-2\beta}(\mathbb{T}^2)}^2 \,\mathrm{d}s\right\} \leqslant C\left(\eta, T\right),
	$$
	and
	$$
	\limsup _{R \rightarrow \infty} \sup _{\varepsilon \in(0,1)} \varepsilon \log \mathbb{P}\big(\sup _{t \in[0, T]} \left\| \theta_{\varepsilon, m(\varepsilon)}(t)\right\|_{H^{\alpha-2\beta}(\mathbb{T}^2)}^2
	+\int_0^T \left\| \theta_{\varepsilon, m(\varepsilon)}(s)\right\|_{H^{2\alpha-2\beta}(\mathbb{T}^2)}^2 \,\mathrm{d}s\geqslant R\big)=-\infty .
	$$ 
\end{lemma}
\begin{proof}
	Applying It\^o's formula  to $ \eqref{Galerkin approximation equations-delta=0},$ for any $ t\in[0,T], $
	\begin{equation}\label{ito-H}
		\begin{aligned}
			\frac{1}{2}\left\| \theta_{\varepsilon, m(\varepsilon)}(t)\right\|_{H^{\alpha-2\beta}(\mathbb{T}^2)}^2
			&=\frac{1}{2}\left\| \theta_{\varepsilon, m(\varepsilon)}(0)\right\|_{H^{\alpha-2\beta}(\mathbb{T}^2)}^2 -\int_0^t  \left\| \theta_{\varepsilon, m(\varepsilon)}(s)\right\|_{H^{2\alpha-2\beta}(\mathbb{T}^2)}^2 \,\mathrm{d}  s \\
			& +\varepsilon^{1 / 2} \int_0^t\left\langle \Lambda^{2\alpha-2\beta}\theta_{\varepsilon, m(\varepsilon)}(s), P_{m(\varepsilon)} \mathrm{d} W(s)\right\rangle +\frac{\varepsilon}{2}t\left\|P_{m(\varepsilon)}\Lambda^{\alpha}\right\|_{HS}^2.
		\end{aligned}
	\end{equation}
	It follows from \eqref{ito-H}, Burkholder-Davis-Gundy inequality, and the scaling regime \eqref{scaling-m}  that
	\begin{equation}\label{L^2-Ito-estimate-of-SDE}
		\begin{aligned}
			\frac{1}{2}	\mathbb{E}\sup _{t \in[0, T]}\left\| \theta_{\varepsilon, m(\varepsilon)}(t)\right\|_{H^{\alpha-2\beta}(\mathbb{T}^2)}^2 
			&+\mathbb{E}\int_0^T \left\| \theta_{\varepsilon, m(\varepsilon)}(s)\right\|_{H^{2\alpha-2\beta}(\mathbb{T}^2)}^2   \,\mathrm{d} s\\&\leqslant\frac{1}{2}\mathbb{E}\left\| \theta_{\varepsilon, m(\varepsilon)}(0)\right\|_{H^{\alpha-2\beta}(\mathbb{T}^2)}^2  +CT.
		\end{aligned}
	\end{equation}
	For any $ \eta>0 $, the process $ \exp\left\{\frac{\eta}{\varepsilon^{1 / 2}}\int_0^t\left\langle \Lambda^{2\alpha-2\beta}\theta_{\varepsilon, m(\varepsilon)}(s), P_{m(\varepsilon)} \mathrm{d} W(s)\right\rangle\right\} $ is a submartingale. By Doob's maximal inequality, we see that \begin{equation}\label{Doob-ineq}
		\begin{aligned}
			\mathbb{E}\left[\sup_ { t \in [ 0 , T ] }\right. &\left.  \exp\left\{\frac{2\eta}{\varepsilon^{1 / 2}} \int_0^t\left\langle \Lambda^{2\alpha-2\beta}\theta_{\varepsilon, m(\varepsilon)}(s), P_{m(\varepsilon)} \mathrm{d} W(s)\right\rangle\right\}\right]\\&\leqslant 4	\mathbb{E}\left[\exp\left\{\frac{2\eta}{\varepsilon^{1 / 2}} \int_0^T\left\langle \Lambda^{2\alpha-2\beta}\theta_{\varepsilon, m(\varepsilon)}(s), P_{m(\varepsilon)} \mathrm{d} W(s)\right\rangle \right\}\right].
		\end{aligned}
	\end{equation}
	For any $\varepsilon$ and $\eta>0$, taking the exponential function of both sides of \eqref{ito-H} and using \eqref{scaling-m} yield that 
	\begin{equation}\label{exponential-of-L^2}
		\begin{aligned}
			& \exp \left\{\frac{\eta}{2\varepsilon} \left\| \theta_{\varepsilon, m(\varepsilon)}(t)\right\|_{H^{\alpha-2\beta}(\mathbb{T}^2)}^2 \right\} \cdot\exp \left\{\frac{\eta}{\varepsilon}\int_{0}^t  \left\| \theta_{\varepsilon, m(\varepsilon)}(s)\right\|_{H^{2\alpha-2\beta}(\mathbb{T}^2)}^2 \,\mathrm{d} s\right\}  \\
			\leqslant & \exp\big\{\frac{C(\eta)}{\varepsilon}T\big\}	\cdot	\exp \left\{\frac{\eta}{2\varepsilon} \left\| \theta_{\varepsilon, m(\varepsilon)}(0)\right\|_{H^{\alpha-2\beta}(\mathbb{T}^2)}^2 \right\} \cdot \exp \left\{\frac{\eta}{\varepsilon^{1 / 2}} \int_0^t\left\langle \Lambda^{2\alpha-2\beta}\theta_{\varepsilon, m(\varepsilon)}(s), P_{m(\varepsilon)}\mathrm{d} W(s)\right\rangle\right\}.  
		\end{aligned}
	\end{equation}
	For any $ \eta\in (0,1/2) $, according to the scaling regime \eqref{scaling-m}, the initial data part can be calculated by
	\begin{equation}\label{exp-estimate-0}
		\begin{aligned}
			\quad\varepsilon\log\mathbb{E}\Big(\exp\Big\{\frac{\eta}{\varepsilon}\left\| \theta_{\varepsilon, m(\varepsilon)}(0)\right\|_{H^{\alpha-2\beta}(\mathbb{T}^2)}^2\Big\}\Big)&=\sum_{0<|k|\leqslant m(\varepsilon)}\varepsilon\log\mathbb{E}\Big(\exp\Big\{\frac{\eta}{\varepsilon}|k|^{2\alpha-4\beta}\langle e_k, \theta_{\varepsilon, m(\varepsilon)}(0)\rangle^2\Big\}\Big)\\&=\sum_{0<|k|\leqslant m(\varepsilon)}\varepsilon\log(1-\eta)^{-1/2}\leqslant\varepsilon\sum_{0<|k| \leqslant m(\varepsilon)}1\leqslant C(\eta). 	
		\end{aligned}
	\end{equation}
	Taking supremum and expectation  to \eqref{exponential-of-L^2}, combining with H\"older's inequality, \eqref{Doob-ineq}, and  \eqref{exp-estimate-0}, we obtain
	\begin{equation}\label{Expection-of-exponential-of-L^2}
		\begin{aligned}
			&\mathbb{E} \exp \left\{\frac{\eta}{2\varepsilon} \sup _{t \in[0, T]}\left\| \theta_{\varepsilon, m(\varepsilon)}(t)\right\|_{H^{\alpha-2\beta}(\mathbb{T}^2)}^2
			+\frac{\eta}{ \varepsilon} \int_0^T \left\| \theta_{\varepsilon, m(\varepsilon)}(s)\right\|_{H^{2\alpha-2\beta}(\mathbb{T}^2)}^2 \,\mathrm{d}s\right\} \\
			=&\mathbb{E} \left[ \sup _{t \in[0, T]}\exp \left\{\frac{\eta}{2\varepsilon} \left\| \theta_{\varepsilon, m(\varepsilon)}(t)\right\|_{H^{\alpha-2\beta}(\mathbb{T}^2)}^2
			+\frac{\eta}{\varepsilon} \int_0^T \left\| \theta_{\varepsilon, m(\varepsilon)}(s)\right\|_{H^{2\alpha-2\beta}(\mathbb{T}^2)}^2\,\mathrm{d}s\right\} \right] \\
			\leqslant& \exp \left\{\frac{C(\eta)}{\varepsilon} T\right\} \mathbb{E} \exp\bigg[ \left\{\frac{\eta}{2\varepsilon} \left\| \theta_{\varepsilon, m(\varepsilon)}(0)\right\|_{H^{\alpha-2\beta}(\mathbb{T}^2)}^2 \right\} \\\cdot& \sup_{ t \in [ 0 , T ] }\left( \exp\left\{\frac{\eta}{\varepsilon^{1 / 2}} \int_0^t\left\langle \Lambda^{2\alpha-2\beta}\theta_{\varepsilon, m(\varepsilon)}(s), P_{m(\varepsilon)} \mathrm{d} W(s)\right\rangle\right\} \right) \bigg] \\
			\lesssim& \exp \left\{\frac{C(\eta)}{\varepsilon} T\right\} \left[ \mathbb{E} \exp\left\{\frac{\eta}{\varepsilon} \left\| \theta_{\varepsilon, m(\varepsilon)}(0)\right\|_{H^{\alpha-2\beta}(\mathbb{T}^2)}^2 \right\}  \right]^{1/2}\\ \cdot&\left[ \mathbb{E} \sup_{ t \in [ 0 , T ] }\exp \left\{\frac{2\eta}{\varepsilon^{1 / 2}} \int_0^t\left\langle \Lambda^{2\alpha-2\beta}\theta_{\varepsilon, m(\varepsilon)}(s), P_{m(\varepsilon)} \mathrm{d} W(s)\right\rangle\right\} \right] ^{1/2}\\ \lesssim&
			\exp \left\{\frac{C(\eta)}{\varepsilon} T\right\} 	\left[\mathbb{E}\exp\left\{\frac{2\eta}{\varepsilon^{1 / 2}} \int_0^T\left\langle \Lambda^{2\alpha-2\beta}\theta_{\varepsilon, m(\varepsilon)}(s), P_{m(\varepsilon)} \mathrm{d} W(s)\right\rangle\right\}\right]^{1/2}\end{aligned}
	\end{equation}
	By the covariance structure of Brownian motions and the definition of the quadratic variation process, it holds that
	$$
	\left\langle\left\langle\frac{4\eta}{\varepsilon^{1 / 2}} \int_0^{\cdot}\left\langle \Lambda^{2\alpha-2\beta}\theta_{\varepsilon, m(\varepsilon)}(s), P_{m(\varepsilon)} \mathrm{d} W(s)\right\rangle\right\rangle\right\rangle_t 
	=  \frac{16\eta^2}{\varepsilon} \int_0^t  \left\| \theta_{\varepsilon, m(\varepsilon)}(s)\right\|_{H^{2\alpha-2\beta}(\mathbb{T}^2)}^2 \,\mathrm{d}s .
	$$
	Thus $\exp \left\{\frac{4\eta}{\varepsilon^{1 / 2}} \int_0^t\left\langle \Lambda^{2\alpha-2\beta}\theta_{\varepsilon, m(\varepsilon)}(s), P_{m(\varepsilon)} \mathrm{d} W(s)\right\rangle-\frac{8\eta^2}{ \varepsilon} \int_0^t  \left\| \theta_{\varepsilon, m(\varepsilon)}(s)\right\|_{H^{2\alpha-2\beta}(\mathbb{T}^2)}^2 \,\mathrm{d}s\right\}$ is a local martingale. Moreover, for any $ \eta \in (0, 1/8) $, 
	\begin{equation*}
		M_1(t):=\exp \left\{\frac{4\eta}{\varepsilon^{1 / 2}} \int_0^t\left\langle \Lambda^{2\alpha-2\beta}\theta_{\varepsilon, m(\varepsilon)}(s), P_{m(\varepsilon)} \mathrm{d} W(s)\right\rangle-\frac{\eta}{ \varepsilon} \int_0^t  \left\| \theta_{\varepsilon, m(\varepsilon)}(s)\right\|_{H^{2\alpha-2\beta}(\mathbb{T}^2)}^2 \,\mathrm{d}s\right\}
	\end{equation*}
	is a supermartingale. In combination with H\"older's inequality, we deduce that
	\begin{equation}\label{exponential-of-martingale-term}
		\begin{aligned}
			&\mathbb{E}\left[\exp\left\{\frac{2\eta}{\varepsilon^{1 / 2}} \int_0^T\left\langle \Lambda^{2\alpha-2\beta}\theta_{\varepsilon, m(\varepsilon)}(s), P_{m(\varepsilon)} \mathrm{d} W(s)\right\rangle\right\}\right]\\
			\lesssim& \left[\mathbb{E} \exp\left\lbrace  \frac{\eta}{\varepsilon}\int_0^T \left\| \theta_{\varepsilon, m(\varepsilon)}(s)\right\|_{H^{2\alpha-2\beta}(\mathbb{T}^2)}^2 \,\mathrm{d} s\right\rbrace \right]^{1/2}\\ 
			& \cdot\left[ \mathbb{E}\exp \left\{\frac{4\eta}{\varepsilon^{1 / 2}} \int_0^T\left\langle \Lambda^{2\alpha-2\beta}\theta_{\varepsilon, m(\varepsilon)}(s), P_{m(\varepsilon)} \mathrm{d} W(s)\right\rangle-\frac{\eta}{ \varepsilon} \int_0^T\ \left\| \theta_{\varepsilon, m(\varepsilon)}(s)\right\|_{H^{2\alpha-2\beta}(\mathbb{T}^2)}^2 \,\mathrm{d} s\right\} \right] ^{1/2}\\
			\lesssim&  \left[\mathbb{E} \exp\left\lbrace  \frac{\eta}{\varepsilon} \int_{0}^{T} \left\| \theta_{\varepsilon, m(\varepsilon)}(s)\right\|_{H^{2\alpha-2\beta}(\mathbb{T}^2)}^2  \,\mathrm{d} s\right\rbrace \right]^{1/2}.
		\end{aligned}
	\end{equation}
	Substituting \eqref{exponential-of-martingale-term} into 	\eqref{Expection-of-exponential-of-L^2}, it follows that $$\begin{aligned}
		&\quad\mathbb{E} \exp \left\{ \frac{\eta}{\varepsilon} \int_0^T  \left\| \theta_{\varepsilon, m(\varepsilon)}(s)\right\|_{H^{2\alpha-2\beta}(\mathbb{T}^2)}^2 \,\mathrm{d}s\right\}\\&\leqslant	\exp \left\{\frac{C(\eta)}{\varepsilon} T\right\} 	\left[\mathbb{E}\exp\left\{\frac{2\eta}{\varepsilon^{1 / 2}} \int_0^T\left\langle \Lambda^{2\alpha-2\beta}\theta_{\varepsilon, m(\varepsilon)}(s), P_{m(\varepsilon)} \mathrm{d} W(s)\right\rangle\right\}\right]^{1/2}\leqslant
		\exp \left\{\frac{C(\eta)}{\varepsilon} T\right\}.
	\end{aligned}$$In combination with \eqref{Expection-of-exponential-of-L^2}  and \eqref{exponential-of-martingale-term}, we conclude that	$$
	\mathbb{E} \exp \left\{\frac{\eta}{2\varepsilon} \sup _{t \in[0, T]} \left\| \theta_{\varepsilon, m(\varepsilon)}(t)\right\|_{H^{\alpha-2\beta}(\mathbb{T}^2)}^2+\frac{\eta}{ \varepsilon} \int_0^T  \left\| \theta_{\varepsilon, m(\varepsilon)}(s)\right\|_{H^{2\alpha-2\beta}(\mathbb{T}^2)}^2 \,\mathrm{d}s\right\} \leqslant
	\exp \left\{\frac{C(\eta)}{\varepsilon} T\right\}.
	$$
	Finally, according to the exponential Chebyshev's inequality,
	$$
	\begin{aligned}
		&\varepsilon \log \mathbb{P}\left(\frac{1}{2} \sup _{t \in[0, T]} \left\| \theta_{\varepsilon, m(\varepsilon)}(t)\right\|_{H^{\alpha-2\beta}(\mathbb{T}^2)}^2 + \int_0^T  \left\| \theta_{\varepsilon, m(\varepsilon)}(s)\right\|_{H^{2\alpha-2\beta}(\mathbb{T}^2)}^2 \,\mathrm{d}s \geqslant R\right) \\
		\leqslant & -\eta R+\varepsilon \log \mathbb{E}\left(\frac{\eta}{2\varepsilon} \sup _{t \in[0, T]} \left\| \theta_{\varepsilon, m(\varepsilon)}(t)\right\|_{H^{\alpha-2\beta}(\mathbb{T}^2)}^2 +\frac{\eta}{ \varepsilon} \int_0^T  \left\| \theta_{\varepsilon, m(\varepsilon)}(s)\right\|_{H^{2\alpha-2\beta}(\mathbb{T}^2)}^2 \,\mathrm{d} s\right) \\
		\leqslant & -\eta R+C(\eta)T .
	\end{aligned}
	$$
	We complete the proof by sending $R \rightarrow \infty$.
\end{proof}

\begin{remark}
	The proof of the exponential estimates for $ \theta_{\varepsilon,\delta(\varepsilon)} $ can be handled in much the same way under the scaling regime \eqref{scaling-delta} with initial data satisfying Assumption \ref{Gaussian-initial-data}. The only difference is that we need to replace the finite-dimensional It\^o's formula by the pathwise $H^{\alpha-2\beta}$-energy equality \eqref{pathwise energy inequality}. In contrast, we  prove the following exponential estimates for $ \theta_{\varepsilon,\delta(\varepsilon)} $ under the scaling regime \eqref{scaling-delta}. These results hold for  $\theta_{\varepsilon, m(\varepsilon)}$ under the scaling regime \eqref{scaling-m} by the same arguments. 
\end{remark}

\begin{lemma}\label{exponential-estimate-lemma2}
	Let $ \varepsilon, \delta(\varepsilon)>0$. Suppose that $\theta_{\varepsilon,\delta(\varepsilon)}$ is a stochastic generalized Leray solution of \eqref{SPDE-SQG} in the sense of Definition \ref{def-stochastic-leray}, with initial data $ \theta_{\varepsilon, \delta(\varepsilon)}(0) $ satisfying Assumption \ref{Gaussian-initial-data}. Fix $\sigma\in(0,1/2)$ and $l>3+2\beta-\alpha$. Then there exists $\eta_0>0$ such that for every  $\eta\in[0,\eta_0]$, under the scaling regime \eqref{scaling-delta}, we have 
	\begin{equation}\label{exponential-estimate-lemma2-1}
		\varepsilon\log\mathbb{E}\Big(\exp\Big\{\frac{\eta}{\varepsilon}\|\theta_{\varepsilon,\delta(\varepsilon)}\|_{W^{\sigma,2}([0,T];H^{-l}(\mathbb{T}^2))}^2\Big\}\Big)\leq C(\eta,T), 
	\end{equation}
	and
	\begin{equation}\label{exponential-estimate-lemma2-2}
		\limsup_{R\rightarrow\infty}\sup_{\varepsilon\in(0,1)}\varepsilon\log\mathbb{P}\Big(\|\theta_{\varepsilon,\delta(\varepsilon)}\|_{W^{\sigma,2}([0,T];H^{-l}(\mathbb{T}^2))}>R\Big)=-\infty.  
	\end{equation}
\end{lemma}
\begin{proof}
	
	Taking the $ W^{\sigma,2}([0,T];H^{-l}(\mathbb{T}^2)) $ norm to $ \theta_{\varepsilon,\delta(\varepsilon)}(t)$, it follows from the convexity of the exponential function that 
	\begin{equation*}
		\begin{aligned}
			&\quad\varepsilon\log\mathbb{E}\exp\Big\{\frac{\eta}{\varepsilon}\|	\theta_{\varepsilon,\delta(\varepsilon)}\|_{W^{\sigma,2}([0,T];H^{-l}(\mathbb{T}^2))}\Big\}
			\\\lesssim& \varepsilon\log\mathbb{E}\exp\Big\{\frac{\eta}{\varepsilon}\|	\theta_{\varepsilon,\delta(\varepsilon)}(0)\|_{W^{\sigma,2}([0,T];H^{-l}(\mathbb{T}^2))}\Big\}+\varepsilon\log\mathbb{E}\exp\Big\{\frac{\eta}{\varepsilon}\Big\|\int^{\cdot}_0\Lambda^{2\alpha} \theta_{\varepsilon,\delta(\varepsilon)}\Big\|_{W^{\sigma,2}([0,T];H^{-l}(\mathbb{T}^2))}\Big\}\\&+\varepsilon\log\mathbb{E}\exp\Big\{\frac{\eta}{\varepsilon}\Big\|\int^{\cdot}_0\nabla \cdot\left(  \theta_{\varepsilon,\delta(\varepsilon)}R^\perp \theta_{\varepsilon,\delta(\varepsilon)}\right)  \Big\|_{W^{\sigma,2}([0,T];H^{-l}(\mathbb{T}^2))}\Big\}\\&+\varepsilon\log\mathbb{E}\exp\Big\{\frac{\eta}{\varepsilon^{1/2}}\|\Lambda^{2\beta}\sqrt{Q_{\delta(\varepsilon)}}W\|_{W^{\sigma,2}([0,T];H^{-l}(\mathbb{T}^2))}\Big\}=:J_1+J_2+J_3+J_4.	\end{aligned}
	\end{equation*} 
	Due to Proposition \ref{prop-ldp-initial} and the fact that
	\begin{equation*}
		\|\theta_{\varepsilon,\delta(\varepsilon)}(0)\|_{W^{\sigma,2}([0,T];H^{-l}(\mathbb{T}^2))}\leq C(T)\|\theta_{\varepsilon,\delta(\varepsilon)}(0)\|_{H^{\alpha-2\beta}(\mathbb{T}^2)},
	\end{equation*}
	we have  \begin{equation}\label{J1}
		J_1\leq C(\eta,T). 
	\end{equation}
	Using the Sobolev embedding  $W^{1,2}([0,T];H^{2\alpha-2\beta}(\mathbb{T}^2))\subset W^{\sigma,2}([0,T];H^{2\alpha-l}(\mathbb{T}^2))$ and Lemma \ref{exponential-estimate-lemma1}, we deduce that \begin{equation}\label{J2}
		\begin{aligned}
			J_2\lesssim&\varepsilon\log\mathbb{E}\exp\Big\{\frac{\eta}{\varepsilon}\Big\|\int^{\cdot}_0\Lambda^{\alpha} \theta_{\varepsilon,\delta(\varepsilon)}\Big\|_{W^{1,2}([0,T];H^{\alpha-2\beta}(\mathbb{T}^2))}\Big\}\\
			\lesssim&\varepsilon\log\mathbb{E}\exp\Big\{\frac{\eta}{\varepsilon}\int^T_0\|\Lambda^{\alpha} \theta_{\varepsilon,\delta(\varepsilon)}\|_{H^{\alpha-2\beta}(\mathbb{T}^2)}^2\,\mathrm{d}s\Big\}\leq  C(\eta,T). 
		\end{aligned}	
	\end{equation}
	For the nonlinear term, since $ l>3+2\beta-\alpha$, \eqref{decomposition} and \eqref {high-frequency} imply that for any test function $ \phi \in H^\infty(\mathbb{T}^2) $, for any $ j\in\mathbb{Z} $ and $ \iota>0, $
	\begin{equation}\label{exp-nonlinear}
		\begin{aligned}
			&\quad|\langle\nabla \cdot\left(  \theta_{\varepsilon,\delta(\varepsilon)}R^\perp \theta_{\varepsilon,\delta(\varepsilon)}\right) , \phi \rangle|\\& \lesssim	|\langle\theta_{\varepsilon,\delta(\varepsilon)}R^\perp (S_j\theta_{\varepsilon,\delta(\varepsilon)}), \nabla \phi \rangle|+|\langle S_j \theta_{\varepsilon,\delta(\varepsilon)}R^\perp (H_j\theta_{\varepsilon,\delta(\varepsilon)}), \nabla \phi \rangle|+|\langle\nabla \cdot\left( H_j \theta_{\varepsilon,\delta(\varepsilon)}R^\perp (H_j\theta_{\varepsilon,\delta(\varepsilon)})\right), \phi \rangle|\\& \lesssim \left[\left\|S_j \theta_{\varepsilon,\delta(\varepsilon)}\right\|_{H^{2\beta-\alpha}(\mathbb{T}^2)}\|\theta_{\varepsilon,\delta(\varepsilon)}\|_{H^{\alpha-2\beta}(\mathbb{T}^2)}+\|\theta_{\varepsilon,\delta(\varepsilon)}\|_{H^{\alpha-2\beta}(\mathbb{T}^2)}^2\right]\|\phi \|_{H^{3+2\beta-\alpha+\iota}(\mathbb{T}^2)} \\&\lesssim \|\theta_{\varepsilon,\delta(\varepsilon)}\|_{H^{\alpha-2\beta}(\mathbb{T}^2)}^2\|\phi \|_{H^{3+2\beta-\alpha+\iota}(\mathbb{T}^2)}.
		\end{aligned}
	\end{equation}
	Consequently, according to Lemma \ref{exponential-estimate-lemma1},
	\begin{equation}\label{J3}
		J_3
		\lesssim\varepsilon\log\mathbb{E}\exp\Big\{\frac{\eta}{\varepsilon}\int^T_0\|\theta_{\varepsilon,\delta(\varepsilon)}\|_{H^{\alpha-2\beta}(\mathbb{T}^2)}^2\,\mathrm{d}s\Big\} \leq C(\eta,T). 
	\end{equation}
	It remains to estimate the Brownian motion term $ J_4 $. For any $ c>0,$ it follows from Young's inequality that
	\begin{equation}\label{exponential-for-gaussian}
		\begin{aligned}
			J_4\lesssim&\varepsilon\log\mathbb{E}\exp\Big\{\frac{\eta}{\varepsilon^{1/2}}\|W(\cdot)\|_{W^{\sigma,2}([0,T];H^{\alpha-2\beta-1}(\mathbb{T}^2))}\Big\}\\
			\lesssim&\varepsilon\log\mathbb{E}\exp\Big\{c\|W(\cdot)\|_{W^{\sigma,2}([0,T];H^{\alpha-2\beta-1}(\mathbb{T}^2))}^2+\frac{\eta^2}{4c\varepsilon}\Big\}.
		\end{aligned}
	\end{equation}
	Since $W$ induces a Gaussian measure on $W^{\sigma,2}([0,T];H^{\alpha-2\beta-1}(\mathbb{T}^2))$, with the help of Fernique's theorem \cite[Remark 2.8]{DPZ14}, one can choose a  sufficiently small $c>0$ such that
	\begin{align}\label{fernique}
		\mathbb{E}\exp\Big\{c\|W(\cdot)\|_{W^{\sigma,2}([0,T];H^{\alpha-2\beta-1}(\mathbb{T}^2))}^2\Big\}\leq C(T)<\infty.
	\end{align}
	Substituting \eqref{fernique} into \eqref{exponential-for-gaussian}, we obtain that
	\begin{equation}\label{J4}
		J_4\leq C(\eta,T). 
	\end{equation}
	Adding \eqref{J1}, \eqref{J2}, \eqref{J3}, and \eqref{J4} together, we complete the proof of \eqref{exponential-estimate-lemma2-1}. Finally, \eqref{exponential-estimate-lemma2-2} follows from the exponential Chebyshev's inequality. 
	
\end{proof}

\begin{lemma}\label{exponential-estimate-lemma3}
	Let $\varepsilon,\delta(\varepsilon) $ and $\theta_{\varepsilon,\delta(\varepsilon)}$ satisfy the conditions in Lemma \ref{exponential-estimate-lemma2}. Fix $\sigma\in(0,1/2)$ and $l>3+2\beta-\alpha$. Then there exists $\eta_0>0$ such that for every  $\eta\in[0,\eta_0]$, under the scaling regime \eqref{scaling-delta}, we have
	\begin{equation*}
		\varepsilon\log\mathbb{E}\Big(\exp\Big\{\frac{\eta}{\varepsilon}\|\theta_{\varepsilon,\delta(\varepsilon)}\|_{C^{\sigma}([0,T];H^{-l}(\mathbb{T}^2))}^2\Big\}\Big)\leq C(\eta,T), 
	\end{equation*}
	and
	\begin{equation*}
		\limsup_{R\rightarrow\infty}\sup_{\varepsilon\in(0,1)}\varepsilon\log\mathbb{P}\Big(\|\theta_{\varepsilon,\delta(\varepsilon)}\|_{C^{\sigma}([0,T];H^{-l}(\mathbb{T}^2))}>R\Big)=-\infty.  
	\end{equation*}
\end{lemma}
\begin{proof}
	The proof is similar to the proof of Lemma \ref{exponential-estimate-lemma2}. We only focus on the estimate of the nonlinear term. For any $ \phi\in H^{\infty}(\mathbb{T}^2) $ and any $ 0\leq s<t\leq T$,  \eqref{exp-nonlinear} and H\"older's inequality imply that  $$\begin{aligned}
		&\quad|\langle\int^{t}_s\nabla \cdot\left(  \theta_{\varepsilon,\delta(\varepsilon)}R^\perp \theta_{\varepsilon,\delta(\varepsilon)}\right)\,\mathrm{d}r, \phi \rangle|\lesssim \int^{t}_s|\langle\nabla \cdot\left(  \theta_{\varepsilon,\delta(\varepsilon)}R^\perp \theta_{\varepsilon,\delta(\varepsilon)}\right) , \phi \rangle|\,\mathrm{d}r\\&\lesssim\int^{t}_s \|\theta_{\varepsilon,\delta(\varepsilon)}\|_{H^{\alpha-2\beta}(\mathbb{T}^2)}^2\|\phi \|_{H^{l}(\mathbb{T}^2)}\,\mathrm{d}r\lesssim(t-s)^{1/2}\sup _{t \in[0, T]} \left\| \theta_{\varepsilon, \delta(\varepsilon)}(t)\right\|_{H^{\alpha-2\beta}(\mathbb{T}^2)}^2\|\phi \|_{H^{l}(\mathbb{T}^2)}.
	\end{aligned}
	$$
	Applying Lemma \ref{exponential-estimate-lemma1}, it follows that
	\begin{align*}
		&\quad\varepsilon\log\mathbb{E}\exp\Big\{\frac{\eta}{\varepsilon}\Big\|\int^{\cdot}_0\nabla \cdot\left(  \theta_{\varepsilon,\delta(\varepsilon)}R^\perp \theta_{\varepsilon,\delta(\varepsilon)}\right) \Big\|_{C^{\sigma}([0,T];H^{-l}(\mathbb{T}^2))}\Big\}\\&
		\lesssim\varepsilon\log\mathbb{E}\exp\Big\{\frac{\eta}{\varepsilon}\sup _{t \in[0, T]} \left\| \theta_{\varepsilon, \delta(\varepsilon)}(t)\right\|_{H^{\alpha-2\beta}(\mathbb{T}^2)}^2\Big\}\leq  C(\eta,T).
	\end{align*}
\end{proof}
\begin{proposition}[Exponential tightness]\label{prop-exponential-tightness}
	Let $\varepsilon,\delta(\varepsilon) $ and $\theta_{\varepsilon,\delta(\varepsilon)}$ satisfy the conditions in Lemma \ref{exponential-estimate-lemma2}. Then under the scaling regime \eqref{scaling-delta},  there exists a sequence of compact set $\{K_n\}_{n\geq1}\subset\mathbb{X}_{\alpha,\beta}$, such that the laws $\{\mu_{\varepsilon,\delta(\varepsilon)}\}$ of $\{\theta_{\varepsilon,\delta(\varepsilon)}\}$ satisfy   
	\begin{equation*}
		\limsup_{n\rightarrow\infty}\limsup_{\varepsilon\rightarrow0}\varepsilon\log\mu_{\varepsilon,\delta(\varepsilon)}(K_n)=-\infty. 
	\end{equation*}
\end{proposition}
\begin{proof}
	This follows by the same argument as in \cite[Corollary 4.2]{GHW24_LLNS}. In our case, $ K_n  $ can be chosen as
	\begin{equation*}
		K_n:=\{\theta\in\mathbb{X}_{\alpha,\beta}:\|\theta\|_{L^2([0,T];H^{2\alpha-2\beta}(\mathbb{T}^2))}+\|\theta\|_{W^{\sigma,2}([0,T];H^{-l}(\mathbb{T}^2))}+\|\theta\|_{C^{\sigma}([0,T];H^{-l}(\mathbb{T}^2))}\leq n\},
	\end{equation*}where $\sigma\in(0,1/2)$ and  $l>3+2\beta-\alpha$.
	Indeed, Lemma \ref{exponential-estimate-lemma1}, Lemma \ref{exponential-estimate-lemma2}, and Aubin-Lions lemma imply the exponential tightness of $\{\mu_{\varepsilon,\delta(\varepsilon)}\}$ in $ L^2([0,T];H^{\alpha-2\beta}(\mathbb{T}^2)). $ The exponential tightness in  $	L_w^2([0,T];H^{2\alpha-2\beta}(\mathbb{T}^2))$ is a direct conclusion of Lemma \ref{exponential-estimate-lemma1}. Finally, it follows from Lemma \ref{exponential-estimate-lemma3} that $\{\mu_{\varepsilon,\delta(\varepsilon)}\}$ is exponentially tight in $C([0,T];H^{\alpha-2\beta}_w(\mathbb{T}^2)) $. 
\end{proof}

\section{Characterization of Rate Function}\label{section-rate}
Let $g\in L^2([0,T];H)$. We recall that the skeleton equation for the large deviations is given by \eqref{PDE-SQG}
in our setting, and the rate function is given by 
\begin{equation}\label{rate-function}
	\mathcal{I}(\theta)=\mathcal{I}_{0}(\theta(0))+\mathcal{I}_{dyna}(\theta),\ \ \theta\in\mathbb{X}_{\alpha,\beta},
\end{equation}
where
\begin{equation}\label{rate-stationary}
	\mathcal{I}_{0}(\phi)=
	\|\phi\|_{H^{\alpha-2\beta}(\mathbb{T}^2)}^2, \quad \phi\in H^{\alpha-2\beta}(\mathbb{T}^2),
\end{equation}
and 
\begin{equation}\label{rate-dynamic}
	\mathcal{I}_{dyna}(\theta)=\frac{1}{2}\inf\Big\{\|g\|_{L^2([0,T];H)}^2:	\partial_t\theta=\Lambda^{2\alpha}\theta-u_\theta\cdot\nabla \theta+\Lambda^{2\beta}g\Big\}.
\end{equation}
The equality in \eqref{rate-dynamic} holds in the following sense.

\begin{definition}[Weak solution of the skeleton equation]\label{def-weak-solution-skeleton-equation}
	Let $\theta_0\in H^{\alpha-2\beta}(\mathbb{T}^2)$ and $g\in L^2([0,T];H)$. We say that $\theta$ is a weak solution of \eqref{PDE-SQG} with initial data $\theta_0$ if 
	
	(i)
	$
	\theta\in  L^\infty([0,T];H^{\alpha-2\beta}(\mathbb{T}^2)) \cap  L^2([0,T];H^{2\alpha-2\beta}(\mathbb{T}^2))\cap C([0,T];H_w^{\alpha-2\beta}(\mathbb{T}^2)).
	$
	
	(ii) For every $\varphi\in  C^\infty([0,T];H^\infty(\mathbb{T}^2))$, for all $ t\in[0,T]$,
	\begin{equation}\label{weak-solution-skeleton-equation}
		\begin{aligned}
			&\langle \theta(t),\varphi(t)\rangle=\langle \theta_0,\varphi(0)\rangle+\int^t_0\langle \theta,\partial_s\varphi\rangle \,\mathrm{d}s-\int^t_0\langle \theta,\Lambda^{2\alpha}\varphi\rangle  \,\mathrm{d}s\\&-\frac{1}{2}\int^t_0\left\langle R_2 \theta,\left[\Lambda, \partial_1 \varphi\right]\Lambda^{-1} \theta\right\rangle \,\mathrm{d} s+\frac{1}{2}\int_0^t\left\langle R_1 \theta  ,\left[\Lambda, \partial_2 \varphi\right]\Lambda^{-1} \theta\right\rangle \,\mathrm{d}s+\int^t_0\langle\Lambda^{2\beta}\varphi,g\rangle \,\mathrm{d}s.
		\end{aligned}
	\end{equation}
\end{definition}
In the sequel, we will provide a variational characterization of $\mathcal{I}$, which is equivalent to the expression \eqref{rate-function}, \eqref{rate-stationary}, and \eqref{rate-dynamic}. We now define the map $\Lambda_0(\cdot,\cdot):C_b(H^{\alpha-2\beta}_w(\mathbb{T}^2))\times H^{\alpha-2\beta}(\mathbb{T}^2)\rightarrow\mathbb{R}$ by
\begin{equation}\label{Lambda-0}
	\Lambda_0(\psi,\phi):=\psi(\phi)-\lambda(\psi),
\end{equation}
where 	\begin{equation}\label{lambda}
	\lambda(\psi):=\sup_{\phi\in  {H^{\alpha-2\beta}(\mathbb{T}^2)}}\{\psi(\phi)-\mathcal{I}_{0}(\phi)\}.
\end{equation}
Moreover,  $\Lambda_1^T(\cdot,\cdot):C^{\infty}([0,T];H^\infty(\mathbb{T}^2))\times \mathbb{X}_{\alpha,\beta}\rightarrow\mathbb{R}$ is defined by
\begin{align*}
	&\Lambda_1^T(\varphi,\theta):=\langle \theta(T),\varphi(T)\rangle-\langle \theta(0),\varphi(0)\rangle-\int_{0}^{T}\langle \theta,\partial_t\varphi\rangle\,\mathrm{d}t+\int_{0}^{T}\langle \theta,\Lambda^{2\alpha}\varphi\rangle\,\mathrm{d}t\\
	&+\frac{1}{2}\int_{0}^{T}\left\langle R_2 \theta,\left[\Lambda, \partial_1 \varphi\right]\Lambda^{-1} \theta\right\rangle \,\mathrm{d}t-\frac{1}{2}\int_0^T\left\langle R_1 \theta  ,\left[\Lambda, \partial_2 \varphi\right]\Lambda^{-1} \theta\right\rangle \,\mathrm{d}t-\frac{1}{2}\|\Lambda^{2\beta}\varphi\|_{L^2([0,T];H)}^2.
\end{align*}
\begin{proposition}\label{prop-variational-characterization-rate-function}
	Let the rate functions $\mathcal{I}_{0},\mathcal{I}_{dyna},\mathcal{I}$ be defined by \eqref{rate-stationary}, \eqref{rate-dynamic}, and \eqref{rate-function}, respectively.  Then for any $\theta\in\mathbb{X}_{\alpha,\beta}$ with $\mathcal{I}(\theta)<\infty$, we have 
	\begin{equation}
		\mathcal{I}(\theta)=\sup_{\substack{\varphi\in C^{\infty}([0,T];H^\infty(\mathbb{T}^2))\\\psi\in C_b(H^{\alpha-2\beta}_w(\mathbb{T}^2))}}\Big\{\Lambda_0(\psi,\theta(0))+\Lambda_1(\varphi,\theta)\Big\}.  	
	\end{equation}
\end{proposition}
\begin{proof}
	It is sufficient to prove that for any $\theta\in\mathbb{X}_{\alpha,\beta}$ with $\mathcal{I}(\theta)<\infty$, we have 
	\begin{equation}\label{variational-stationary-part}
		\mathcal{I}_{0}(\theta(0))=\sup_{\psi\in C_b(H^{\alpha-2\beta}_w(\mathbb{T}^2))}	\Lambda_0(\psi,\theta(0)),
	\end{equation}
	and
	\begin{equation}\label{variational-dynamical-part}
		\mathcal{I}_{dyna}(\theta)=\sup_{\varphi\in C^{\infty}([0,T];H^\infty(\mathbb{T}^2))}\Lambda_1^T(\varphi,\theta).
	\end{equation}
	According to Proposition \ref{prop-ldp-initial}, $\mu_{\varepsilon, \delta(\varepsilon)}^0= \mathcal{G}(0, \varepsilon Q_{\delta(\varepsilon)}/2) $ satisfies the large deviations on $ H^{\alpha-2\beta}_w(\mathbb{T}^2)$  with rate function $ \mathcal{I}_0 $ under Assumption \ref{Gaussian-initial-data} and the scaling regime \eqref{scaling-delta}. Thus \eqref{variational-stationary-part} follows from Proposition \ref{prop-ldp-initial} and Bryc's lemma (see, for example, \cite[Theorem 4.4.2]{DZ10_LDP}). 
	Next we prove \eqref{variational-dynamical-part}. Let
	$F^T(\cdot,\theta):C^{\infty}([0,T];H^\infty(\mathbb{T}^2))\rightarrow\mathbb{R}$ be the linear part of $ 	\Lambda_1^T(\cdot,\theta):$
	\begin{equation}\label{the-map-F}
		\begin{aligned}
			F^T(\varphi,\theta):=&\langle \theta(T),\varphi(T)\rangle-\langle \theta(0),\varphi(0)\rangle-\int_{0}^{T}\langle \theta,\partial_t\varphi\rangle\,\mathrm{d}t+\int_{0}^{T} \langle\theta,\Lambda^{2\alpha}\varphi\rangle\,\mathrm{d}t\\
			&+\frac{1}{2}\int_{0}^{T}\left\langle R_2 \theta,\left[\Lambda, \partial_1 \varphi\right]\Lambda^{-1} \theta\right\rangle \,\mathrm{d}t-\frac{1}{2}\int_0^T\left\langle R_1 \theta  ,\left[\Lambda, \partial_2 \varphi\right]\Lambda^{-1} \theta\right\rangle \,\mathrm{d}t.
		\end{aligned}
	\end{equation}
	For any $\theta$ satisfying $I_{dyna}(\theta)<\infty$, by the definition of $\mathcal{I}_{dyna}(\theta)$,  there exists a function $g\in L^2([0,T];H)$ such that, for any $ \varphi \in C^{\infty}([0,T];H^\infty(\mathbb{T}^2)) $,
	\begin{align*}
		F^T(\varphi,\theta)=\int_{0}^{T}\langle g,\Lambda^{2\beta}\varphi\rangle\,\mathrm{d}t. 
	\end{align*}
	Using the definition of $\Lambda_1^T(\cdot,\theta)$ and Cauchy-Schwarz inequality, we have
	\begin{align*}
		\Lambda_1^T(\varphi,\theta)=\int^T_0\langle g,\Lambda^{2\beta}\varphi\rangle\,\mathrm{d} t-\frac{1}{2}\|\Lambda^{2\beta}\varphi\|_{L^2([0,T];H)}^2\leq\frac{1}{2}\|g\|_{L^2([0,T];H)}^2. 	
	\end{align*}
	Taking the supremum over $ \varphi \in C^{\infty}([0,T];H^\infty(\mathbb{T}^2)) $, it follows that 
	\begin{equation}\label{eq-inf}
		\sup_{ \varphi \in C^{\infty}([0,T];H^\infty(\mathbb{T}^2)) }\Lambda_1^T(\varphi,\theta)\leq\frac{1}{2}\|g\|_{L^2([0,T];H)}^2.	
	\end{equation}  Taking the infimum of the right-hand side of \eqref{eq-inf} over all $ g\in L^2([0,T];H) $ such that $\theta$ is a weak solution of \eqref{PDE-SQG} with $ g $, we conclude that  	\begin{align}\label{dynamical-ineq}
		\sup_{ \varphi \in C^{\infty}([0,T];H^\infty(\mathbb{T}^2)) }\Lambda_1^T(\varphi,\theta)\leq	\mathcal{I}_{dyna}(\theta).	
	\end{align} 
	In the sequel, we prove the reverse inequality of \eqref{dynamical-ineq}.  It is sufficient to discuss the case that $ \sup\limits_{ \varphi \in C^{\infty}([0,T];H^\infty(\mathbb{T}^2)) }\Lambda_1^T(\varphi,\theta) <\infty$.  By the definition of $ \Lambda_1^T $,
	\begin{align*}
		\sup_{\substack{\varphi \in C^{\infty}([0,T];H^\infty(\mathbb{T}^2)) ,\\\|\varphi\|_{L^2([0,T];H^{2\beta}(\mathbb{T}^2))}\leq 1 }}	|F^T(\varphi,\theta)|\leq&\sup\limits_{ \varphi \in C^{\infty}([0,T];H^\infty(\mathbb{T}^2)) }\Lambda_1^T(\varphi,\theta)+\frac{1}{2}<+\infty.
	\end{align*}
	Since $ C^{\infty}([0,T];H^\infty(\mathbb{T}^2))$ is dense in $ L^2([0,T];H^{2\beta}(\mathbb{T}^2)) $,  $ F^T(\cdot,\theta)$ can be extended to a continuous linear functional from $L^2([0,T];H^{2\beta}(\mathbb{T}^2)) $ to $\mathbb{R}$. Using Riesz's representation theorem, one can see that there exists $\Psi^\theta\in L^2([0,T];H^{2\beta}(\mathbb{T}^2)) $ such that
	\begin{equation}\label{riesz}
		F^T(\varphi,\theta)=\langle\Psi^\theta,\varphi\rangle_{L^2([0,T];H^{2\beta}(\mathbb{T}^2))}. 
	\end{equation}
	Let $g=g^\theta:=\Lambda^{2\beta}\Psi^\theta \in L^2([0,T];H)$, \eqref{riesz} implies that
	$\theta$ is a weak solution of \eqref{PDE-SQG} with respect to the control  $g$ in the sense of Definition \ref{def-weak-solution-skeleton-equation}.  Moreover, a dual Hilbert space analysis as in \cite[Appendix B]{GH23} implies that $$	\sup_{\varphi\in C^{\infty}([0,T];H^\infty(\mathbb{T}^2))}\Lambda_1^T(\varphi,\theta)=\frac{1}{2}\|g^\theta\|_{L^2([0,T];H)}^2 \geq	\mathcal{I}_{dyna}(\theta). 	$$
\end{proof}

\begin{remark}\label{uniqueness-of-g}
	It can be seen from the above proof that, for any $\theta\in\mathbb{X}_{\alpha,\beta}$ with $\mathcal{I}(\theta)<\infty$, $ g^\theta $ is the unique element in $ L^2([0,T];H) $ such that $ \theta $ is a weak solution of \eqref{PDE-SQG}.
\end{remark}
\begin{lemma} \label{lemma-lsc}
	Let $\mathcal{I}$ be the rate function defined by \eqref{rate-function}. Then $\mathcal{I}$ is lower semi-continuous with respect to the topology of $\mathbb{X}_{\alpha,\beta}$. 
\end{lemma}
\begin{proof}
	The lower semi-continuity of the $H^{\alpha-2\beta}(\mathbb{T}^2)$-norm with respect to the weak topology of $H^{\alpha-2\beta}(\mathbb{T}^2)$ implies  that the map $ \theta\mapsto\mathcal{I}_{0}(\theta(0))$ is lower semi-continuous on $\mathbb{X}_{\alpha,\beta}$. We are left with the task of proving the  lower semi-continuity of $ \mathcal{I}_{dyna}$. Suppose that $ \theta \in \mathbb{X}_{\alpha,\beta}$ and $ \{\theta_n\}_{n\geq 1} \subset \mathbb{X}_{\alpha,\beta} $ satisfy $\theta_n\rightarrow \theta$ in $\mathbb{X}_{\alpha,\beta}$ as $n\rightarrow\infty$.  
	With the help of Proposition  \ref{prop-variational-characterization-rate-function}, it is sufficient to show that for arbitrary fixed $\varphi\in C^{\infty}([0,T];H^\infty(\mathbb{T}^2))$, the map $\Lambda_1^T(\varphi,\cdot)$ is continuous with respect to the topology of $\mathbb{X}_{\alpha,\beta}$:
	\begin{equation}\label{convergence-lambda-1}
		\lim_{n\rightarrow \infty}\Lambda_1^T(\varphi,\theta_n)=\Lambda_1^T(\varphi,\theta).
	\end{equation}The definition of the topology of $\mathbb{X}_{\alpha,\beta}$ implies the convergence of $\langle \theta_n(T),\varphi(T)\rangle$, $\langle \theta_n(0),\theta(0)\rangle$, and  $\int_{0}^T\langle \theta_n,\partial_t\varphi\rangle\,\mathrm{d}t$ directly.
	The convergence of the dissipation term $ \int_{0}^T\langle \theta_n,\Lambda^{2\alpha}\varphi\rangle\,\mathrm{d}t$ follows from the fact that $\theta_n\rightarrow \theta$ weakly in $ L^2([0,T];H^{2\alpha-2\beta}(\mathbb{T}^2))  $. 
	
	It remains to show the convergence of the nonlinear term. For any $ j\in \mathbb{Z} $, according to \eqref{decomposition} and  \eqref{tempered distribution}, it holds that
	\begin{align*}
		&\quad\Big|\int_{0}^{T}\left( \left\langle R_2 \theta_n,\left[\Lambda, \partial_1 \varphi\right]\Lambda^{-1} \theta_n\right\rangle -\left\langle R_1 \theta_n  ,\left[\Lambda, \partial_2 \varphi\right]\Lambda^{-1} \theta_n\right\rangle\right) \mathrm{d}t\\&\quad\quad-\int_{0}^{T}\left( \left\langle R_2 \theta,\left[\Lambda, \partial_1 \varphi\right]\Lambda^{-1} \theta\right\rangle -\left\langle R_1 \theta  ,\left[\Lambda, \partial_2 \varphi\right]\Lambda^{-1} \theta\right\rangle\right) \mathrm{d}t \Big|
		\\&\leq \left|\int_0^T\langle\nabla\cdot\left( \theta_n R^{\perp}\left(S_j \theta_n\right)\right) , \varphi \rangle\,\mathrm{d} t-\int_0^T \langle\nabla\cdot\left(\theta R^{\perp}\left(S_j \theta\right)\right) , \varphi \rangle \,\mathrm{d} t\right|\\&+\left|\int_0^T \langle\nabla\cdot\left( S_j\theta_n R^{\perp}\left(H_j \theta_n\right)\right) , \varphi \rangle \,\mathrm{d} t-\int_0^T\langle \nabla\cdot\left(S_j\theta R^{\perp}\left(H_j \theta\right)\right), \varphi \rangle\,\mathrm{d} t\right|\\&+\left|\int_0^T \langle\nabla\cdot\left( H_j\theta_n R^{\perp}\left(H_j \theta_n\right)\right), \varphi \rangle\,\mathrm{d} t \right|+\left|\int_0^T\langle \nabla\cdot\left(H_j\theta R^{\perp}\left(H_j \theta\right)\right), \varphi \rangle\,\mathrm{d}t \right|\\&=:I^n_{1j}+I^n_{2j}+I^n_{3j}+I_{4j}.
	\end{align*}
	The convergence of $ \theta_n $ in $ \mathbb{X}_{\alpha,\beta} $ implies the boundedness of $ L^2\left( [0,T];H^{2\alpha-2\beta}(\mathbb{T}^2)\right)$-norm of $ \{\theta_n\} $. Combining this with  \eqref{high-frequency} for $r=2\beta-2\alpha$ and the fact that  $\left\|H_j \theta_n\right\|_{H^{2\alpha-2\beta}(\mathbb{T}^2)} \leq \left\| \theta_n\right\|_{H^{2\alpha-2\beta}(\mathbb{T}^2)}, $  we have
	\begin{equation}\label{estimate-i3}
		\begin{aligned}
			I^n_{3j}& \leq C(\varphi)  \int_0^T\left\|H_j \theta_n\right\|_{H^{2\alpha-2\beta}(\mathbb{T}^2)}\left\|H_j \theta_n\right\|_{H^{2\beta-2\alpha-1}(\mathbb{T}^2)} \,\mathrm{d}t \\
			& \lesssim C(\varphi) 2^{- j(4\alpha+1-4\beta)}  \int_0^T\left\|H_j \theta_n\right\|_{H^{2\alpha-2\beta}(\mathbb{T}^2)}^2 \,\mathrm{d} t \\
			& \lesssim C(\varphi) 2^{- j(4\alpha+1-4\beta)}  \int_0^T\left\|\theta_n\right\|_{H^{2\alpha-2\beta}(\mathbb{T}^2)}^2 \,\mathrm{d} t \lesssim C(\varphi) 2^{- j(4\alpha+1-4\beta)}. 
		\end{aligned}
	\end{equation}
	Applying the same argument as in $  I^n_{3j}$, we can see that, \begin{equation}\label{estimate-i4}
		\lim\limits_{j\rightarrow \infty}	I_{4j} =0.
	\end{equation} 
	According to  \eqref{estimate-i3} and \eqref{estimate-i4}, for any $ \eta>0,$  we can choose a sufficiently large positive integer $ j_0,$ such that $ I^n_{3j_0}+I_{4j_0} \leq \eta $ for all $ n\geq 1. $ By \cite[Lemma 9.3]{Mar08_existence}, $ S_{j_0}\theta_n\rightarrow S_{j_0}\theta$ and $R^\perp( S_{j_0}\theta_n)\rightarrow R^\perp(S_{j_0}\theta)$ strongly in $ L^{2}([0,T];H^{1/2}(\mathbb{T}^2))$, which implies
	\begin{equation}\label{estimate-i12}
		\lim\limits_{n\rightarrow \infty}	I^n_{1j_0} =\lim\limits_{n\rightarrow \infty}	I^n_{2j_0}=0
	\end{equation} since $\theta_n\rightarrow \theta$ weakly in $ L^2([0,T];H^{2\alpha-2\beta}(\mathbb{T}^2)).$
	Hence \begin{equation*}
		\limsup_{n\rightarrow\infty}\left(I ^n_{1j_0}+I^n_{2j_0}+I^n_{3j_0}+I_{4j_0}\right) \leq \eta.
	\end{equation*}
	By the arbitrariness of $ \eta, $ we conclude that
	\begin{align*}
		\lim_{n\rightarrow \infty}&\int_{0}^{T}\left\langle R_2 \theta_n,\left[\Lambda, \partial_1 \varphi\right]\Lambda^{-1} \theta_n\right\rangle -\left\langle R_1 \theta_n  ,\left[\Lambda, \partial_2 \varphi\right]\Lambda^{-1} \theta_n\right\rangle \mathrm{d}t\\&=\int_{0}^{T} \left\langle R_2 \theta,\left[\Lambda, \partial_1 \varphi\right]\Lambda^{-1} \theta\right\rangle -\left\langle R_1 \theta  ,\left[\Lambda, \partial_2 \varphi\right]\Lambda^{-1} \theta\right\rangle\mathrm{d}t.
	\end{align*}
\end{proof}

\section{Upper Bound for Large Deviations}\label{section-upper}
In this section, we proceed with the proof of the  large deviations for \eqref{SPDE-SQG}. Thanks to the exponential tightness (Proposition \ref{prop-exponential-tightness}), it suffices to prove the weak large deviations in the sense that the upper bound holds for all compact sets in $ \mathbb{X}_{\alpha,\beta} $.

For every $ \psi\in C_b(H_w^{\alpha-2\beta}(\mathbb{T}^2))$ and $ \varphi \in C^{\infty}([0,T];H^\infty(\mathbb{T}^2))$,  
we define  $\mathcal{M}^{\psi,\varphi }:[0,T]\times\mathbb{X}_{\alpha,\beta}\rightarrow\mathbb{R}$ by
\begin{align*}
	\mathcal{M}^{\psi,\varphi }(t,\theta):=	F^t(\varphi,\theta )+\Lambda_0(\psi,\theta(0)),
\end{align*}
where $F^t(\varphi,\theta )$ is defined by  \eqref{the-map-F} with $ T $ replaced by $ t $.
Let $Q^{\psi,\varphi }:[0,T]\times\mathbb{X}_{\alpha,\beta}\rightarrow\mathbb{R}$ be defined by
\begin{equation*}
	Q^{\psi,\varphi }(t,\theta):=\exp\Big\{\varepsilon^{-1}\big( 	\mathcal{M}^{\psi,\varphi }(t,\theta)-\frac{1}{2}\int^t_0\|\varphi\|_{H^{2\beta}(\mathbb{T}^2)}^2\,\mathrm{d}s\big) \Big\}.  
\end{equation*}
\begin{proposition}[Upper bound]\label{proposition-upperbound}
	For every $\varepsilon,\delta(\varepsilon)>0$ and $m(\varepsilon)\in\mathbb{N}_+$, let $\theta_{\varepsilon,\delta(\varepsilon)}$ be a stochastic generalized Leray solution of \eqref{SPDE-SQG} in the sense of Definition \ref{def-stochastic-leray} with initial data $ \theta_{\varepsilon, \delta(\varepsilon)}(0)$ satisfying Assumption \ref{Gaussian-initial-data},  and let $\theta_{\varepsilon,m(\varepsilon)}$ be the solution of \eqref{Galerkin approximation equations-delta=0} with initial data $ \theta_{\varepsilon, m(\varepsilon)}(0)\sim \mathcal{G}(0,\varepsilon P_{m(\varepsilon)}/2)$.  
	\item	(i) Let $\mu_{\varepsilon}=\mu_{\varepsilon,\delta(\varepsilon)}$ be the laws of $\theta_{\varepsilon,\delta(\varepsilon)}$ on $\mathbb{X}_{\alpha,\beta}$. Assume that the scaling regime \eqref{scaling-delta} holds for $(\varepsilon,\delta(\varepsilon))$, 
	then for any closed set $F\subset\mathbb{X}_{\alpha,\beta}$, 
	\begin{equation}\label{upper-bound}
		\limsup_{\varepsilon\rightarrow0}\varepsilon\log\mu_{\varepsilon}(F)\leq-\inf_{\theta\in F}\mathcal{I}(\theta). 	
	\end{equation}
	\item(ii)  Let $\mu_{\varepsilon}=\mu_{\varepsilon,m(\varepsilon)}$ be the laws of $\theta_{\varepsilon,m(\varepsilon)}$ on $\mathbb{X}_{\alpha,\beta}$. Assume that the scaling regime \eqref{scaling-m} holds for $(\varepsilon,m(\varepsilon))$, then for any closed set $F\subset\mathbb{X}_{\alpha,\beta}$, the upper bound \eqref{upper-bound} holds as well.
\end{proposition}

\begin{proof} We only prove (i), and (ii) could be handled in a similar way as (i). Due to Proposition \ref{prop-exponential-tightness}, it suffices to prove \eqref{upper-bound} for every compact $ F $ in $\mathbb{X}_{\alpha,\beta}  $.
	
	By Assumption \ref{Gaussian-initial-data}, Proposition \ref{prop-ldp-initial}, and Varadhan's integration lemma (see, for example, \cite[Theorem 4.3.1]{DZ10_LDP}), 
	\begin{equation}\label{varadhan}
		\begin{aligned}
			\limsup_{\varepsilon\rightarrow0}\varepsilon\log\int_{\mathbb{X}_{\alpha,\beta}}Q^{\psi,\varphi }(0,\theta)\mu_{\varepsilon}(\mathrm{d}\theta)&=	\limsup_{\varepsilon\rightarrow0}\varepsilon\log\int_{\mathbb{X}_{\alpha,\beta}}\exp\{\varepsilon^{-1}\Lambda_0(\psi,\theta(0))\}\mu_{\varepsilon}(\mathrm{d}\theta)\\&=\limsup_{\varepsilon\rightarrow0}\Big(\varepsilon\log\mathbb{E}\exp\{\varepsilon^{-1}\psi(\theta_{\varepsilon}(0))\}-\lambda(\psi)\Big)=0. 
		\end{aligned}
	\end{equation}
	Since $\theta_{\varepsilon,\delta(\varepsilon)}$ is a  solution of \eqref{SPDE-SQG} in the sense of Definition \ref{def-stochastic-leray} with initial data $ \theta_{\varepsilon, \delta(\varepsilon)}(0)$, the stochastic process
	$\mathcal{M}^{\psi,\varphi }(t,\theta_{\varepsilon,\delta(\varepsilon)})$ is a martingale. Due to the contraction property of $ Q_{\delta(\varepsilon)} $ on $ H^{2\beta}(\mathbb{T}^2)$,  the quadratic variation of $\mathcal{M}^{\psi,\varphi }(\cdot,\theta_{\varepsilon,\delta(\varepsilon)})$ satisfies
	\begin{equation*}
		\langle\langle\mathcal{M}^{\psi,\varphi }(\cdot,\theta_{\varepsilon,\delta(\varepsilon)})\rangle\rangle_t=\varepsilon\int_0^t\|\sqrt{Q_{\delta(\varepsilon)}}\Lambda^{2\beta}\varphi\|_{H}^2\,\mathrm{d}s\leq\varepsilon\int_0^t\|\varphi\|_{H^{2\beta}(\mathbb{T}^2)}^2\,\mathrm{d}s. 
	\end{equation*}
	It follows that $Q^{\psi,\varphi }(t,\theta_{\varepsilon,\delta(\varepsilon)})$ is a supermartingale. For every compact set $F\subset \mathbb{X}_{\alpha,\beta}$, we have 
	\begin{equation}\label{upper-bound-estimate-1}
		\begin{aligned}
			&\mu_{\varepsilon}(F)\leq\sup_{\theta\in F}(Q^{\psi,\varphi }(T,\theta))^{-1}\int_FQ^{\psi,\varphi }(T,\theta)\mu_{\varepsilon}(\mathrm{d}\theta)\\
			\leq&\sup_{\theta\in F}(Q^{\psi,\varphi }(T,\theta))^{-1}\int_{\mathbb{X}_{\alpha,\beta}}Q^{\psi,\varphi }(T,\theta)\mu_{\varepsilon}(\mathrm{d}\theta)\\
			\leq&\sup_{\theta\in F}(Q^{\psi,\varphi }(T,\theta))^{-1}\int_{\mathbb{X}_{\alpha,\beta}}Q^{\psi,\varphi }(0,\theta)\mu_{\varepsilon}(\mathrm{d}\theta)\\
			\leq&\exp\Big\{-\varepsilon^{-1}\inf_{\theta\in F}\Big(	\mathcal{M}^{\psi,\varphi }(T,\theta)-\frac{1}{2}\int^T_0\|\varphi\|_{H^{2\beta}(\mathbb{T}^2)}^2\,\mathrm{d}s\Big)\Big\}\int_{\mathbb{X}_{\alpha,\beta}}Q^{\psi,\varphi }(0,\theta)\mu_{\varepsilon}(\mathrm{d}\theta). 
		\end{aligned}
	\end{equation}
	In \eqref{upper-bound-estimate-1}, we used the non-negativity of  
	$ Q^{\psi,\varphi } $ in the first and second inequalities, and properties of supermartingales have been employed in the third inequality. The definition of $ Q^{\psi,\varphi } $ 
	has been used to derive the last inequality. 
	
	Taking the logarithmic function in \eqref{upper-bound-estimate-1} yields
	\begin{align*}
		\varepsilon\log\mu_{\varepsilon}(F)\leq&-\inf_{\theta\in F}\Big\{\Lambda^T_1(\varphi,\theta)+\Lambda_0(\psi,\theta(0))\Big\}+\varepsilon\log\int_{\mathbb{X}_{\alpha,\beta}}Q^{\psi,\varphi }(0,\theta)\mu_{\varepsilon}(\mathrm{d}\theta). 
	\end{align*}
	Letting $ \varepsilon\rightarrow 0 $ and using \eqref{varadhan}, we conclude that
	\begin{align*}
		\limsup_{\varepsilon\rightarrow0}\varepsilon\log\mu_{\varepsilon}(F)\leq&\inf_{\substack{\varphi\in C^{\infty}([0,T];H^\infty(\mathbb{T}^2))\\\psi\in C_b(H_w^{\alpha-2\beta}(\mathbb{T}^2))}}\Big(-\inf_{\theta\in F}\Big\{\Lambda_1^T(\varphi,\theta)+\Lambda_0(\psi,\theta(0))\Big\}\Big)\\
		=&-\sup_{\substack{\varphi\in C^{\infty}([0,T];H^\infty(\mathbb{T}^2))\\\psi\in C_b(H_w^{\alpha-2\beta}(\mathbb{T}^2))}}\inf_{\theta\in F}\Big\{\Lambda^T_1(\varphi,\theta)+\Lambda_0(\psi,\theta(0))\Big\}.
	\end{align*}
	Since $F$ is compact in $\mathbb{X}_{\alpha,\beta}$, by \cite[Appendix 2, Lemma 3.2]{KL99_IPS}, one can swap the above infimum and supremum. Finally, Proposition \ref{prop-variational-characterization-rate-function} implies that 
	\begin{align*}
		\limsup_{\varepsilon\rightarrow0}\varepsilon\log\mu_{\varepsilon}(F)\leq-\inf_{\theta\in F}\mathcal{I}(\theta).
	\end{align*}
\end{proof}

\section{Lower Bound for Large Deviations}\label{section-lower}
In this section, we will prove a restricted large deviations lower bound using the entropy method as described in \cite{Mariani10} (see Lemma \ref{lemma-entropy-method}). Due to the lack of well-posedness of the skeleton equation \eqref{PDE-SQG}, it is challenging to match the upper and lower bounds in $\mathbb{X}_{\alpha,\beta}$. Instead, we will match the upper and lower bounds within the $ \mathcal{I}$-closure of the weak-strong uniqueness class defined in Definition \ref{def-weak-strong-uniqueness-intro}. 

Let $\mathcal{C}_0$ be a weak-strong uniqueness class of \eqref{PDE-SQG} and let  $\mathcal{C}$ be its $\mathcal{I}$-closure.  We will take $E=\mathbb{X}_{\alpha,\beta}$ and $I=\mathcal{I}|_{\mathcal{C}_0}$ in Lemma \ref{lemma-entropy-method} to obtain the lower bound of large deviations. Moreover, when utilizing Lemma \ref{lemma-entropy-method}, we focus only on the case $ \theta\in \mathcal{C}_0 \subset \mathbb{X}_{\alpha,\beta}$ with $ \mathcal{I}(\theta)< \infty $. Otherwise, the condition (ii) in Lemma \ref{lemma-entropy-method} holds obviously, and one can take $\pi^{\varepsilon, \theta}=\delta_\theta$ to verify the condition (i).

\begin{proposition}[Restricted lower bound]\label{prop-restricted-lower-bound}
	For every $\varepsilon,\delta(\varepsilon)>0$ and $m(\varepsilon)\in\mathbb{N}_+$, let $\theta_{\varepsilon,\delta(\varepsilon)}$ be a stochastic generalized Leray solution of \eqref{SPDE-SQG} in the sense of Definition \ref{def-stochastic-leray} with initial data $ \theta_{\varepsilon, \delta(\varepsilon)}(0)$ satisfying Assumption \ref{Gaussian-initial-data},  and let $\theta_{\varepsilon,m(\varepsilon)}$ be the solution of \eqref{Galerkin approximation equations-delta=0} with initial data $ \theta_{\varepsilon, m(\varepsilon)}(0)\sim \mathcal{G}(0,\varepsilon P_{m(\varepsilon)}/2)$. Let $ \mathcal{C}_0\subset \mathbb{X}_{\alpha,\beta} $ be a weak-strong uniqueness class of \eqref{PDE-SQG} and let $\mathcal{C}=\overline{\mathcal{C}_0}^{\mathcal{I}}$. 	
	
	\item	(i) Let $\mu_{\varepsilon}=\mu_{\varepsilon,\delta(\varepsilon)}$ be the law of $\theta_{\varepsilon,\delta(\varepsilon)}$ on $\mathbb{X}_{\alpha,\beta}$. Assume that the scaling regime \eqref{scaling-delta} holds for $(\varepsilon,\delta(\varepsilon))$, 
	then for any open set $G\subset\mathbb{X}_{\alpha,\beta}$, 
	\begin{equation}\label{lower-bound}
		\liminf_{\varepsilon\rightarrow0}\varepsilon\log\mu_{\varepsilon}(G)\geq-\inf_{\theta\in G\cap \mathcal{C}}\mathcal{I}(
		\theta). 	
	\end{equation}
	\item(ii)  Let $\mu_{\varepsilon}=\mu_{\varepsilon,m(\varepsilon)}$ be the law of $\theta_{\varepsilon,m(\varepsilon)}$ on $\mathbb{X}_{\alpha,\beta}$. Assume that the scaling regime \eqref{scaling-m} holds for $(\varepsilon,m(\varepsilon))$, then for any open set $G\subset\mathbb{X}_{\alpha,\beta}$, the restricted lower bound \eqref{lower-bound} holds as well.
	
\end{proposition} 

\begin{lemma}\label{lemma-entropy-estimate} 	For every $\theta\in \mathcal{C}_0$ with $ \mathcal{I}(\theta)< \infty $, let $ g\in L^2([0,T];H)$ be the control such that $ \theta $ solves \eqref{PDE-SQG} weakly.
	For every $ \varepsilon>0 $ and $m(\varepsilon)\in\mathbb{N}_+$ satisfying the  scaling regime \eqref{scaling-m}, let $\theta_{\varepsilon,m(\varepsilon)}$ and $\mu_{\varepsilon,m(\varepsilon)}$ be as in Proposition \ref{prop-restricted-lower-bound}. Then there exists a family of probability measure  $\{\pi_{\varepsilon,m(\varepsilon), \theta}\}$ on $ \mathbb{X}_{\alpha,\beta} $ that satisfies the entropy inequality
	\begin{equation}\label{entropy-estimate}
		\limsup _{\varepsilon\rightarrow 0} \operatorname{Ent}\left(\pi_{\varepsilon,m(\varepsilon), \theta} \mid \mu_{\varepsilon,m(\varepsilon)}\right) \leq \mathcal{I}(\theta).
	\end{equation}
	
\end{lemma}

\begin{proof}
	Let $\left(\Omega, \mathcal{F},\{\mathcal{F}(t)\}_{t \in[0, T]}, \overline{\mathbb{P}}\right)$ be a stochastic basis and $ \overline{W}$ be a cylindrical Wiener process on $ H^{\alpha-2\beta}(\mathbb{T}^2)$ such that  
	$\bar{\theta}_{\varepsilon,m(\varepsilon)}$ is a solution of \eqref{Galerkin approximation equations-delta=0} with initial data $\bar{\theta}_{\varepsilon,m(\varepsilon)}(0) \sim\mathcal{G}(0,\varepsilon P_{m(\varepsilon)}/2)$  and $\bar{\theta}_{\varepsilon,m(\varepsilon)}(0)$ is independent of $\overline{W}$. For any two measures  $ \mu, \nu $ on $ \Omega$, we denote by $\frac{\mathrm{d}\mu}{\mathrm{d}\nu}$ the Radon-Nikodym derivative of $ \mu $ with respect to $ \nu. $
	Now we define the $ H $-valued random variable $Y_0^{\varepsilon,m(\varepsilon)} $ by
	$$
	Y_0^{\varepsilon,m(\varepsilon)} :=\frac{\mathrm{d}\mathcal{G}(P_{m(\varepsilon)}\theta(0),\varepsilon P_{m(\varepsilon)}/2)}{\mathrm{d} \mathcal{G}(0,\varepsilon P_{m(\varepsilon)}/2)}\left(\bar{\theta}_{\varepsilon,m(\varepsilon)}(0)\right),
	$$
	and define the stochastic process $Z^{\varepsilon,m(\varepsilon)}$ by
	\begin{equation}\label{def-of-Z}
		Z^{\varepsilon,m(\varepsilon)}(t):=\exp \left(-\frac{1}{\sqrt{\varepsilon}} \int_0^t\left\langle P_{m(\varepsilon)}g(s), \,\mathrm{d}\overline{W}(s)\right\rangle-\frac{1}{2 \varepsilon} \int_0^t\left\|P_{m(\varepsilon)} g(s)\right\|_{H}^2 \,\mathrm{d} s\right) Y_0^{\varepsilon,m(\varepsilon)}.
	\end{equation}
	Let $ \mathbb{P} $ be the probability measure on $\left(\Omega, \mathcal{F}\right) $ such that $Z^{\varepsilon,m(\varepsilon)}(T)=\frac{\mathrm{d} \mathbb{P}}{\mathrm{d}\overline{\mathbb{P}}}.$
	By Cameron-Martin theorem,   $W(t):=\overline{W}(t)+\varepsilon^{-1 / 2}\int_0^t P_{m(\varepsilon)}g(s)\,\mathrm{d}s$ is a cylindrical Wiener process on $\left(\Omega, \mathcal{F},\mathbb{P}\right)$ and
	$\bar{\theta}_{\varepsilon,m(\varepsilon)}$ is a solution of 
	\begin{equation}\label{stochastic-control-equation}
		\begin{aligned}
			\mathrm{d} \bar{\theta}_{\varepsilon,m(\varepsilon)}(t)=&-\Lambda^{2\alpha} \bar{\theta}_{\varepsilon,m(\varepsilon)}(t)\,\mathrm{d} t-P_m\left(u_{\bar\theta_{\varepsilon,m(\varepsilon)}}(t)\cdot \nabla \bar{\theta}_{\varepsilon,m(\varepsilon)}(t)\right)\,\mathrm{d} t \\&+\sqrt{\varepsilon} \Lambda^{2\beta} \,\mathrm{d} W^{m(\varepsilon)}(t)+P_{m(\varepsilon)}\Lambda^{2\beta}g\,\mathrm{d}t
		\end{aligned}
	\end{equation}
	with the initial data $\bar{\theta}_{\varepsilon,m(\varepsilon)}(0)\sim\mathcal{G}(P_{m(\varepsilon)}\theta(0),\varepsilon P_{m(\varepsilon)}/2).$  
	
	Let $ \pi_{\varepsilon,m(\varepsilon),\theta}  $ be the law of $\bar{\theta}_{\varepsilon,m(\varepsilon)}$ on  $\left(\Omega, \mathcal{F},\mathbb{P}\right)$. 
	According to \eqref{def-of-Z}, the definition of relative entropy, and the expression of  $ \mathcal{I}_0 $, it holds that
	$$
	\begin{aligned}
		\varepsilon \operatorname{Ent}\left(\pi_{\varepsilon,m(\varepsilon),\theta} | \mu_{\varepsilon,m(\varepsilon)}\right)& \leq \varepsilon \operatorname{Ent}(\mathbb{P} |\overline{\mathbb{P}})  =\varepsilon \mathbb{E}\left[\log Z^{\varepsilon,m(\varepsilon)}(T)\right] \\
		& =\frac{1}{2} \int_0^T\|P_{m(\varepsilon)}g(s)\|_{H}^2 \,\mathrm{d} s+\varepsilon \mathbb{E}\left[\log Y_0^{\varepsilon,m(\varepsilon)}\right]\\&\leqslant \frac{1}{2} \int_0^T\|g(s)\|_{H}^2 \,\mathrm{d}s+\mathcal{I}_0(\theta(0))=\mathcal{I}(\theta).
	\end{aligned}
	$$
	\
\end{proof}
\begin{lemma}\label{lemma-converge-to-delta-v}
	For every $ \varepsilon>0 $ and $m(\varepsilon)\in\mathbb{N}_+$ satisfying the  scaling regime \eqref{scaling-m}, 	 let $\theta,\, g,\,\pi^{\varepsilon,m(\varepsilon),\theta}$ and $ \mu_{\varepsilon,m(\varepsilon)}$ be as in Lemma \ref{lemma-entropy-estimate}. Then  $\pi^{\varepsilon,m(\varepsilon), \theta}$ converges to $\delta_\theta$ in law on $\mathbb{X}_{\alpha,\beta}$ as $ \varepsilon\rightarrow 0$.
\end{lemma}

\begin{proof}
	The proof will be divided into three steps. First, we show that $ \{\pi^{\varepsilon,m(\varepsilon),\theta}\} $ is tight in $\mathbb{X}_{\alpha,\beta}$ and that any limit point $\bar{\theta}$ is a weak solution of \eqref{PDE-SQG} in the sense of Definition \ref{def-weak-solution-skeleton-equation}. Next, we prove the $ H^{\alpha-2\beta}$-energy inequality for $\bar{\theta}$. Finally, we conclude that $ \bar{\theta}=\theta $ in $\mathbb{X}_{\alpha,\beta}$ by using the weak-strong uniqueness.  
	
	\textbf{Step 1.  $\bar{\theta}$ is a weak solution of the skeleton equation.}
	Tightness of $\{\bar{\theta}_{\varepsilon,m(\varepsilon)}\}$ in $ \mathbb{X}_{\alpha,\beta} $ follows by the same method as in the proof of Proposition \ref{prop-exponential-tightness}, thus we omit the proof. 
	According to  Jakubowski-Skorokhod representation theorem \cite{Jak97},  for any sequence $\varepsilon_k \rightarrow 0$, there exists a subsequence  $(\varepsilon_k,m_k(\varepsilon))$ (still denoted by $ (\varepsilon,m(\varepsilon)) $), a new stochastic basis (still denoted by $(\Omega, \mathcal{F},\left\{\mathcal{F}_t\right\}_{t \in[0,T]}, \mathbb{P})$), a cylindrical Wiener process (still denoted by $ W $), a sequence of  $\mathbb{X}_{\alpha,\beta}$-valued random elements with the same laws as $\{\bar{\theta}_{\varepsilon,m(\varepsilon)}\}$ on the new probability space (still denoted by $\{\bar{\theta}_{\varepsilon,m(\varepsilon)}\}$),  and a limit point $\bar{\theta}$,  such that  $\mathbb{P}$-almost surely, $
	\bar{\theta}_{\varepsilon,m(\varepsilon)} \rightarrow \bar{\theta} \, \text { in } \mathbb{X}_{\alpha,\beta}$.
	We claim that $ \bar{\theta} $ is a weak solution of \eqref{PDE-SQG}.
	For any $\varphi \in C^{\infty}\left([0, T]; H^\infty(\mathbb{T}^2)\right)$  and $ t\in[0,T]$, we will pass to the limits in all terms of 
	\begin{equation}\label{test-function-bar-theta-m}
		\begin{aligned}
			&\left\langle \bar{\theta}_{\varepsilon,m(\varepsilon)}(t), \varphi(t)\right\rangle+\int_0^t\left\langle \bar{\theta}_{\varepsilon,m(\varepsilon)}, \Lambda^{2\alpha} \varphi\right\rangle \,\mathrm{d}s=\left\langle  \bar{\theta}_{\varepsilon,m(\varepsilon)}(0),\varphi(0)\right\rangle+\int_0^t\langle \bar{\theta}_{\varepsilon,m(\varepsilon)},\partial_s\varphi \rangle\,\mathrm{d}s\\&-\frac{1}{2}\int^t_0\left\langle R_2 \bar{\theta}_{\varepsilon,m(\varepsilon)},\left[\Lambda, \partial_1 \varphi\right]\Lambda^{-1} \bar{\theta}_{\varepsilon,m(\varepsilon)}\right\rangle \,\mathrm{d} s+\frac{1}{2}\int_0^t\left\langle R_1 \bar{\theta}_{\varepsilon,m(\varepsilon)} ,\left[\Lambda, \partial_2 \varphi\right]\Lambda^{-1} \bar{\theta}_{\varepsilon,m(\varepsilon)}\right\rangle \,\mathrm{d}s\\&+\sqrt{\varepsilon}\int_{0}^t \langle \Lambda^{2\beta}\varphi,P_{m(\varepsilon)}\,\mathrm{d}W(s)\rangle+\int_0^t\left\langle P_{m(\varepsilon)}g, \Lambda^{2\beta} \varphi\right\rangle \,\mathrm{d}s.
		\end{aligned}
	\end{equation}
	It follows from the proof of Lemma \ref{lemma-lsc} that, $\mathbb{P}$-almost surely,  for all $ t\in[0,T], $ as $ \varepsilon\rightarrow 0, $  the first six terms in \eqref{test-function-bar-theta-m} converge.
	For the stochastic integral term, applying Doob's $L^2$-inequality, we obtain
	$$
	\mathbb{E} \left( \sup _{t \in[0,T]} \left|\sqrt{\varepsilon}\int_{0}^t \langle \Lambda^{2\beta}\varphi,P_{m(\varepsilon)}\,\mathrm{d}W(s)\rangle\right|^2 \right)  \leq 4\varepsilon \mathbb{E} \left|\int_{0}^{T} \langle \Lambda^{2\beta}\varphi,P_{m(\varepsilon)}\,\mathrm{d}W(s)\rangle\right|^2.
	$$ 
	According to  It\^o isometry,
	$$\mathbb{E} \left|\int_{0}^{T} \langle \Lambda^{2\beta}\varphi,P_{m(\varepsilon)}\,\mathrm{d}W(s)\rangle\right|^2\leq \|\Lambda^{2\beta}\varphi\|_{L^2({0,T};H)}.$$
	Hence we can pass to a further subsequence so that $\sqrt{\varepsilon}\int_{0}^t \langle \Lambda^{2\beta}\varphi,P_{m(\varepsilon)}\,\mathrm{d}W(s)\rangle \rightarrow 0$, almost surely, for all $t\in [0,T]$.   Finally, the convergence of the last term follows from the fact that $ g\in L^2([0, T];H)$ and H\"older's inequality.
	Therefore we conclude that, $ \mathbb{P}$-almost surely, for all $t\in[0,T]$,
	\begin{equation*}
		\begin{aligned}
			&\langle \bar{\theta}(t), \varphi(t)\rangle+\int_0^t\left\langle \bar{\theta}, \Lambda^{2\alpha} \varphi\right\rangle \,\mathrm{d}s=\left\langle  \bar{\theta}(0),\varphi(0)\right\rangle+\int_0^t\langle \bar{\theta},\partial_s\varphi \rangle\,\mathrm{d}s\\&-\frac{1}{2}\int^t_0\left\langle R_2 \bar{\theta},\left[\Lambda, \partial_1 \varphi\right]\Lambda^{-1} \bar{\theta}\right\rangle \,\mathrm{d} s+\frac{1}{2}\int_0^t\left\langle R_1 \bar{\theta} ,\left[\Lambda, \partial_2 \varphi\right]\Lambda^{-1} \bar{\theta}\right\rangle \,\mathrm{d}s+\int_0^t\left\langle g, \Lambda^{2\beta} \varphi\right\rangle \,\mathrm{d}s.
		\end{aligned}
	\end{equation*}
	
	\textbf{Step 2. $ \bar{\theta} $ satisfies $ H^{\alpha-2\beta}$-energy inequality.}
	Applying It\^o's formula to	$\bar{\theta}_{\varepsilon,m(\varepsilon)}$, we deduce that  $\mathbb{P}$-almost surely, for all $t\in[0,T]$,
	\begin{equation}\label{energy-ineq-galerkin}
		\begin{aligned}
			&\quad\frac{1}{2}\left\|\bar{\theta}_{\varepsilon,m(\varepsilon)}(t)\right\|_{H^{\alpha-2\beta}(\mathbb{T}^2)}^2+\int_0^t\left\|\bar{\theta}_{\varepsilon,m(\varepsilon)}(s)\right\|_{H^{2\alpha-2\beta}(\mathbb{T}^2)}^2 \,\mathrm{d}s \\&= \frac{1}{2}\left\|\bar{\theta}_{\varepsilon,m(\varepsilon)}(0)\right\|_{H^{\alpha-2\beta}(\mathbb{T}^2)}^2 +\sqrt{\varepsilon} \int_0^t\left\langle  \Lambda^{2\alpha-2\beta} \bar{\theta}_{\varepsilon,m(\varepsilon)}(s),  P_{m(\varepsilon)}\mathrm{d}W(s)\right\rangle\\&+\frac{\varepsilon}{2}\left\|P_{m(\varepsilon)}\Lambda^{\alpha}\right\|_{HS}^2t
			+\int_0^t\left\langle\Lambda^{2\alpha-2\beta} \bar{\theta}_{\varepsilon,m(\varepsilon)}(s),P_{m(\varepsilon)} g(s)\right\rangle\,\mathrm{d}s.
		\end{aligned}
	\end{equation}
	Letting $ \varepsilon\rightarrow 0 $ in \eqref{energy-ineq-galerkin}, it follows from the fact that $
	\bar{\theta}_{\varepsilon,m(\varepsilon)} \rightarrow \bar{\theta} \, \text { in } \mathbb{X}_{\alpha,\beta}
	$ and the lower semi-continuity of the $ H^s(\mathbb{T}^2) $ norms that
	$$
	\begin{aligned}
		&\quad\frac{1}{2}\|\bar{\theta}(t)\|_{H^{\alpha-2\beta}(\mathbb{T}^2)}^2+\int_0^t\|\bar{\theta}(s)\|_{H^{2\alpha-2\beta}(\mathbb{T}^2)}^2 \,\mathrm{d} s \\&\leq \liminf _{\varepsilon \rightarrow 0} 	\frac{1}{2}\left\|\bar{\theta}_{\varepsilon,m(\varepsilon)}(t)\right\|_{H^{\alpha-2\beta}(\mathbb{T}^2)}^2+\int_0^t\left\|\bar{\theta}_{\varepsilon,m(\varepsilon)}(s)\right\|_{H^{2\alpha-2\beta}(\mathbb{T}^2)}^2 \,\mathrm{d}s .
	\end{aligned}
	$$					
	Let $M^{\varepsilon,m(\varepsilon)}_2(t):=\int_{0}^t \sqrt{\varepsilon}\left\langle \Lambda^{2\alpha-2\beta}\bar{\theta}_{\varepsilon,m(\varepsilon)}(s),P_{m(\varepsilon)}  \,\mathrm{d}W(s)\right\rangle$.  Applying Doob's $L^2$-inequality, we obtain
	$$
	\mathbb{E} \left( \sup _{t \leq T} \left|M_2^{\varepsilon,m(\varepsilon)}(t)\right|^2 \right) \leq 4 \mathbb{E}\left\langle\langle M_2^{\varepsilon,m(\varepsilon)}\right\rangle\rangle_T \leq 4\varepsilon \mathbb{E} \int_0^T\left\|\bar{\theta}_{\varepsilon,m(\varepsilon)}(s)\right\|_{H^{2\alpha-2\beta}(\mathbb{T}^2)}^2 \,\mathrm{d}s \rightarrow 0.
	$$ Hence we can pass to a further subsequence so that $M_2^{\varepsilon,m(\varepsilon)}(t) \rightarrow 0$ almost surely.  
	By the scaling regime \eqref{scaling-m}, the correction term $\frac{\varepsilon}{2}\left\|P_{m(\varepsilon)}\Lambda^{\alpha}\right\|_{HS}^2t \rightarrow 0$. 
	Finally,  according to the weak convergence $ \Lambda^{2\alpha-2\beta}\bar{\theta}_{\varepsilon,m(\varepsilon)} \rightarrow \Lambda^{2\alpha-2\beta} \bar{\theta}$ in $ L^2([0,T];H)$ and the strong convergence $ P_{m(\varepsilon)}g \rightarrow g$ in $ L^2([0,T];H)$, it holds that 
	$$
	\left|\int_0^t\left\langle\Lambda^{2\alpha-2\beta} \bar{\theta}_{\varepsilon,m(\varepsilon)}, P_{m(\varepsilon)} g\right\rangle \,\mathrm{d}s-\int_0^t\left\langle\Lambda^{2\alpha-2\beta} \bar{\theta},  g\right\rangle \,\mathrm{d}s\right| \rightarrow 0 .
	$$
	Therefore we conclude that, for all $t\in[0,T]$, \begin{equation*}
		\frac{1}{2}\left\|\bar{\theta}(t)\right\|_{H^{\alpha-2\beta}(\mathbb{T}^2)}^2+\int_0^t\left\|\bar{\theta}(s)\right\|_{H^{2\alpha-2\beta}(\mathbb{T}^2)}^2 \,\mathrm{d}s \leq \frac{1}{2}\left\|\bar{\theta}(0)\right\|_{H^{\alpha-2\beta}(\mathbb{T}^2)}^2 
		+\int_0^t\left\langle\Lambda^{2\alpha-2\beta} \bar{\theta},  g\right\rangle \,\mathrm{d}s.
	\end{equation*}
	
	\textbf{Step 3. $ \bar{\theta}=\theta $ in $\mathbb{X}_{\alpha,\beta}. $ }
	Thanks to the choice $\theta \in \mathcal{C}_0$, it follows from Definition \ref{def-weak-strong-uniqueness-intro} that $\theta=\bar{\theta}$ in $\mathbb{X}_{\alpha,\beta}$.  Hence $\pi^{\varepsilon,m, \theta}$ converges to $\delta_\theta$ in law on $\mathbb{X}_{\alpha,\beta}$ as $ \varepsilon\rightarrow 0$.
\end{proof}
We can now prove the large deviations lower bound restricted to $\mathcal{C}_0$.
\begin{lemma}[Lower bound on $\mathcal{C}_0$]
	
	For every $\varepsilon,\delta(\varepsilon)>0$ and $m(\varepsilon)\in\mathbb{N}_+$, let $\theta_{\varepsilon,\delta(\varepsilon)}$ be a stochastic generalized Leray solution of \eqref{SPDE-SQG} in the sense of Definition \ref{def-stochastic-leray} with initial data $ \theta_{\varepsilon, \delta(\varepsilon)}(0)$ satisfying Assumption \ref{Gaussian-initial-data},  and let $\theta_{\varepsilon,m(\varepsilon)}$ be the solution of \eqref{Galerkin approximation equations-delta=0} with initial data $ \theta_{\varepsilon, m(\varepsilon)}(0)\sim \mathcal{G}(0,\varepsilon P_{m(\varepsilon)}/2)$. Suppose that $ \mathcal{C}_0\subset \mathbb{X}_{\alpha,\beta} $ is a weak-strong uniqueness class of \eqref{PDE-SQG}.
	
	\item	(i) Let $\mu_{\varepsilon}=\mu_{\varepsilon,\delta(\varepsilon)}$ be the law of $\theta_{\varepsilon,\delta(\varepsilon)}$ on $\mathbb{X}_{\alpha,\beta}$. Assume that the scaling regime \eqref{scaling-delta} holds for $(\varepsilon,\delta(\varepsilon))$, 
	then for any open set $G\subset\mathbb{X}_{\alpha,\beta}$, 
	\begin{equation}\label{lower-bound-0}
		\liminf_{\varepsilon\rightarrow0}\varepsilon\log\mu_{\varepsilon}(G)\geq-\inf_{\theta\in G\cap \mathcal{C}_0}\mathcal{I}(
		\theta). 	
	\end{equation}
	\item(ii)  Let $\mu_{\varepsilon}=\mu_{\varepsilon,m(\varepsilon)}$ be the law of $\theta_{\varepsilon,m(\varepsilon)}$ on $\mathbb{X}_{\alpha,\beta}$. Assume that the scaling regime \eqref{scaling-m} holds for $(\varepsilon,m(\varepsilon))$, then for any open set $G\subset\mathbb{X}_{\alpha,\beta}$, the restricted lower bound \eqref{lower-bound-0} holds as well.						
\end{lemma}
\begin{proof}
	(ii) is a direct consequence of Lemma \ref{lemma-entropy-estimate}, Lemma \ref{lemma-converge-to-delta-v}, and Lemma \ref{lemma-entropy-method}.  And (i) can be proved in much the same way, the only difference being in the analysis of the pathwise $H^{\alpha-2\beta}$-energy inequality. More precisely, in the compactness analysis, we need to replace the It\^o's formula \eqref{energy-ineq-galerkin} by the $ H^{\alpha-2\beta} $-energy inequality in Definition \ref{def-stochastic-leray}.
\end{proof}

Now we come to the proof of the restricted lower bound on $  \mathcal{C}= \overline{\mathcal{C}_0}^{\mathcal{I}}$.

\textit{Proof of Proposition \ref{prop-restricted-lower-bound}}.
For any open set $G\subset \mathbb{X}_{\alpha,\beta}$ and any $\eta>0$, there exists an element $\theta^0 \in G \cap \mathcal{C}$ such that	
\begin{equation}\label{restricted-lower-bound-estimate1}
	\mathcal{I}(\theta^0) \leq \inf _{\theta \in G\cap \mathcal{C}} \mathcal{I}(\theta)+\frac{\eta}{2}.
\end{equation}
By the definition of $  \mathcal{C} $,  there exists a sequence $\{\theta^{(n)}\}_{n\geq1} \subset \mathcal{C}_0$, such that $\theta^{(n)} \rightarrow \theta^0$ in  $\mathbb{X}_{\alpha,\beta}$ and $\mathcal{I}\left(\theta^{(n)}\right) \rightarrow \mathcal{I}(\theta^0)$ as $ n\rightarrow\infty $. It  follows that
\begin{equation}\label{restricted-lower-bound-estimate2}
	\inf_{\theta \in G\cap \mathcal{C}_0}\mathcal{I}(\theta)\leq \mathcal{I}(\theta^0)+\frac{\eta}{2}.
\end{equation}
Using \eqref{restricted-lower-bound-estimate1}, \eqref{restricted-lower-bound-estimate2}, and the arbitrariness of $ \eta $, we obtain that
$$
\inf_{\theta \in G\cap \mathcal{C}_0}\mathcal{I}(\theta)= \inf_{\theta \in G\cap \mathcal{C}}\mathcal{I}(\theta).
$$
Hence \eqref{lower-bound} follows from \eqref{lower-bound-0}.

\

\section{Probabilistic Approaches to the Energy Equality}\label{section-energy-eq}	
In this section, we are devoted to exploring the relationship between the deterministic energy equality and large deviations. For any $ \alpha, \beta $ satisfying \eqref{pama-range}, 
we introduce the time-reversal operator $\mathfrak{T}_T: \mathbb{X}_{\alpha,\beta}\rightarrow\mathbb{X}_{\alpha,\beta}$, defined by $ \mathfrak{T}_T\theta:=-\theta(T-\cdot) $. 
\begin{theorem}\label{theorem-energy-equality}
	Assume further that $ \beta=\alpha/2.$	Let $\mathcal{C}_0\subset \mathbb{X}_{\alpha,\alpha/2}$ be a weak-strong uniqueness class of \eqref{PDE-SQG} such that $\mathcal{C}=\overline{\mathcal{C}_0}^{\mathcal{I}}$ contains non-empty time-reversible subsets. Let $\mathcal{R} \subset \mathcal{C}$ satisfy  $\mathcal{R}=\mathfrak{T}_T \mathcal{R}.$  Suppose that $ \theta\in \mathcal{R} $ is  a weak solution of \eqref{PDE-SQG} for some $g \in L^2([0, T];H)$ in the sense of Definition \ref{def-weak-solution-skeleton-equation}. Then the kinetic energy equality holds:
	\begin{equation}
		\frac{1}{2}\|\theta(T)\|_H^2+\int_0^T\|\theta(s)\|_{H^\alpha(\mathbb{T}^2)}^2\,\mathrm{d}s=\frac{1}{2}\|\theta(0)\|_H^2+\int_0^T\langle\Lambda^{\alpha}\theta, g\rangle\,\mathrm{d}s.
	\end{equation}
\end{theorem} Theorem \ref{theorem-energy-equality} establishes a connection between the weak-strong uniqueness regularity class and the kinetic energy equality. However, the proof of this purely analytic result is based on large deviations for \eqref{Galerkin approximation equations-delta=0}. The key argument relies on the time-reversibility property, and the proof will be provided later. Beyond time-reversibility, we emphasize that a Hamiltonian energy equality still holds in the case $ \beta=\alpha/2+1/4$. This leads to the following strengthening of Theorem \ref{theorem-energy-equality}: 
\begin{customthm}\label{theorem-generalized-energy-equality}
	Let $\mathcal{C}_0\subset \mathbb{X}_{\alpha,\beta}$ be a weak-strong uniqueness class of \eqref{PDE-SQG} such that $\mathcal{C}=\overline{\mathcal{C}_0}^{\mathcal{I}}$ contains non-empty time-reversible subsets. Let $\mathcal{R} \subset \mathcal{C}$ satisfy  $\mathcal{R}=\mathfrak{T}_T \mathcal{R}.$  Suppose that $ \theta\in \mathcal{R} $ is  a weak solution of \eqref{PDE-SQG} for some $g \in L^2([0, T];H)$ in the sense of Definition \ref{def-weak-solution-skeleton-equation}. Then the $ H^{\alpha-2\beta}$-energy equality holds:
	\begin{equation}\label{generlized-energy-eq}
		\frac{1}{2}\|\theta(T)\|_{H^{\alpha-2\beta}(\mathbb{T}^2)}^2+\int_0^T\|\theta(s)\|_{H^{2\alpha-2\beta}(\mathbb{T}^2)}^2\,\mathrm{d}s=\frac{1}{2}\|\theta(0)\|_{H^{\alpha-2\beta}(\mathbb{T}^2)}^2+\int_0^T\langle\Lambda^{2\alpha-2\beta}\theta, g\rangle\,\mathrm{d}s.
	\end{equation}
\end{customthm}
The probabilistic method breaks down when $ \beta\neq \alpha/2. $ Instead, we will provide an analytic proof for this more general case, following a similar approach to the one used in \cite[Theorem 2.7]{GHW24_LLNS}. 

The remainder of this section will be dedicated to proving Theorem \ref{theorem-energy-equality} and Theorem \ref{theorem-generalized-energy-equality}.
\begin{lemma}\label{time-revesible-property}
	For any  $ \varepsilon>0 $ and $ m(\varepsilon)\in \mathbb{N}_+$, suppose that $ \theta_{\varepsilon, m} $ is the solution of \eqref{Galerkin approximation equations-delta=0} with  $ \beta=\alpha/2 $ and $ \theta_{\varepsilon, m(\varepsilon)}(0)\sim \mathcal{G}(0,\varepsilon P_{m(\varepsilon)}/2)$. Then the law of $ \theta_{\varepsilon, m(\varepsilon)} $ is invariant under $\mathfrak{T}$:
	$$
	\mathfrak{T}_T \theta_{\varepsilon, m(\varepsilon)}:=\left(-\theta_{\varepsilon, m(\varepsilon)}(T-t): 0 \leq t \leq T\right)\overset{\,\mathrm{d}}{=}\left(\theta_{\varepsilon, m(\varepsilon)}(t): 0 \leq t \leq T\right) .
	$$
\end{lemma}

\begin{proof}
	The proof is elementary. See \cite[Proposition 4.1]{Tot20} for details.
\end{proof}

\begin{proof}[Proof of Theorem \ref{theorem-energy-equality}] 
	We begin by proving that $ \mathcal{I}(\mathfrak{T}_T\theta)=\mathcal{I}(\theta). $ 
	Suppose that $ \varepsilon,m(\varepsilon) $ and $ \theta_{\varepsilon, m(\varepsilon)} $ satisfy the conditions in  Lemma \ref{time-revesible-property}. Let $ \mu_{\varepsilon, m(\varepsilon)} $ be the law of $ \theta_{\varepsilon, m(\varepsilon)} $ on $\mathbb{X}_{\alpha,\alpha/2}$ and let $ \tilde{\mu}_{\varepsilon,m(\varepsilon)} :=\mu_{\varepsilon,m(\varepsilon)} \circ \mathfrak{T}_T^{-1}$.  
	For any  $ \eta>0, $ it follows from Lemma \ref{lemma-lsc} that there exists an open set $ G$ in $ \mathbb{X}_{\alpha,\alpha/2} $ such that $  \theta\in G $ and $ \mathcal{I}(\hat\theta)> \mathcal{I}(\mathfrak{T}_T\theta)-\eta$ for all $ \hat\theta\in \mathfrak{T}_T\overline{G}.$ Moreover, the continuity of $ \mathfrak{T}_T:\mathbb{X}_{\alpha,\alpha/2}\rightarrow\mathbb{X}_{\alpha,\alpha/2} $ implies that $ \mathfrak{T}_T^{-1}\overline{G}=\mathfrak{T}_T\overline{G} $ is closed in $ \mathbb{X}_{\alpha,\alpha/2}. $
	According to Proposition \ref{prop-restricted-lower-bound},  Lemma \ref{time-revesible-property}, and  Proposition \ref{upper-bound}, we have  \begin{equation*}
		\begin{aligned}
			-\mathcal{I}(\theta)\leq\liminf _{\varepsilon \rightarrow 0} \varepsilon &\log \mu_{\varepsilon,m(\varepsilon)}(G)=\liminf _{\varepsilon \rightarrow 0} \varepsilon \log \tilde{\mu}_{\varepsilon,m(\varepsilon)}(G)\leq\limsup _{\varepsilon \rightarrow 0} \varepsilon \log \tilde{\mu}_{\varepsilon,m(\varepsilon)}(\overline{G})\\&=\limsup _{\varepsilon \rightarrow 0} \varepsilon \log \mu_{\varepsilon,m(\varepsilon)} (\mathfrak{T}_T\overline{G})\leq -\inf_{ \hat\theta\in \mathfrak{T}_T\overline{G}}\mathcal{I}( \hat\theta)\leq-\mathcal{I}(\mathfrak{T}_T\theta)+\eta.
		\end{aligned}
	\end{equation*}
	Due to the arbitrariness of $ \eta, $ we deduce that $ \mathcal{I}(\mathfrak{T}_T\theta)\leq\mathcal{I}(\theta). $ Since  $\mathcal{R}=\mathfrak{T}_T \mathcal{R}$, we can replace $ \theta $ by $\mathfrak{T}_T\theta$ to obtain that $ \mathcal{I}(\mathfrak{T}_T\theta)=\mathcal{I}(\theta). $
	
	Thanks to Remark \ref{uniqueness-of-g},  $ g $ is the unique element in $ L^2([0,T];H)$ such that $ \theta $ is a weak solution of \eqref{PDE-SQG} for $ g$. Thus
	\begin{equation}\label{eq-I-theta}
		\mathcal{I}(\theta)=\|\theta(0)\|^2_{H}+\frac{1}{2}\big\|g\big\|^2_{L^2([0,T];H)}.
	\end{equation}
	Now we define $\tilde{g}: [0,T]\times\mathbb{T}^2\rightarrow \mathbb{R}$ by $ \tilde{g}(t,x):=g(T-t,x)-2\Lambda^{\alpha}(\mathfrak{T}_T\theta)$. We have $ \tilde{g}\in L^2([0,T];H)$ since $ \theta\in \mathcal{R}\subset \mathbb{X}_{\alpha,\alpha/2} .$
	Moreover, a direct calculation shows that $ \tilde{\theta} $ is a weak solution of \eqref{PDE-SQG} with $ \tilde{g}$ in the sense of Definition \ref{def-weak-solution-skeleton-equation}. 
	Hence
	\begin{equation}\label{eq-I-theta-reversible}
		\begin{aligned}
			\mathcal{I}( \mathfrak{T}_T\theta)&=\|\theta(T)\|^2_{H}+\frac{1}{2}\big\|\tilde{g}\big\|^2_{L^2([0,T];H)}\\&=\|\theta(T)\|^2_{H}+\frac{1}{2}\big\|g\big\|^2_{L^2([0,T];H)}+2\int_0^T\|\theta(s)\|_{H^\alpha(\mathbb{T}^2)}^2-2\int_{0}^{T}\langle\Lambda^{\alpha}\theta,g\rangle\,\mathrm{d}s.
		\end{aligned}
	\end{equation}
	Consequently, in combination with $	\mathcal{I}( \mathfrak{T}_T\theta)=\mathcal{I}(\theta)$, \eqref{eq-I-theta} and \eqref{eq-I-theta-reversible} lead to the kinetic energy equality 
	\begin{equation*}
		\frac{1}{2}\|\theta(T)\|_H^2+\int_0^T\|\theta(s)\|_{H^\alpha(\mathbb{T}^2)}^2\,\mathrm{d}s=\frac{1}{2}\|\theta(0)\|_H^2+\int_0^T\langle\Lambda^{\alpha}\theta, g\rangle\,\mathrm{d}s.
	\end{equation*}
\end{proof}
\begin{proof}[Proof of Corollary \ref{Main-Theorem-corollary}] 
	The proof follows from an argument analogous to \cite[Remark 2.8]{GHW24_LLNS}, which is similar in spirit to the probabilistic proof of Theorem \ref{theorem-energy-equality}. Thus we omit it.
\end{proof}
The following lemma will be utilized in the proof of Theorem \ref{theorem-generalized-energy-equality}.
\begin{lemma}\label{lemma-energy-ineq-on-C}
	Let $ \mathcal{C}_0\subset \mathbb{X}_{\alpha,\beta} $ be a weak-strong uniqueness class of \eqref{PDE-SQG} and let $\mathcal{C}=\overline{\mathcal{C}_0}^{\mathcal{I}}$. Suppose that $ \theta\in \mathcal{C} $ is  a weak solution of \eqref{PDE-SQG} for some $g \in L^2([0, T];H)$ in the sense of Definition \ref{def-weak-solution-skeleton-equation}. Then the $ H^{\alpha-2\beta} $-energy inequality holds:
	\begin{equation*}
		\frac{1}{2}\|\theta(T)\|_{H^{\alpha-2\beta}(\mathbb{T}^2)}^2+\int_0^T\|\theta(s)\|_{H^{2\alpha-2\beta}(\mathbb{T}^2)}^2\,\mathrm{d}s\leq\frac{1}{2}\|\theta(0)\|_{H^{\alpha-2\beta}(\mathbb{T}^2)}^2+\int_0^T\langle\Lambda^{2\alpha-2\beta}\theta, g\rangle\,\mathrm{d}s.
	\end{equation*}
\end{lemma}
\begin{proof}
	By the definition of $\mathcal{C}$, there exists a sequence $\{\theta^{(n)}\}\subset \mathcal{C}_0$, such that $\theta^{(n)} \rightarrow \theta $ in $ \mathbb{X}_{\alpha,\beta}$ and $ \mathcal{I}\big(\theta^{(n)}\big) \rightarrow \mathcal{I}(\theta)$. Let $g^{(n)} \in L^2([0,T];H) $ be the control such that $\theta^{(n)} $ solves \eqref{PDE-SQG} weakly. We claim that $ g^{(n)} \rightarrow g $ in $L^2([0,T];H) $. In fact, $ \mathcal{I}\big(\theta^{(n)}\big) \rightarrow \mathcal{I}(\theta)$ implies that $ \{  g^{(n)} \} $ is uniformly  bounded in $L^2([0,T];H), $ hence weakly compact. For any $\varphi\in C^{\infty}([0,T];H^\infty(\mathbb{T}^2))$, let  $\Lambda_1^T(\varphi,\cdot), 	F^T(\varphi,\cdot)$ be defined  as in Section \ref{section-rate}. Suppose that $ g_0 $ is a limit point of  $ \{  g^{(n)} \} $ with respect to the weak topology. Then it follows from \eqref{convergence-lambda-1} that we can let $ n\rightarrow\infty $ in \eqref{riesz} with $ \Psi^\theta $ replaced by $  \Psi^{\theta^{(n)}} $ to obtain that \begin{equation*}
		F^T(\varphi,\theta)=\langle g_0,\Lambda^{2\beta}\varphi\rangle_{L^2([0,T];H)}.
	\end{equation*} That is, $ \theta $ is a weak solution of \eqref{PDE-SQG} for the control $ g_0. $
	We deduce from Remark \ref{uniqueness-of-g} that $ g=g_0 $ in  $L^2([0,T];H)$. Therefore, $ g^{(n)} \rightarrow g$ weakly in $L^2([0,T];H)$. Since $ \mathcal{I}\big(\theta^{(n)}\big) \rightarrow \mathcal{I}(\theta)$ implies  $ \|g^{(n)}\|_{L^2([0,T];H)}\rightarrow \|g\|_{L^2([0,T];H)},$ we can conclude that $ g^{(n)} \rightarrow g $ (strongly) in $L^2([0,T];H). $
	
	As in the proof of Lemma \ref{lemma-converge-to-delta-v}, $ \theta^{(n)}\in\mathcal{C}_0 $ satisfies the $ H^{\alpha-2\beta} $-energy inequality:
	\begin{equation*}
	\frac{1}{2}\|\theta^{(n)}(T)\|_{H^{\alpha-2\beta}(\mathbb{T}^2)}^2+\int_0^T\|\theta^{(n)}(s)\|_{H^{2\alpha-2\beta}(\mathbb{T}^2)}^2\,\mathrm{d}s\leq\frac{1}{2}\|\theta^{(n)}(0)\|_{H^{\alpha-2\beta}(\mathbb{T}^2)}^2+\int_0^T\langle\Lambda^{2\alpha-2\beta}\theta^{(n)}, g^{(n)}\rangle\,\mathrm{d}s.
	\end{equation*} Letting $ n\rightarrow\infty, $ due to the strong convergence of $g^{(n)}$, a repetition of Step 2  in the proof of  Lemma \ref{lemma-converge-to-delta-v} yields that 	\begin{equation*}
		\frac{1}{2}\|\theta(T)\|_{H^{\alpha-2\beta}(\mathbb{T}^2)}^2+\int_0^T\|\theta(s)\|_{H^{2\alpha-2\beta}(\mathbb{T}^2)}^2\,\mathrm{d}s\leq\frac{1}{2}\|\theta(0)\|_{H^{\alpha-2\beta}(\mathbb{T}^2)}^2+\int_0^T\langle\Lambda^{2\alpha-2\beta}\theta, g\rangle\,\mathrm{d}s.
	\end{equation*}
\end{proof}
\noindent\textit{Proof of Theorem \ref{theorem-generalized-energy-equality}.}
We are now in the setting of Lemma \ref{lemma-energy-ineq-on-C}. It follows  that 	\begin{equation}\label{generalize-energy-ineq-1}
	\frac{1}{2}\|\theta(T)\|_{H^{\alpha-2\beta}(\mathbb{T}^2)}^2+\int_0^T\|\theta(s)\|_{H^{2\alpha-2\beta}(\mathbb{T}^2)}^2\,\mathrm{d}s\leq\frac{1}{2}\|\theta(0)\|_{H^{\alpha-2\beta}(\mathbb{T}^2)}^2+\int_0^T\langle\Lambda^{2\alpha-2\beta}\theta, g\rangle\,\mathrm{d}s.
\end{equation}
Recall that the map $ \tilde{g}$ defined by $ \tilde{g}(t,x):=g(T-t,x)-2\Lambda^{\alpha}(\mathfrak{T}_T\theta)$ is an element of 
$ L^2([0,T];H)$ and $ \mathfrak{T}_T\theta $ is a weak solution of \eqref{PDE-SQG} for $ \tilde{g}$. Since $ \mathfrak{T}_T\theta\in\mathcal{R}\subset\mathcal{C}$, we can 
apply Lemma \ref{lemma-energy-ineq-on-C}  for $ \mathfrak{T}_T\theta $ and $ \tilde{g}$ to obtain 
\begin{equation}\label{generlized-energy-ineq-2}
	\begin{aligned}
		&\quad\frac{1}{2}\|\theta(0)\|_{H^{\alpha-2\beta}(\mathbb{T}^2)}^2+\int_0^T\|\theta(s)\|_{H^{2\alpha-2\beta}(\mathbb{T}^2)}^2\,\mathrm{d}s\leq\frac{1}{2}\|\theta(T)\|_{H^{\alpha-2\beta}(\mathbb{T}^2)}^2+\int_0^T\langle\Lambda^{2\alpha-2\beta}\mathfrak{T}_T\theta, \tilde g\rangle\,\mathrm{d}s\\&=\frac{1}{2}\|\theta(T)\|_{H^{\alpha-2\beta}(\mathbb{T}^2)}^2-\int_0^T\langle\Lambda^{2\alpha-2\beta}\theta,  g\rangle\,\mathrm{d}s+2\int_0^T\|\theta(s)\|_{H^{2\alpha-2\beta}(\mathbb{T}^2)}^2\,\mathrm{d}s.
	\end{aligned}
\end{equation}
Combining \eqref{generalize-energy-ineq-1} with \eqref{generlized-energy-ineq-2}, we conclude that \eqref{generlized-energy-eq} holds.

\

\section{Characterization of Quasi-potential}\label{section-quasipotential}
In this section, we will analyze the explicit representation of the quasi-potential $ \mathcal{U}. $ This potentially indicates a new probabilistic perspective of the uniqueness problem for the deterministic SQG equation.   

\subsection{Conditional equivalence}
To begin with, we need to extend the definition of the weak solution to infinite time  intervals.
\begin{definition}\label{def-sol-infinity}
	Let $g\in L_{loc}^2([0,\infty);H)$. We say that $\theta$ is a weak solution of  
	\begin{equation}\label{PDE-infty}
		\partial_t\theta=-\Lambda^{2\alpha}\theta-R^{\perp}\theta\cdot\nabla \theta+\Lambda^{2\beta} g,
	\end{equation}
	if $\theta\in L^\infty_{loc}([0,\infty);H^{\alpha-2\beta}(\mathbb{T}^2)) \cap L^2_{loc}([0,\infty);H^{2\alpha-2\beta}(\mathbb{T}^2))\cap C([0,\infty);H_w^{\alpha-2\beta}(\mathbb{T}^2)),$ and for every $\varphi\in  C^\infty([0,\infty);H^\infty(\mathbb{T}^2)), $  for all $ t>0, $ 
	\begin{equation}\label{weak-solution-infinity}
		\begin{aligned}
			\langle \theta(t),\varphi(t)\rangle&=\langle \theta(0),\varphi(0)\rangle-\int^t_0\langle \theta,\Lambda^{2\alpha}\varphi\rangle  \,\mathrm{d}s-\frac{1}{2}\int^t_0\left\langle R_2 \theta,\left[\Lambda, \partial_1 \varphi\right]\Lambda^{-1} \theta\right\rangle \,\mathrm{d}s\\&+\frac{1}{2}\int_0^t\left\langle R_1 \theta  ,\left[\Lambda, \partial_2 \varphi\right]\Lambda^{-1} \theta\right\rangle \,\mathrm{d}s+\int^t_0\langle \theta,\partial_s\varphi\rangle \,\mathrm{d}s+\int^t_0\langle\Lambda^{2\beta}\varphi,g\rangle \,\mathrm{d}s.
		\end{aligned}
	\end{equation}
\end{definition}
With the above definition, we can similarly define the weak-strong uniqueness class on $ [0,\infty) $
as in Definition \ref{def-weak-strong-uniqueness-intro}. Let  $$ \mathbb{X}^\infty_{\alpha,\beta}:= L^2_{loc}([0,\infty);H^{\alpha-2\beta}(\mathbb{T}^2)) \cap	L^2_{w,loc}([0,\infty);H^{2\alpha-2\beta}(\mathbb{T}^2)) \cap C([0,\infty);H^{\alpha-2\beta}_w(\mathbb{T}^2)).$$

\begin{definition}[Weak-Strong Uniqueness Class on $ [0,\infty) $]\label{def-weak-strong-uniqueness-infty}
	We say that
	$\mathcal{C}_0 \subset \mathbb{X}^\infty_{\alpha,\beta}$ is a weak-strong uniqueness class of \eqref{PDE-infty}  if for every control $g\in L^2_{loc}([0,\infty);H)$, the following holds:   for arbitrary two weak solutions $\theta_1,\theta_2$ of the skeleton equation \eqref{PDE-infty} in the sense of Definition \ref{def-sol-infinity}  with the same initial data $\theta_1(0)=\theta_2(0)$ and the same control $ g $, we have $\theta_1=\theta_2$ in  $ \mathbb{X}_{\alpha,\beta}^\infty$ as long as $\theta_1\in\mathcal{C}_0$ and  $\theta_2$ satisfies the $ H^{\alpha-2\beta}$-energy inequality: for every $ t>0 $,
	\begin{equation}
		\frac{1}{2}\|\theta_2(t)\|_{H^{\alpha-2\beta}(\mathbb{T}^2)}^2+\int_0^t\|\theta_2(s)\|_{H^{2\alpha-2\beta}(\mathbb{T}^2)}^2 \,\mathrm{d}s \leq \frac{1}{2}\left\|\theta_2(0)\right\|_{H^{\alpha-2\beta}(\mathbb{T}^2)}^2+\int_0^t\langle\Lambda^{2\alpha-2\beta} \theta_2, g\rangle \,\mathrm{d}s.
	\end{equation}
\end{definition}
Let $ \mathcal{C}_0 $ be a  weak-strong uniqueness class on $ [0,\infty).$ Denote \begin{equation*}
	\begin{aligned}
		\mathcal{A}_0:=\big\lbrace \theta\in\mathcal{C}_0: 	\partial_t\theta\in L^2([0,\infty) ; H^{\alpha-2\beta}(\mathbb{T}^2)), R^{\perp}\theta\cdot\nabla \theta\in L^2([0,\infty) ;  H^{\alpha-2\beta}(\mathbb{T}^2))&, \\ \lim\limits_{t\rightarrow \infty}\|\theta(t)\|_{H^{\alpha-2\beta}(\mathbb{T}^2)}=0& \big\rbrace.
	\end{aligned}
\end{equation*}		
For any $ \tilde{\theta}:(-\infty,0]\times\mathbb{T}^2\rightarrow\mathbb{R}$ that satisfying $ -\tilde{\theta}(-\cdot)\in \mathcal{A}_0$, we define
\begin{equation*}
	S_{-\infty}(\tilde\theta):=\frac{1}{2}\int_{-\infty}^{0}\|	\partial_t\tilde\theta(s)+\Lambda^{2\alpha}\tilde\theta(s)+R^{\perp}\tilde\theta(s)\cdot\nabla\tilde \theta(s)\|_{H^{-2\beta}(\mathbb{T}^2)}^2\,\mathrm{d}s,
\end{equation*}
and define the map  $ \mathcal{U}: H^{\alpha-2\beta}(\mathbb{T}^2)\rightarrow [0,\infty)$  as 
\begin{equation}\label{restrictquasipotential}
	\mathcal{U}(\phi):=\inf \left\{S_{-\infty}(\tilde{\theta}): -\tilde{\theta}(-\cdot)\in \mathcal{A}_0,\, \tilde\theta(0)=\phi\right\}.
\end{equation}

\begin{proposition}\label{quasi-Gaussian}
	Let $ \mathcal{C}_0,  \mathcal{A}_0,  \mathcal{U} $ and $ S_{-\infty} $ be defined as above. 
	\item (i)  For every $\phi\in H^{\alpha-2\beta}(\mathbb{T}^2)$, it holds that $\mathcal{U}(\phi)\geq\|\phi\|_{H^{\alpha-2\beta}(\mathbb{T}^2)}^2$.
	\item (ii)
	The equality $\mathcal{U}(\phi)=\|\phi\|_{H^{\alpha-2\beta}(\mathbb{T}^2)}^2$ holds for a given $\phi\in H^{\alpha-2\beta}(\mathbb{T}^2)$ as long as the equation
	\begin{equation}\label{PDE-tilde-theta}
		\left\{\begin{array}{l}
			\partial_t\bar\theta=-\Lambda^{2\alpha} \bar\theta-R^{\perp}\bar\theta\cdot\nabla \bar\theta, \\
			\bar\theta(0)=\phi,
		\end{array}\right.
	\end{equation}
	admits a unique weak solution in $ \mathcal{A} _0 $ in the sense of Definition \ref{def-sol-infinity}.
\end{proposition}
\begin{proof}
	(i). 
	It is sufficient to discuss the case that $\mathcal{U}(\phi)<\infty. $ In this case, there exists a function  $\tilde\theta:(-\infty,0]\times\mathbb{T}^2\rightarrow\mathbb{R}$ such that $-\tilde\theta(-\cdot)\in\mathcal{A}_0$, $ \tilde\theta(0)=\phi$ and $ \langle R^{\perp}\tilde\theta(s)\cdot\nabla \tilde\theta(s), \tilde\theta(s)\rangle_{H^{\alpha-2\beta}(\mathbb{T}^2)}=0 $ for all $ s<0. $  Then it follows from the definition of $\mathcal{A}_0$ that 
	\begin{equation*}
		\begin{aligned}
			&\left\|\Lambda^{-2\beta}(	\partial_t\tilde\theta(s)+\Lambda^{2\alpha} \tilde\theta(s)+R^{\perp}\tilde\theta(s)\cdot\nabla \tilde\theta(s))\right\|_{H}^2\\=&\left\|\Lambda^{-2\beta}(	\partial_t\tilde\theta(s)-\Lambda^{2\alpha} \tilde\theta(s)+R^{\perp}\tilde\theta(s)\cdot\nabla \tilde\theta(s))\right\|_{H}^2
			+4\|\Lambda^{2\alpha-2\beta} \tilde\theta(s)\|_{H}^2\\&+4\left\langle\Lambda^{-2\beta}(	\partial_t\tilde\theta(s)-\Lambda^{2\alpha} \tilde\theta(s)+R^{\perp}\tilde\theta(s)\cdot\nabla \tilde\theta(s)),  \Lambda^{2\alpha-2\beta}\tilde\theta(s)\right\rangle \\=&\left\|\Lambda^{-2\beta}(	\partial_t\tilde\theta(s)-\Lambda^{2\alpha} \tilde\theta(s)+R^{\perp}\tilde\theta(s)\cdot\nabla \tilde\theta(s))\right\|_H^2+4\left\langle \Lambda^{\alpha-2\beta} 	\partial_t\tilde\theta(s),\Lambda^{\alpha-2\beta}\tilde\theta(s)\right\rangle.
		\end{aligned}
	\end{equation*}
	Therefore, by the definition of $ 
	S_{-\infty}$, we have
	\begin{align}\label{S-infty-}
		&\quad	S_{-\infty}(\tilde\theta) \notag\\&=\frac{1}{2} \int_{-\infty}^0\left\|\Lambda^{-2\beta}(	\partial_t\tilde\theta(s)-\Lambda^{2\alpha} \tilde\theta(s)+R^{\perp}\tilde\theta(s)\cdot\nabla \tilde\theta(s))\right\|_{H}^2 \,\mathrm{d}s+\int_{-\infty}^0 \frac{\mathrm{d}}{\mathrm{d} s}\|\tilde\theta(s)\|_{H^{\alpha-2\beta}(\mathbb{T}^2)}^2 \,\mathrm{d}s \notag\\
		& =\frac{1}{2} \int_{-\infty}^0\left\|\Lambda^{-2\beta}(	\partial_t\tilde\theta(s)-\Lambda^{2\alpha} \tilde\theta(s)+R^{\perp}\tilde\theta(s)\cdot\nabla \tilde\theta(s))\right\|_{H}^2 \,\mathrm{d}s+\|\tilde\theta(0)\|_{H^{\alpha-2\beta}(\mathbb{T}^2)}^2 .
	\end{align}			
	Thus $ \tilde\theta(0)=\phi$ implies $\mathcal{U}(\phi)\geq\|\phi\|_{H^{\alpha-2\beta}(\mathbb{T}^2)}^2$.
	
	(ii). Now we aim to prove that the equality $\mathcal{U}(\phi)=\|\phi\|_{H^{\alpha-2\beta}(\mathbb{T}^2)}^2$ holds as long as \eqref{PDE-tilde-theta} admits a weak solution in $ \mathcal{A}_0. $ Assume that  $ \bar{\theta} $ is  the solution.
	Then 
	$	\tilde{\theta}(\cdot):=-\bar\theta(-\cdot)$
	satisfies $	\partial_t\tilde\theta(s)-\Lambda^{2\alpha} \tilde{\theta}(s)+R^{\perp}\tilde{\theta}(s)\cdot\nabla \tilde{\theta}(s) =0$ for all $ s\leq 0 $ and hence the first term on the right-hand side of  \eqref{S-infty-} is $0$. This completes the proof.
\end{proof}
\begin{remark}\label{rmk-A0}
	We provide further remarks on the choice of $\mathcal{A}_0$. It is evident that such a conditional equivalence does not rely on the concrete concept of 'weak-strong uniqueness' in the definition of $\mathcal{A}_0$. In other words, the conditional equivalence between the restricted quasi-potential and the Gaussian rate function depends only on the existence of solutions of \eqref{PDE-tilde-theta} within the restriction class of the quasi-potential. Therefore, the class $\mathcal{A}_0$ can be replaced by other regularity classes.
	
	Nevertheless, to demonstrate the relationships between the non-Gaussian large deviations (see Definition \ref{Asymptotic non-Gaussian equilibrium feature} below) in the equilibrium of the stochastic PDE \eqref{SPDE-SQG} and the uniqueness problem of the PDE \eqref{PDE-tilde-theta}, one needs to specify the restriction class of the quasi-potential to govern the equilibrium-large deviations rate. As indicated in \cite[Proposition 4.1]{CP22}, this restriction class depends on where the upper and lower bounds of the uniform dynamical large deviations match. To invoke the dynamical large deviations result in this paper, one expects to define $\mathcal{A}_0$ with trajectories lying in the $\mathcal{I}$-closure class $\mathcal{C}$, which is the domain where the upper and lower bounds match.
	
	However, the extension of the lower bound from the restriction class $\mathcal{C}_0$ to $\mathcal{C}$ technically relies on the choice of the random initial data. If  the random initial data was replaced by a deterministic initial data, it is unclear whether the extension of the lower bound holds on $\mathcal{C}$, as well as the uniform dynamical large deviations lower bound. Therefore, we specify that the trajectories of $\mathcal{A}_0$ lie only in the smaller set $\mathcal{C}_0$ instead of $\mathcal{C}$.			
\end{remark}

\subsection{Potential relationships to open problems for PDEs}

To demonstrate relationships to open problems for PDEs, we introduce the concept of non-Gaussian large deviations in equilibrium for stochastic PDE \eqref{SPDE-SQG}.
\begin{definition}[Non-Gaussian large deviations]\label{Asymptotic non-Gaussian equilibrium feature} 
	Let  $\{\nu_{\varepsilon}\}_{\varepsilon>0}$ be a family of  probability measures on $H_w^{\alpha-2\beta}(\mathbb{T}^2)$.	We say that  $\{\nu_{\varepsilon}\}_{\varepsilon>0}$ satisfies non-Gaussian  large deviations on $H_w^{\alpha-2\beta}(\mathbb{T}^2)$, if it satisfies large deviations  on $H_w^{\alpha-2\beta}(\mathbb{T}^2)$ with speed $ \varepsilon^{-1} $ and  rate function $ I_0(\cdot)\neq \|\cdot\|^2_{H^{\alpha-2\beta}(\mathbb{T}^2)}.$
\end{definition}

Let $\{\nu_{\varepsilon,\delta(\varepsilon)}\}_{\varepsilon>0}$ be a family of ergodic invariant measures of \eqref{SPDE-SQG}. Inspired by \cite{BC17} and \cite{CP22}, the large deviations for $\{\nu_{\varepsilon,\delta(\varepsilon)}\}_{\varepsilon>0}$ under a scaling regime $(\varepsilon,\delta(\varepsilon))\rightarrow(0,0)$ strongly rely on uniform dynamical large deviations, and its rate function is expected to be governed by the quasi-potential. Based on the discussion in Remark \ref{rmk-A0}, one can only expect to observe the rate function for the lower bound governed by the restricted quasi-potential \eqref{restrictquasipotential}. This connects to the open problem of uniqueness for the SQG equation \eqref{PDE-tilde-theta}.

As long as one can prove large deviations for $\{\nu_{\varepsilon,\delta(\varepsilon)}\}_{\varepsilon>0}$ with rate function \eqref{restrictquasipotential}, then Proposition \ref{quasi-Gaussian} indicates the following relationships. On the one hand, if a solution of \eqref{PDE-tilde-theta} exists in $\mathcal{A}_0$ (which implies uniqueness), then, by applying Proposition \ref{quasi-Gaussian}, one can observe that the expected rate function \eqref{restrictquasipotential} equals to the Gaussian rate function $\|\cdot\|^2_{H^{\alpha-2\beta}(\mathbb{T}^2)} $. Conversely, if non-Gaussian large deviations can be proven for $\{\nu_{\varepsilon,\delta(\varepsilon)}\}_{\varepsilon>0}$ in the sense of Definition \ref{Asymptotic non-Gaussian equilibrium feature}, the uniqueness problem of \eqref{PDE-tilde-theta} could be answered negatively. Therefore, Proposition \ref{quasi-Gaussian} suggests a new perspective for understanding the uniqueness problem of the PDE \eqref{PDE-tilde-theta} through probabilistic approaches. However, rigorously establishing this relationship requires proving large deviations of the invariant measure, which lies beyond the scope of this paper. 

\begin{remark}
	
	It is worth pointing out that, in \cite{Tot20}, the existence of Gaussian stationary solutions for \eqref{SPDE-SQG} with $\delta=0$ is proved using the energy solutions approach. However, this does not contradict the discussion of non-Gaussian large deviations for invariant measures in the sense of Definition \ref{Asymptotic non-Gaussian equilibrium feature}, as the aforementioned non-Gaussian large deviations focus on fluctuations in spaces with higher regularity due to the regularization of the noise and the choice of the scaling regime $(\varepsilon, \delta(\varepsilon))$.
	
	On the other hand, in \cite{HLZZ24}, the authors proved there exist infinitely many non-Gaussian ergodic stationary solutions of \eqref{SPDE-SQG} with $ \delta=0$ by showing the $ L^\infty(\Omega\times[0,\infty)) $-boundedness of the solutions with respect to some Besov norm. In their work, the parameters $ \alpha $ and $\beta$ satisfy $ \alpha\in[0,3/4) $ and $ \beta\in[0,1/8)$,  and the solutions lie in the Besov space $ B^{-1/2}_{p,1} $ for any $ p\geq 2. $ We note that the non-Gaussian large deviations defined above differs from the non-Gaussianity in \cite{HLZZ24}. Definition \ref{Asymptotic non-Gaussian equilibrium feature} only captures the asymptotic behavior of the measures by analyzing the rate function, demonstrating that it differs from the Gaussian measure $ \mathcal{G}\left(0, \varepsilon Q_{\delta(\varepsilon)} / 2\right) $ on $H^{\alpha-2\beta}(\mathbb{T}^2)$. However, it does not exclude the possibility that this family of measures are Gaussian distribution concentrated on a different space. We hope to bridge the discussion between non-Gaussian ergodic stationary solutions and the uniqueness problem of the PDE \eqref{PDE-tilde-theta} via the non-Gaussian large deviations.
	
\end{remark}
\

\appendix
\renewcommand{\appendixname}{Appendix~\Alph{section}}
\renewcommand{\theequation}{A.\arabic{equation}}
\section{Examples of the Weak-strong Uniqueness Classes}\label{section-example} 

In this section, we provide some examples of the weak-strong uniqueness classes $\mathcal{C}_0$ and  $ \mathcal{R}\subset\overline{\mathcal{C}_0}^{\mathcal{I}}$ for different choices of $\alpha$ and $\beta$ satisfying \eqref{pama-range}. These examples cover subcritical, critical, and supercritical cases.

\begin{Example}
	Suppose that $1/2<\alpha < 1$ and $ \beta=\alpha/2+1/4$.
	Let  $$\mathcal{C}_0^1:= \{\theta \in \mathbb{X}_{\alpha,\alpha/2+1/4}:\theta\in L^r\left([0,T] ; L^{p}(\mathbb{T}^2)\right), \, \frac{1}{p}+\frac{\alpha}{r}=\alpha-\frac{1}{2}, \, p\geq 1,\, r>0 \}.$$
	For every $g \in L^2([0, T];H)$, assume that $\theta_1 \in \mathcal{C}_0^1$ and  $\theta_2\in \mathbb{X}_{\alpha,\alpha/2+1/4}$ solve the skeleton equation \eqref{PDE-SQG} in the sense of Definition \ref{def-weak-solution-skeleton-equation} with the same initial data $\theta_1 (0)=\theta_2(0) \in H^{-1/2}(\mathbb{T}^2)$, and suppose further that $\theta_2$  satisfy the $H^{-1/2}$-energy inequality: for every $t \in[0, T]$,
	\begin{equation}\label{energy inequality for theta2}
		\frac{1}{2}\|\theta_2(t)\|_{H^{-1/2}(\mathbb{T}^2)}^2+\int_0^t\|\theta_2(s)\|_{H^{\alpha-1/2}(\mathbb{T}^2)}^2 \,\mathrm{d}s \leqslant \frac{1}{2}\left\|\theta_2(0)\right\|_{H^{-1/2}(\mathbb{T}^2)}^2+\int_0^t\langle\Lambda^{\alpha-1/2}\theta_2, g\rangle \,\mathrm{d}s.
	\end{equation}
	Then $\theta_1 =\theta_2$ in $ \mathbb{X}_{\alpha,\alpha/2+1/4}$.
	
\end{Example}
\begin{proof}
	The same conclusion holds in the case where $ \beta=\alpha/2$. We refer readers to \cite[Theorem 2.2, Remark 2.3]{CW99_Weak_strong} for the proof.
\end{proof}
In the next example, we will take $ \mathcal{R}^1:=  \mathbb{X}_{\alpha,\alpha/2+1/4} \cap L^4\left([0,T] ; L^{4}(\mathbb{T}^2)\right)$ and show that  $ \mathcal{R}^1$ is contained in the $ \mathcal{I}$-closure of $ \mathcal{C}_0^1. $ \begin{Example}\label{eg-C-subcritical}
	Let $\theta \in \mathcal{R}^1$ with $\mathcal{I}(\theta)<\infty$. Then there exists a sequence $\{\theta^{(n)} \}_{n\geq1}\subset \mathcal{C}_0^1$ such that $\theta^{(n)} \rightarrow \theta$ in  $\mathbb{X}_{\alpha,\alpha/2+1/4}$ and $\mathcal{I}\left(\theta^{(n)}\right) \rightarrow \mathcal{I}(\theta)$ as $n\rightarrow\infty$. 
\end{Example}
\begin{proof}
	Let $\{\eta_\epsilon\}_{\epsilon\in(0,1)}$ be the standard convolution kernel on $ \mathbb{T}^2.$ For any real-valued function $ f $ and every $ n\geq 1$ , we define $f^{(n)}$ by $f^{(n)}:=f\ast\eta_{1/n}$.  
	Now for $\theta^{(n)}:=\theta\ast\eta_{1/n}$, it suffices to prove
	
	(i) $ \theta^{(n)}\in \mathcal{C}_0^1 $,\quad
	(ii) $ \theta^{(n)}\rightarrow \theta $ in $ \mathbb{X}_{\alpha,\alpha/2+1/4} $ as $n\rightarrow \infty, $\quad
	(iii) $\mathcal{I}\left(\theta^{(n)}\right) \rightarrow \mathcal{I}(\theta)$.
	
	From the property of convolution kernel $\eta_\epsilon$, we deduce that (i) and (ii) hold. Now we proceed to prove (iii). By the lower semi-continuity of $\mathcal{I}$, it suffices to show that $\limsup_{n\rightarrow\infty}\mathcal{I}\left(\theta^{(n)}\right) \leq \mathcal{I}(\theta)$. 
	
	The fact that $ \theta^{(n)}(0)\rightarrow\theta(0) $ in $H^{-1/2}(\mathbb{T}^2)$ implies $\mathcal{I}_0\left(\theta^{(n)}(0)\right) \rightarrow \mathcal{I}_0(\theta(0))$. For the dynamic cost, thanks to Proposition \ref{prop-variational-characterization-rate-function}, there exists an element $g \in L^2([0, T];H)$ such that
	\begin{equation*}
		\partial_t \theta=-\Lambda^{2\alpha} \theta-(u_\theta \cdot \nabla \theta)+\Lambda^{\alpha+1/2} g
	\end{equation*}
	holds in the sense of \eqref{weak-solution-skeleton-equation}. 
	Taking the convolution with $\eta_{1/n}$, it follows that
	\begin{equation*}
		\partial_t \theta^{(n)}=-\Lambda^{2\alpha} \theta^{(n)}-(u_{\theta^{(n)}} \cdot \nabla \theta^{(n)})+\Lambda^{\alpha+1/2} g^{(n)}+\Lambda^{\alpha+1/2}R^{(n)},
	\end{equation*}
	where \begin{equation*}
		R^{(n)}:=\Lambda^{-\alpha-1/2}(u_{\theta^{(n)}} \cdot \nabla \theta^{(n)}-(u_\theta \cdot \nabla \theta)^{(n)})=\Lambda^{-\alpha-1/2}\nabla\cdot\left(u_{\theta^{(n)}}\theta^{(n)}-(u_\theta \theta)^{(n)} \right).
	\end{equation*}
	We will show that $\left\|R^{(n)}\right\|_{L^2([0, T]; H) } \rightarrow 0$ as $n\rightarrow \infty$. 
	
	Using  Cauchy-Schwarz inequality, Calder\'on-Zygmund inequality,  and the fact that $ \theta\in \mathcal{R}^1, $
	we have $u_\theta\theta\in L^2\left([0,T] ; L^{2}(\mathbb{T}^2;\mathbb{R}^2)\right)$. 
	Since $ \alpha>1/2 $ and $ (u_\theta\theta)^{(n)}\rightarrow u_\theta\theta $ in $L^2\left([0,T] ; L^{2}(\mathbb{T}^2;\mathbb{R}^2)\right) $, we find that 
	\begin{equation*}
		\begin{aligned}
			\left\| \Lambda^{-\alpha-1/2}\nabla\cdot\left(u_\theta\theta-(u_\theta\theta)^{(n)} \right)\right\|_{L^2\left([0,T] ; H\right)}&\lesssim\left\| u_\theta\theta-(u_\theta\theta)^{(n)} \right\|_{L^2\left([0,T] ; H^{1/2-\alpha}(\mathbb{T}^2;\mathbb{R}^2)\right)}\\& \lesssim\left\| u_\theta\theta-(u_\theta\theta)^{(n)} \right\|_{L^2\left([0,T] ; L^2(\mathbb{T}^2;\mathbb{R}^2)\right)} \rightarrow 0.
		\end{aligned}
	\end{equation*}
	Hence
	\begin{equation*}R^{1,n} := \Lambda^{-\alpha-1/2}\nabla\cdot\left(u_\theta\theta-(u_\theta\theta)^{(n)} \right) \rightarrow 0 \quad\text{in} \quad L^2\left([0,T] ; H\right).
	\end{equation*}
	Similarly, we see that $ u_{\theta^{(n)}}\theta^{(n)}-u_\theta\theta=(u_{\theta^{(n)}}-u_\theta)\theta^{(n)}+u_\theta(\theta^{(n)}-\theta)\rightarrow 0 $ in $ L^2\left([0,T] ; L^{2}(\mathbb{T}^2;\mathbb{R}^2)\right) $. Thus $$ R^{2,n} :=\Lambda^{-\alpha-1/2}\nabla\cdot\left( u_{\theta^{(n)}}\theta^{(n)}-u_\theta\theta \right) \rightarrow 0 \quad\text{in} \quad L^2\left([0,T] ; H\right) .$$
	Therefore,   $\left\|R^{(n)}\right\|_{L^2([0, T];H)} =\left\|R^{1,n}+R^{2,n}\right\|_{L^2([0, T];H)} \rightarrow 0$ as $n\rightarrow \infty$.  Since $\theta^{(n)}$ is a weak solution of \eqref{PDE-SQG} with control $g^{(n)}+R^{(n)}$, we conclude that
	\begin{equation*}
		\mathcal{I}_{dyna}\left(\theta^{(n)}\right) \leq\left\|g^{(n)}+R^{(n)}\right\|_{L^2\left([0, T];H\right) }^2 \rightarrow\|g\|_{L^2\left([0, T];H\right) }^2=\mathcal{I}_{dyna}(\theta).
	\end{equation*}
\end{proof}
\begin{remark}
	The fact that $ \frac{1}{4}+\frac{\alpha}{4} \geqslant\alpha- \frac{1}{2}$ implies that the space $ \mathcal{R}^1 $ is larger than $ \mathcal{C}_0^1 $ when taking $ r=4 $ in the definition of  $ \mathcal{C}_0^1 $.
\end{remark}

\begin{Example}
	Suppose that  $0<\alpha \leqslant 1/2$ and $ \beta=\alpha/2$.
	Let  $$\mathcal{C}_0^2:= \{\theta \in \mathbb{X}_{\alpha,\alpha/2}: \nabla \theta\in L^r\left([0,T] ; L^{p}(\mathbb{T}^2;\mathbb{R}^2)\right), \, \frac{1}{p}+\frac{\alpha}{r}=\alpha, \, p>	\frac{1}{\alpha}, \, r>0\}.$$
	For every $g \in L^2([0, T];H)$, assume that $\theta_1 \in \mathcal{C}_0^2$ and  $\theta_2\in \mathbb{X}_{\alpha,\alpha/2}$ solve the skeleton equation \eqref{PDE-SQG} in the sense of Definition \ref{def-weak-solution-skeleton-equation} with the same initial data $\theta_1 (0)=\theta_2(0) \in H$, and suppose further that $\theta_2$  satisfy the kinetic energy inequality: 	for every $t \in[0, T]$,
	\begin{equation}\label{energy-inequality-a3}
		\frac{1}{2}\|\theta_2(t)\|_H^2+\int_0^t\|\theta_2(s)\|_{H^{\alpha}(\mathbb{T}^2)}^2 \,\mathrm{d}s \leqslant \frac{1}{2}\left\|\theta_2(0)\right\|_H^2+\int_0^t\langle\Lambda^{\alpha}\theta_2, g\rangle \,\mathrm{d}s.
	\end{equation}
	Then $\theta_1=\theta_2$ in $ \mathbb{X}_{\alpha,\alpha/2}$.
	
\end{Example}

\begin{proof}  This is a direct consequence of	\cite[Theorem 1.1]{DC06_weak_strong}. 
\end{proof}

In the next example, we will take $ \mathcal{R}^2:=  \mathbb{X}_{\alpha,\alpha/2} \cap L^4\left([0,T] ; H^{1-\alpha,4}(\mathbb{T}^2)\right)$ and show that  $ \mathcal{R}^2$ is contained in the $ \mathcal{I}$-closure of $ \mathcal{C}_0^2. $ Here $H^{1-\alpha,4}(\mathbb{T}^2) $ is the homogeneous Sobolev space defined by
\begin{equation*}
	H^{s,p}(\mathbb{T}^2):=\big\{f\in L^{p}(\mathbb{T}^2): \text { there exists some }  g \in  L^{p}(\mathbb{T}^2) , f=\Lambda^{-s} g\big\} \quad s\geq0,\quad p\geq 1,
\end{equation*} with the norm $ \|f\|_{H^{s,p}(\mathbb{T}^2)}:=\|\Lambda^{s}f\|_{L^{p}(\mathbb{T}^2)}$.
We need the following product-type estimate to deal with the nonlinear term.
\begin{lemma}\label{lemma-product-estimate}\cite[Lemma A.4]{Res95}
	Suppose that $s>0$ and $p \in(1, \infty)$. For any $f, g \in C^{\infty}\left(\mathbb{T}^2\right)$, 
	\begin{equation*}
		\left\|\Lambda^s(f g)\right\|_{L^p(\mathbb{T}^2)} \leq C\left(\|f\|_{L^{p_1}(\mathbb{T}^2)}\left\|\Lambda^s g\right\|_{L^{p_2}(\mathbb{T}^2)}+\|g\|_{L^{p_1}(\mathbb{T}^2)}\left\|\Lambda^s f\right\|_{L^{p_2}(\mathbb{T}^2)}\right),
	\end{equation*}
	where $p_1, p_2\in(1, \infty)$ satisfy		$	\frac{1}{p}=\frac{1}{p_1}+\frac{1}{p_2}.$
\end{lemma}
\begin{Example}
	Let $\theta \in \mathcal{R}^2$ with $\mathcal{I}(\theta)<\infty$. Then there exists a sequence $\{\theta^{(n)} \}_{n\geq1}\subset \mathcal{C}_0^2$ such that $\theta^{(n)} \rightarrow \theta$ in  $\mathbb{X}_{\alpha,\alpha/2}$ and $\mathcal{I}\left(\theta^{(n)}\right) \rightarrow \mathcal{I}(\theta)$ as $n\rightarrow\infty$. 
\end{Example}

\begin{proof}
	As in the proof of Example \ref{eg-C-subcritical}, it is sufficient to  show  $\left\|R^{(n)}\right\|_{L^2([0, T]; H) } \rightarrow 0$ as $n\rightarrow \infty$,	where \begin{equation*}
		R^{(n)}:=\Lambda^{-\alpha}(u_{\theta^{(n)}} \cdot \nabla \theta^{(n)}-(u_\theta \cdot \nabla \theta)^{(n)})=\Lambda^{-\alpha}\nabla\cdot\left(u_{\theta^{(n)}}\theta^{(n)}-(u_\theta \theta)^{(n)} \right).
	\end{equation*}	
	Since both Riesz transform and $ \Lambda^{1-\alpha}$  are defined by Fourier multipliers, it follows from Calder\'on-Zygmund inequality that,  for all $ t\in [0,T], \|u_\theta(t)\|_{H^{1-\alpha,4}(\mathbb{T}^2;\mathbb{R}^2)}\leq C\|\theta(t)\|_{H^{1-\alpha,4}(\mathbb{T}^2)}.$ 
	Using  Lemma \ref{lemma-product-estimate} with $ p_1=p_2=4$, Cauchy-Schwarz inequality, and Calder\'on-Zygmund inequality,  
	we have
	\begin{equation*}\|u_\theta\theta\|_ {L^2\left([0,T] ; H^{1-\alpha}(\mathbb{T}^2;\mathbb{R}^2)\right) }^2\lesssim \int_0^T\|\theta(t)\|_{H^{1-\alpha,4}(\mathbb{T}^2)}^4\,\mathrm{d}t.
	\end{equation*}
	As $ (u_\theta\theta)^{(n)}\rightarrow u_\theta\theta $ in $L^2\left([0,T] ; H^{1-\alpha}(\mathbb{T}^2;\mathbb{R}^2)\right) $, we find that $$ R^{1,n} := \Lambda^{-\alpha}\nabla\cdot\left(u_\theta\theta-(u_\theta\theta)^{(n)} \right) \rightarrow 0 \quad\text{in} \quad L^2\left([0,T] ; H\right) .$$
	Applying Lemma \ref{lemma-product-estimate}  again, we see that $u_{\theta^{(n)}} \theta^{(n)}-u_\theta\theta=(u_{\theta^{(n)}} -u_\theta)\theta^{(n)}+u_\theta(\theta^{(n)}-\theta)\rightarrow 0 $ in $ L^2\left([0,T] ; H^{1-\alpha}(\mathbb{T}^2;\mathbb{R}^2)\right) $. Thus $ R^{2,n} :=\Lambda^{-\alpha}\nabla\cdot\left(u_{\theta^{(n)}} \theta^{(n)}-u_\theta\theta \right) \rightarrow 0 $ in $ L^2\left([0,T] ; H\right) .$
	Then it follows from the proof of Example \ref{eg-C-subcritical} that $	\mathcal{I}_{dyna}\left(\theta^{(n)}\right)\rightarrow \mathcal{I}_{dyna}(\theta)$.
\end{proof}
\begin{remark}
	The embedding of    $ H^{1,\frac{1}{1/4+\alpha/2}}(\mathbb{T}^2)$ into $ H^{1-\alpha,4}(\mathbb{T}^2)$ and the fact that $ \frac{1}{4}+\frac{\alpha}{2}+\frac{\alpha}{4} \geqslant\alpha$ imply that the space $\mathcal{R}^2$ is larger than $ \mathcal{C}_0^2 $ when taking $ r=4 $ in the definition of  $ \mathcal{C}_0^2$.
\end{remark}

To illustrate the next example in the critical case $ \alpha=1/2$, we recall that the BMO space on the $ d $-dimensional torus is defined as
\begin{equation*}
	BMO(\mathbb{T}^d):=\{f\in L^1(\mathbb{T}^d):\sup _{\{B \subset \mathbb{T}^d: B\text { is a ball}\}}  \frac{1}{|B|} \int_B\left|f-m_Bf\right| \,\mathrm{d} x<\infty\},
\end{equation*}
where $m_B f=\int_B f(x) \,\mathrm{d}x$. $BMO(\mathbb{T}^d)$ endowed with the norm 	\begin{equation*}
	\|f\|_{BMO(\mathbb{T}^d)}=\sup _{\{B \subset \mathbb{T}^d: B\text { is a ball}\}} \frac{1}{|B|} \int_B\left|f-m_Bf\right| \,\mathrm{d} x+\left|\int_{\mathbb{T}^d}f\,\mathrm{d} x\right| 
\end{equation*}
is a Banach space. We recall that a function $ a $ on $ \mathbb{T}^d $ is called an atom if $ a $ is supported on a ball $ B$, $ \int_Ba(x)\,\mathrm{d} x=0$ and $ \|a\|_{L^\infty(\mathbb{T}^d)}\leq \frac{1}{|B|}. $
The Hardy space $\mathcal{H}^1\left(\mathbb{T}^d\right)$ is defined as	\begin{equation*}
		\mathcal{H}^1(\mathbb{T}^d):=\big\{f\in L^1(\mathbb{T}^d): f \text { can be written as } f=\sum_{j \in \mathbb{N}} \lambda_j a_j: \lambda_j\in\mathbb{R},\, a_j \text{ are atoms }, \sum_{j \in \mathbb{N}} |\lambda_j|<\infty\big\}.
\end{equation*}
The norm of $\mathcal{H}^1\left(\mathbb{T}^d\right)$ is given by $\|f\|_{\mathcal{H}^1(\mathbb{T}^d)}=\inf \left\{\sum_{j \in \mathbb{N}}\left|\lambda_j\right|: f=\sum_{j \in \mathbb{N}} \lambda_j a_j\right\}$, where the infinimum is taken over all the atomic decompositions of $f$. 
\begin{lemma}[\cite{CW77}]\label{lemma-dual-H1-BMO}
	$\mathcal{H}^1(\mathbb{T}^d)$ is a Banach space and the dual space of  $\mathcal{H}^1(\mathbb{T}^d)$ is $BMO(\mathbb{T}^d)$.
\end{lemma}
\begin{lemma}\label{lemma-H-1} For any $ f\in L^2{(\mathbb{T}^2)},$ there exists a constant $ C>0 $ such that
	\begin{equation*} 
		\left\|f R_i f\right\|_{\mathcal{H}^1(\mathbb{T}^2)} \leq C\|f\|_{L^2(\mathbb{T}^2)}^2, \quad i=1,2,
	\end{equation*}
	where $ R_i f $ is the i-th Riesz transform of $ f. $
\end{lemma}
For more details of Lemma \ref{lemma-H-1}, we refer to \cite[Section 3]{CC04} and the references therein.
\begin{Example}
	Suppose that $\alpha=1/2$ and $ \beta=\alpha/2+1/4=1/2$. There exists a positive number $ D_\infty $ such that  $\mathcal{C}_0^3:= \{\theta \in \mathbb{X}_{1/2,1/2}:\theta\in L^2\left([0,T] ; L^p(\mathbb{T}^2)\right), \|R^\perp\theta\|_{L^\infty([0,T] ; BMO(\mathbb{T}^2))}\leq D_\infty, p>2\}$ is a weak-strong uniqueness class. More precisely,
	for every $g \in L^2([0, T];H)$, assume that $\theta_1 \in \mathcal{C}_0^3$ and  $\theta_2\in \mathbb{X}_{1/2,1/2}$ solve the skeleton equation \eqref{PDE-SQG} in the sense of Definition \ref{def-weak-solution-skeleton-equation} with the same initial data $\theta_1 (0)=\theta_2(0) \in H^{-1/2}(\mathbb{T}^2)$, and suppose further that $\theta_2$  satisfy the $H^{-1/2}$-energy inequality: 	for every $t \in[0, T]$,
	\begin{equation}\label{energy inequality for theta2-3}
		\frac{1}{2}\|\theta_2(t)\|_{H^{-1/2}(\mathbb{T}^2)}^2+\int_0^t\|\theta_2(s)\|_{H}^2 \,\mathrm{d}s \leqslant \frac{1}{2}\left\|\theta_2(0)\right\|_{H^{-1/2}(\mathbb{T}^2)}^2+\int_0^t\langle\theta_2,g\rangle \,\mathrm{d}s.
	\end{equation}
	Then $\theta_1=\theta_2$ in $\mathbb{X}_{1/2,1/2}$.
	
\end{Example}
\begin{proof}
	The same conclusion holds in $ \mathbb{R}^2 $ case, which was proved in \cite[Theorem 1.3]{Mar08_Weak_strong_critical}. 
	For every $t \in[0, T]$, we expand  $\|\theta_1(t)-\theta_2(t)\|_{{H^{-1/2}(\mathbb{T}^2)}}^2$ as \begin{equation}\label{expansion-3}
		\|\theta_1(t)-\theta_2(t)\|_{{H^{-1/2}(\mathbb{T}^2)}}^2=-2\langle\Lambda^{-1/2}\theta_1(t),\Lambda^{-1/2}\theta_2(t)\rangle +\|\theta_1(t)\|_{{H^{-1/2}(\mathbb{T}^2)}}^2+\|\theta_2(t)\|_{{H^{-1/2}(\mathbb{T}^2)}}^2
	\end{equation} and estimate these three terms respectively. For the first term of the right-hand side of \eqref{expansion-3}, we need the following equality:
	\begin{equation}\label{auxiliary-3}
		\begin{aligned}
			&\quad\int_0^t\left(2 \langle\theta_1, \theta_2\rangle+ \langle \theta_2u_{\theta_2}, \nabla\Lambda^{-1}\theta_1 \rangle+ \langle \theta_1u_{\theta_1}, \nabla\Lambda^{-1}\theta_2 \rangle\right) \,\mathrm{d} s 
			\\&=-\langle\Lambda^{-1/2} \theta_2(t), \Lambda^{-1/2}\theta_1(t)\rangle +\langle \Lambda^{-1/2}\theta_2(0), \Lambda^{-1/2}\theta_1(0)\rangle +\int_0^t\langle\theta_2+\theta_1, g\rangle \,\mathrm{d}s.
		\end{aligned}
	\end{equation}
	This can be proved by taking $\Lambda^{-1}\theta^{(n)}_i (i=1,2)$ as test functions and passing to the limits.	We focus on the convergence of the nonlinear terms. 
	Since $ \theta_1\in\mathcal{C}_0^3$, $\nabla\Lambda^{-1}  \theta_1 =R\theta_1\in L^\infty([0,T];BMO(\mathbb{T}^2))$. Moreover, Lemma \ref{lemma-H-1} implies that $\theta_2u_{\theta_2}\in L^1([0,T];\mathcal{H}^1(\mathbb{T}^2))$.  By the the equivalent characterization of the $ \mathcal{H}^1$-norm (\cite[Theorem 6.4]{Duo01}) and the property of the  convolution kernel $ \eta_{1/n} $,
	\begin{equation*}
		\left| \int_0^t\langle \theta_2u_{\theta_2}, \nabla\Lambda^{-1}(\theta_1-\theta^{(n)}_1)\rangle\,\mathrm{d}s\right|= \left| \int_0^t\langle \theta_2u_{\theta_2}-(\theta_2u_{\theta_2})^{(n)}, \nabla\Lambda^{-1}\theta_1\rangle\,\mathrm{d}s\right| \rightarrow 0.
	\end{equation*} The continuous embedding $BMO(\mathbb{T}^2)\subset L^{\frac{2p}{p-2}}(\mathbb{T}^2)$  (\cite[Corollary 6.12]{Duo01}) implies that $ \theta_1u_{\theta_1}\in  L^2([0,T];L^2(\mathbb{T}^2)).$
	Since  $\theta^{(n)}_2\rightarrow\theta_2$	in $ L^2([0,T];L^2(\mathbb{T}^2)), $ it follows that \begin{equation*}
		\bigg|\int_{0}^{t} \langle \theta_1u_{\theta_1}, \nabla\Lambda^{-1}(\theta_2-\theta^{(n)}_2)\rangle\,\mathrm{d}s\bigg|\rightarrow 0.
	\end{equation*} 
	This completes the proof of \eqref{auxiliary-3}.
	Since \eqref{auxiliary-3} implies that $ \theta_1 $ is sufficiently regular to be taken as a test function, for any $ t\in[0,T],$
	\begin{equation}\label{energy equality for theta1-3}
		\frac{1}{2}\|\theta_1(t)\|_{H^{-1/2}(\mathbb{T}^2)}^2+ \int_0^t\left\|\theta_1(s)\right\|_{H}^2\,\mathrm{d}s =	\frac{1}{2}\|\theta_1(0)\|_{H^{-1/2}(\mathbb{T}^2)}^2+ \int_0^t\langle\theta_1, g\rangle  \,\mathrm{d}s .
	\end{equation}
	According to \eqref{expansion-3}, \eqref{auxiliary-3},  \eqref{energy equality for theta1-3}, and \eqref{energy inequality for theta2-3}, using Lemma \ref{lemma-dual-H1-BMO} and Lemma \ref{lemma-H-1}, we obtain
	\begin{equation*}
		\begin{aligned}
			&\|\theta_1(t)-\theta_2(t)\|_{{H^{-1/2}(\mathbb{T}^2)}}^2+ 2\int_0^t\left\|\theta_1-\theta_2\right\|_{H}^2 \,\mathrm{d} s \\\leqslant&2 \int_0^t \langle \theta_2u_{\theta_2}, \nabla\Lambda^{-1}\theta_1 \rangle+ \langle \theta_1u_{\theta_1}, \nabla\Lambda^{-1}\theta_2 \rangle \,\mathrm{d}s\\=&-2\int_0^t\langle(\theta_2-\theta_1)u_{\theta_1}, \nabla\Lambda^{-1}(\theta_2-\theta_1) \rangle \,\mathrm{d}s\\\leqslant&2\|(\theta_2-\theta_1)R(\theta_2-\theta_1)\|_{L^1([0,T];\mathcal{H}^1(\mathbb{T}^2))}\| u_{\theta_1}\|_{L^\infty([0,T];BMO(\mathbb{T}^2))}  \\\leqslant& C\|\theta_2-\theta_1\|^2_{L^2([0,T];H)}\| R^\perp \theta_1\|_{L^\infty([0,T];BMO(\mathbb{T}^2))} \\\leqslant& CD_\infty\|\theta_2-\theta_1\|^2_{L^2([0,T];H)}.
		\end{aligned}
	\end{equation*}
	Hence we can choose a sufficiently small $ D_\infty $ to deduce that $ \theta_1-\theta_2 = 0. $
\end{proof}
	
	\section*{Acknowledgements}
Lin Wang acknowledges the support by National Key R\&D Program of China (No.2020YFA0712700), National Natural Science Foundation of China (No. 12090010, 12090014, 12471138), and Key Laboratory of Random Complex Structures and Data Science, Academy of Mathematics and Systems Science, Chinese Academy of Sciences (No. 2008DP173182). 
Zhengyan Wu acknowledges the support by the Deutsche Forschungsgemeinschaft (DFG, German Research Foundation) via IRTG 2235 - Project Number 282638148.
The authors are grateful to the anonymous reviewers for their valuable suggestions and insightful feedback that have improved the clarity and quality of this work. The authors also would like to thank  Prof. Zhao Dong for giving the consistent support and encouragement. 
		\bibliographystyle{alpha}
		\bibliography{LDPSQG}
	\end{document}